\documentclass[a4paper,oneside,11pt]{article} 

\usepackage[utf8]{inputenc} 
\usepackage[T1]{fontenc}
\usepackage{textcomp} 
\usepackage[english]{babel}
\usepackage[autolanguage]{numprint} 
\usepackage{hyperref} 
\usepackage{graphicx} 
\usepackage{verbatim} 
\usepackage[all]{xy}
\usepackage{dsfont}
\usepackage{amsmath,amssymb,amsthm} 
\usepackage{lmodern}
\usepackage{ tipa }
\usepackage{latexsym}
\usepackage{mathtools}
\usepackage{relsize}
\usepackage{exscale}

\renewcommand\epsilon\varespilon 
\renewcommand\phi\varphi

\newcommand\NN{\mathbb{N}} 
\newcommand\ZZ{\mathbb{Z}} 
 
\newcommand\RR{\mathbb{R}} 
\newcommand\CC{\mathbb{C}} 
\newcommand\PP{\mathbb{P}}
\newcommand\EE{\mathbb{E}}
\newcommand\tr{\mathrm{tr}}
\newcommand\Car{\mathds{1}}
\newcommand\bornesup{\limsup_{N\to +\infty} \frac{1}{N^{\alpha/2}} \log}
\newcommand\borneinf{\liminf_{N\to +\infty} \frac{1}{N^{\alpha/2}} \log}
\newcommand\bornexpo{\lim_{N\to +\infty} \frac{1}{N^{\alpha/2}} \log}
\newcommand\ee{\varepsilon}

\newcommand\rk{\mathrm{rank}}

\newcommand\M{\mathcal{M}}
\newcommand\supp{\mathrm{supp}}

\swapnumbers 
\theoremstyle{definition} 
\newtheorem{Def}{Definition}[section] 
\theoremstyle{plain} 
\newtheorem{Pro}[Def]{Proposition} 
\newtheorem{Lem}[Def]{Lemma} 
\newtheorem{The}[Def]{Theorem} 
 
\newtheorem{Cor}[Def]{Corollary} 
\newtheorem{Hypo}[Def]{Assumption} 
\theoremstyle{remark}

\newtheorem{Rem}[Def]{Remark}

\title{Large deviations principle for the largest eigenvalue of Wigner matrices without Gaussian tails}  
\author{Fanny Augeri } 
\date{\today} 

\begin{document}

\maketitle

\section*{Abstract}
We prove a large deviations principle for the largest eigenvalue of Wigner matrices without Gaussian tails, that is, the distribution tails of the diagonal entries $\PP( |X_{1,1}|>t)$ and off-diagonal entries $\PP(|X_{1,2}|>t)$ behave like $e^{-bt^{\alpha}}$ and $e^{-at^{\alpha}}$ respectively, for some $a,b\in (0,+\infty)$ and $\alpha\in (0,2)$. The large deviations principle is of speed $N^{\alpha/2}$, and with a good rate function depending only on the distribution tail of the entries.

\section{Introduction and main result}

The study of large deviations in the context of random Hermitian matrices dates back to 1997, with the work of Ben Arous and Guionnet. In \cite{BenArous}, they proved a large deviations principle for the empirical measure of $\beta$-ensembles associated with a quadratic potential, with speed $N^2$ and an explicit rate function. This result answers the question of the large deviations of the empirical spectral measure of the classical random matrix ensembles, GOE, GUE, and GSE,  since their eigenvalues form a $\beta$-ensemble associated with a quadratic potential for $\beta=1,2$ and $4$ respectively. In \cite[p.81]{Guionnet}, this result has been extended by the same authors, for $\beta$-ensembles associated with a potential $V$ growing at infinity faster than $\log |x|$, which include unitary invariant or orthogonally invariant models of random matrices. 
 Recently, it has been shown in \cite{Hardy} that the restriction on the growth of the potential could been lifted, so that one can also consider potentials with logarithmic growth. The large deviations results of the empirical spectral measure of the classical random matrix ensembles rely  heavily on the knowledge of the distribution of the eigenvalues, and its interpretation as a $\beta$-ensemble.

In the setting of the so-called Wigner deformed ensemble, the large deviations of the empirical spectral measure were studied, first in \cite{Cabanal} and then in \cite{Zeitounisphint}, in which a large deviations principle was established for the empirical spectral measure of the sum of Gaussian Wigner matrix and a deterministic Hermitian matrix. For this model, as one cannot compute the joint law of the eigenvalues, the proof relies on the Gaussian nature of the entries and uses Dyson Brownian motion and stochastic calculus. 

Regarding the large deviations of the extreme eigenvalues of Wigner matrices, the first result was proved in \cite{Dembo} in the case of the GOE and then extended in \cite[p.83]{Guionnet} for $\beta$-ensemble, under an extra assumption on the partition function of the Gibbs measure. The large deviations principle is of speed $N$, and with an explicit rate function. The large deviations of the extreme eigenvalues of deformed Wigner ensembles have also been studied. In \cite{Mylene}, the author investigates the case of a GOE (respectively  GUE) matrix perturbed by a rank one deterministic symmetric (respectively Hermitian) matrix. Then in \cite{BGM}, the large deviations for the joint law of the extreme eigenvalues of a deterministic real diagonal  matrix perturbed with a low rank Hermitian matrix with delocalized eigenvectors are studied extensively.

Yet, all those large deviations results rely either on the computation of the joint law of the eigenvalues or on the Gaussian nature of the entries. In \cite{Bordenave} Bordenave and Caputo gave a large deviations principle for the empirical spectral measure of Wigner matrices with coefficients without Gaussian tail, a case where there is no explicit computation of the joint law of the eigenvalues. Recently, this result has been extended in the case of Wishart matrices in  \cite{Groux}. 

Still, in the setting of Wigner's matrices which coefficients have a sub-Gaussian tail but are not Gaussian, the existence of a large deviation principle for the empirical distribution of eigenvalues or the largest eigenvalue is still an open problem.

\subsection{Main result}

The aim of this paper is to derive a large deviations principle for the largest eigenvalue of Wigner matrices under the same statistical assumptions as in \cite{Bordenave}, together with an additional technical assumption.

Let $(X_{i,j})_{i<j}$ be independent and identically distributed (i.i.d) complex-valued random variables, such that $\EE(X_{1,2}) = 0$, $\EE|X_{1,2}|^2 = 1$, and let $(X_{i,i})_{i\geq 1}$ be i.i.d real-valued random variables.

Let $X(N)$ be the $N \times N$ Hermitian matrix with up-diagonal entries $(X_{i,j})_{1\leq i \leq j\leq N}$. We call such a sequence $(X(N))_{N\in \NN}$, a Wigner matrix. In the following, we will drop the $N$ and write $X$ instead of $X(N)$. 

Consider now the normalized random matrix $X_N = X/\sqrt{N}$. Let $\lambda_i$ denote the eigenvalues of $X_N$, with $\lambda_1 \leq \lambda_2 \leq ...\leq \lambda_N$. We define $\mu_{X_N}$ the empirical spectral measure of $X_N$ by,
$$\mu_{X_N} = \frac{1}{N}\sum_{i=1}^N \delta_{\lambda_i}.$$ 
We know by Wigner's theorem (see \cite{Wigner}, \cite[Theorem 2.1.21, 2.21]{Guionnet}, \cite{Silverstein}[Theorem 2.5]), that
$$\mu_{X_N}\underset{N\to +\infty}{ \rightsquigarrow} \sigma_{sc} \text{ a.s, }$$
where $\rightsquigarrow$ denotes the weak convergence and where $\sigma_{sc}$ denotes the semicircular law which is defined by,
$$\sigma_{sc}(dt) =\Car_{t\in [-2, 2]} \frac{1}{2\pi} \sqrt{4-t^2} dx.$$
Furthermore, assuming that $\EE|X_{1,1}|^2 < +\infty$ and $\EE|X_{1,2}|^4 < +\infty$, we know from \cite{Furedi}, \cite{Bai}, and \cite{Silverstein}[Theorem 5.1], that 
$$\lambda_N \underset{N\to +\infty}{\longrightarrow} 2 \text{ a.s. }$$

We recall that a sequence of random variables $(Z_n)_{n\in \NN}$ taking value in some topological space $\mathcal{X}$ equipped with the Borel $\sigma$-field $\mathcal{B}$, follows a large deviations principle (LDP) with speed $\upsilon : \NN\to \NN$, and rate function $J : \mathcal{X} \to [0, +\infty]$, if $J$ is lower semicontinuous and $\upsilon$ increases to infinity and for all $B\in \mathcal{B}$,
$$- \inf_{B^{\circ}}J \leq \liminf_{n\to +\infty} \frac{1}{\upsilon(n)} \log \PP\left(Z_n \in B\right) \leq \limsup_{n\to +\infty} \frac{1}{\upsilon(n)}\log\PP\left(Z_n \in B\right) \leq -\inf_{\overline{B}  } J,$$
where $B^{\circ}$ denotes the interior of $B$ and $\overline{B}$ the closure of $B$. We recall that $J$ is lower semicontinuous if its $t$-level sets $\{ x \in \mathcal{X} : J(x) \leq t \}$ are closed, for any $t\in [0,+\infty)$. Furthermore, if all the level sets are compact, then we say that $J$ is a good rate function.

In the following, we make the following assumptions.
\begin{Hypo}\label{hypo1}
Let $X$ be a Wigner matrix. In the case where $X_{1,2}$ is a complex random variable, $\Re(X_{1,2})$ and $\Im(X_{1,2})$ are independent. There exist  $\alpha \in (0,2)$ and $a, b \in (0, +\infty)$ such that,
\begin{equation} \label{queue de distrib}\lim_{t\to +\infty} -t^{-\alpha}\log \PP\left(|X_{1,1}|>t\right) = b,\end{equation}
$$ \lim_{t\to +\infty} -t^{-\alpha} \log \PP\left(|X_{1,2}|>t\right) = a.$$
Moreover, we assume that there are two probability measures on $\mathbb{S}^1$, $\upsilon_1$ and $\upsilon_2$,  and $t_0>0$, such that for all $t\geq t_0$ and any measurable subset $U$ of $\mathbb{S}^1$,
\begin{equation*} \label{decouplage module argument}\PP\left( X_{1,1}/|X_{1,1}| \in U, |X_{1,1}|\geq t\right) = \upsilon_1(U) \PP\left(|X_{1,1}|\geq t\right),\end{equation*}
$$\PP\left( X_{1,2}/|X_{1,2}| \in U, |X_{1,2}|\geq t\right) = \upsilon_2(U) \PP\left(|X_{1,2}|\geq t\right).$$
In other words, this means that for all indices $i,j$, the absolute value and the angle of $X_{i,j}$ are independent for large values of $|X_{i,j}|$.

\end{Hypo}

\begin{Rem}
The assumption on the independence of the real and imaginary parts of the off-diagonal entries is purely technical. We only make this assumption in order to use the estimates in \cite{resolvantentries} on the entries of the resolvent,  in the proof of an isotropic property of the semi-circular law in Theorem \ref{conv unif resolvante}. Moreover, this assumption is not needed in \cite{Bordenave}.

\end{Rem}
Under these assumptions, it has been proven in \cite{Bordenave} that the empirical spectral measure of the normalized matrix $X_N$, denoted by $\mu_{X_N}$, follows a large deviations principle with respect to the weak topology. The LDP is with speed $N^{1+\alpha/2}$, and good rate function $I$ defined for all $\mu \in \mathcal{M}_1(\RR)$,  where $\mathcal{M}_1(\RR)$ denotes the space of all probability measures on $\RR$, by

$$ I(\mu) = 
\begin{cases}
\Phi(\nu)& \text{if } \mu = \sigma_{sc} \boxplus \nu \text{ for some } \nu \in \mathcal{M}_1(\RR),\\
+\infty & \text{otherwise},
\end{cases}$$
where $\boxplus$ denotes the free convolution, and where $\Phi$ denotes a good rate function (see \cite{Bordenave} for further details).

In the following, for any Hermitian matrix $Y$, we will denote by  $\lambda_Y$ its largest eigenvalue.
We will prove in this paper the following large deviations result.
\begin{The}\label{mainresult}
Under assumptions \eqref{hypo1}, the sequence $(\lambda_{X_N})_{N\in \NN}$ follows a large deviations principle with speed $N^{\alpha/2}$, and good rate function defined for all $x\in \RR$, by
$$J(x) = 
\begin{cases}
cG_{\sigma_{sc}}(x)^{-\alpha} & \text{if } x>2,\\
0 & \text{if } x =2,\\
+\infty & \text{if } x< 2,\\
\end{cases}
$$
where $c$ is a constant depending only on $\alpha, a$ and $b$, and where $G_{\sigma_{sc}}$ denotes the Stieltjes transform of the semicircular law, namely
$$ \forall z \in \CC\setminus (-2,2), \ G_{\sigma_{sc}}(z) = \int \frac{d\sigma_{sc}(t)}{z-t},$$
with $$ \sigma_{sc}(dt) = \Car_{t\in [-2, 2]} \frac{1}{2\pi} \sqrt{4-t^2} dt.$$
\end{The}
Moreover, we will prove that the constant $c$ in Theorem \ref{mainresult}, can be computed explicitly in certain cases, in particular when the entries are real random variables. We refer the reader to the Section \ref{calculc} for further details.

Observe that the rate function is infinite on $(-\infty, 2)$. Indeed, in order to make a deviation of the top eigenvalue at the left of $2$, we need to force the support of the empirical spectral measure to be in $(-\infty, 2-\ee)$, for some $\ee>0$. But this event has an infinite cost at the exponential scale $N^{\alpha/2}$ since the empirical spectral measure follows a large deviation principle with speed $N^{1+\alpha/2}$ according to \cite{Bordenave}.
As illustrated in figure \ref{rate function}, drawn in the case $\alpha=1$, this rate function is also discontinuous at $2$. As we will show, the deviations of the top eigenvalue are given by finite rank perturbations of a Wigner matrix. It is well-known that finite rank perturbations of Wigner matrices show a threshold phenomenon with respect to the strength of the perturbation  (see for example \cite{Peche}, \cite{Feral} \cite{Soshnikov}, \cite{Benaych}, \cite{Maida} for further details), which the rate function seems to reflect through the discontinuity at $2$. This picture may also mean that there is a more subtle behavior of the largest eigenvalue in the right neighborhood of $2$, which is still to be understood.  
\begin{figure}[h]
\centering
\caption{\label{rate function}Graph of the rate function $J$} 
\includegraphics[width = 11cm]{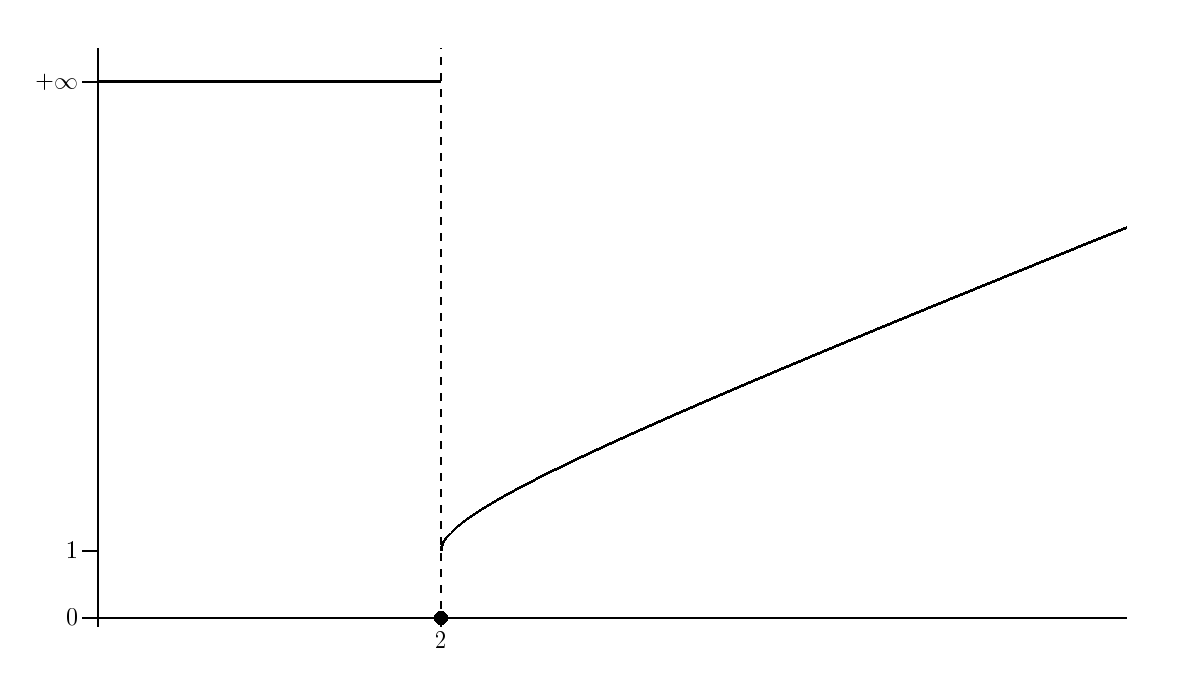}
\end{figure}

\section*{Acknowledgement}
I would like to thank Alice Guionnet for welcoming me at MIT during April and May 2014, where I was able to put into shape this paper. I feel very grateful to have had this opportunity to work with Alice Guionnet and for the time and availability she offered me. I also would like to thank MIT for its hospitality and all the people who made my time over there so enjoyable. Finally I would like to thank my supervisor Charles Bordenave for his inspiring advice and the attention he gave to this paper.

\section{Heuristics}\label{heuristics}
We will show that one can obtain the lower bound of the LDP by finite rank perturbation. For simplicity, let us assume that the $X_{i,j}$'s are exponential variables with parameter $1$. Thus, the matrix $X$ satisfies the assumptions \eqref{hypo1} with $\alpha=1$, and $a=b=1$. In this case, Proposition \ref{calcul expli} shows that the constant $c$ in Theorem \ref{mainresult} is $1$.

Let $x>2$ and $\theta = 1/G_{\sigma_{sc}}(x)$. As $G_{\sigma_{sc}}(x) \in (0,1]$ for all $x\in[2,+\infty)$, we have $\theta>1$. By independence of the entries, we have
\begin{equation} \label{borneinf}\PP\left(\lambda_{X_N}  \simeq x \right) \gtrsim \PP\left(\lambda_{X'_N+\theta e_1 e_1^*} \simeq x\right) \PP\left( \frac{X_{1,1}}{\sqrt{N}}\simeq \theta\right),\end{equation}
with $X'_N = X_N - \frac{X_{1,1}}{\sqrt{N}}e_1 e_1^*$, and $e_1$ the first coordinate vector of $\CC^N$.
Since $\theta>1$, we have according to \cite{Soshnikov}, 
$$\lambda_{X_N+\theta e_1 e_1^*} \underset{N\to +\infty}{\longrightarrow} G_{\sigma_{sc}}^{-1}\left(1/\theta\right) \text{ in probability.}$$
Using Weyl's inequality (see in the Appendix Lemma \ref{Weyl}) and recalling that we chose $x = G_{\sigma_{sc}}^{-1}\left(1/\theta\right)$, we get
\begin{equation}\label{convprob}  \PP\left(\lambda_{X'_N+\theta e_1 e_1^*}\simeq x \right) \underset{N\to +\infty}{\longrightarrow} 1.\end{equation}
But $X_{1,1}$ has exponential law with parameter $1$, thus
\begin{equation} \label{devexpo}\PP\left( \frac{X_{1,1}}{\sqrt{N}}\simeq \theta\right) \simeq e^{-\theta \sqrt{N}}.\end{equation}
Putting together \eqref{borneinf}, \eqref{convprob} and \eqref{devexpo}, we get,
$$\PP\left(\lambda_{X_N}  \simeq x \right) \gtrsim e^{-G_{\sigma_{sc}}(x)^{-1} \sqrt{N}}.$$
which is the lower bound expected by Theorem \ref{mainresult} and Proposition \ref{calcul expli}, for $\alpha =1$ and $a=b=1$.
Note that we could also have used  a deformation of the type 
$$\left(
\begin{array}{cc}
0 & \theta  \\
\theta & 0 
\end{array}
\right),
$$
to get the lower bound of the LDP.
\section{Outline of proof}
The strategy of the proof will closely follow the one of the LDP for the empirical spectral measure derived in \cite{Bordenave}.

Following \cite{Bordenave}, we start by cutting the entries of $X_N$ according to their size.We decompose $X_N$ in the following way. Fix some $d>0$ such that $d\alpha>1$, and let $\ee>0$. We write,

\begin{equation} \label{cut} X_N = A +B^{\ee} +C^{\ee} +D^{\ee},\end{equation}
with, for all $i,j \in \{1,...,N\}$,
$$A_{i,j}=\Car_{|X_{i,j}|_{\infty}\leq(\log N)^d}\frac{X_{i,j}}{\sqrt{N}}, \quad \quad B^{\ee}_{i,j} = \Car_{(\log N)^d<|X_{i,j}|_{\infty} < \ee N^{1/2}} \frac{X_{i,j}}{\sqrt{N}},$$
$$C^{\ee}_{i,j}=\Car_{\ee N^{1/2}\leq|X_{i,j}|_{\infty}\leq \ee^{-1} N^{1/2}}\frac{X_{i,j}}{\sqrt{N}}, \quad \quad D^{\ee}_{i,j} = \Car_{\ee^{-1}N^{1/2}<|X_{i,j}|_{\infty}} \frac{X_{i,j}}{\sqrt{N}},$$
 where $|z|_{\infty} = \max(|\Re(z)|, |\Im(z)|)$ for all complex numbers $z$.

Our first step will be to prove some concentration inequalities in Section \ref{notation}, which we will use throughout this paper, and in particular to prove the exponential tightness of $(\lambda_{X_N})_{N\in \NN}$ in Section \ref{Exponential tightness}.

Then, in Section \ref{Exponential equivalences}, we will focus on trying to identify which parts in the decomposition of $X_N$ significantly contribute to create deviations of the largest eigenvalue with regards to its limiting value $2$.
We start by showing in Section \ref{1Equiv}, that we can neglect the contributions of $B^{\ee}$ and $D^{\ee}$, corresponding to the intermediate and large entries respectively, in the deviations of $\lambda_{X_N}$. Then in Section \ref{2Equiv}, we prove that we can replace $A$ by a Hermitian matrix $H_N$, with entries bounded by $(\log N)^d /\sqrt{N}$, and independent from $C^{\ee}$. 

From the LDP of the empirical spectral measure of $X_N$ of speed $N^{1+\alpha/2}$ proved in \cite{Bordenave}, we deduce in Proposition \ref{pgdinf2} that the deviations at the left of $2$ have an infinite cost at the scale $N^{\alpha/2}$. Therefore, we only need to focus on the deviations of the largest eigenvalue of $H_N +C^{\ee}$ at the right of $2$. As in many papers on finite rank deformations of Wigner matrices (see \cite{Benaych} for exemple), we see the largest eigenvalue of $H_N+C^{\ee}$, provided it is not in the spectrum of $H_N$, as the largest zero of the function,
$$f_N(x) = \det\left(M_N(x)\right), \text{ with } M_N(x) =I_k - (\theta_i\langle u_i, (x-H_N)^{-1} u_j\rangle)_{1\leq i,j \leq k},$$ where $k$ is the rank of $C^{\ee}$, $\theta_1,...,\theta_k$ are the non-zero eigenvalues of $C^{\ee}$ in non-decreasing order, and $u_1,...,u_k$ are orthonormal eigenvectors of $C^{\ee}$ associated to $\theta_1,...,\theta_k$. 

As we will see, this method is made efficient in the study of the deviations of $\lambda_{H_N + C^{\ee}}$ at the right of $2$ by two main facts. Firstly, as we show in Proposition \ref{controlesp}, the spectrum of $H_N$ can be considered at the exponential scale $N^{\alpha/2}$ nearly as contained in $(-\infty,2]$. Secondly, as shown in Lemma \ref{entrees non nulles C}, $C^{ \ee}$ is a sparse matrix so that its rank can be considered at the exponential scale $N^{\alpha/2}$ as bounded.

In Section \ref{equation vpmax}, we focus on showing that the function $f_N$ is exponentially equivalent to a certain limit function $f$, defined for any $x>2$ by,
$$f(x) = \prod_{i=1}^k \left(1 - \theta_iG_{\sigma_{sc}}(x)\right).$$
To this end, we show in Proposition \ref{concentration equation}, using concentration inequalities,  that at the exponential scale $N^{\alpha/2}$, and uniformly in $x$ in a compact subset of $(2,+\infty)$,
\begin{equation}\label{approx conc} M_N(x) \simeq I_k - (\theta_i\langle u_i, \EE\left (x-H_N\right)^{-1} u_j \rangle )_{1\leq i,j \leq k}.\end{equation}
Next,  in Theorem \ref{conv unif resolvante}, we prove an isotropic property of the semi-circular law using the estimates in \cite{Soshnikov} of the entries of the resolvent of Wigner matrices. This  allows us to deduce in Proposition \ref{approxequavp} that
$$M_N(x)\simeq I_k - 
  \left(
     \raisebox{0.5\depth}{%
       \xymatrixcolsep{1ex}%
       \xymatrixrowsep{1ex}%
       \xymatrix{
         \theta_1G_{\sigma_{sc}}(x) \ar @{.}[ddddrrrr]& 0 \ar @{.}[rrr] \ar @{.}[dddrrr] &  & & 0  \ar @{.}[ddd]  \\
         0 \ar@{.}[ddd] \ar@{.}[dddrrr]& & & & \\
         &&&& \\
         &&&& 0 \\
        0 \ar@{.}[rrr] & & & 0 & \theta_k G_{\sigma_{sc}}(x)
       }%
     }
   \right),
$$
where we denote by $G_{\sigma_{sc}}(x)$ the resolvent of the semi-circular law.
Using the fact that the spectral radius of $C^{\ee}$ can be considered as bounded as shown in Lemma \ref{tensionC}, and using the uniform continuity of the determinant on compact sets of $H_k( \CC)$, we get, as stated in Theorem \ref{approx equation au vp}, uniformly in $x$ in any compact subset contained in $(2,+\infty)$,
$$f_N(x) \simeq f(x),  \text{ with } f(x) = \prod_{i=1}^k \left(1 - \theta_iG_{\sigma_{sc}}(x)\right).$$
In Section \ref{equiv expo zero}, we show that provided $\lambda_{H_N+C^{\ee}}$ is greater that $2$, and that $\lambda_{C^{\ee}}$ is greater than $1$, the largest zero of $f_N$, namely $\lambda_{H_N+C^{\ee}}$, is exponentially equivalent to the largest zero of $f$, denoted by $\mu_{N, \ee}$. Easy computations show that
$$ \mu_{N, \ee} = G_ {\sigma_{sc}}^{-1}\left( 1/ \lambda_{C^{\ee}} \right).$$
Despite the fact that $f_N$ and $f$ are holomorphic functions, we cannot use Rouché's theorem to deduce that their zeros are close since we only know that they are close on compact subsets of $(2, +\infty)$. We use here a trick a bit similar to the one used in  \cite[p. 513]{Benaych}, which will allow us to make do with this uniform closeness between $f_N$ and $f$ on compact subsets of $(2, +\infty)$. We perturb the spectrum of $C^{\ee}$ so as to its largest eigenvalue is simple and bounded away from its second largest eigenvalue by some $\gamma>0$. Classical intermediate values theorem then shows that any continuous function $\phi$ close to $f$ on all compact subsets contained in $(2,+\infty)$, admits a zero in $(2,+\infty)$, and that its largest zero is close to the largest zeros of $f$. Since $f$ remains in a compact set of continuous functions, we can prove a uniform continuity property for the "largest zero function" in Lemma \ref{continuite zero}. In Proposition \ref{equivexpo gene}, we deduce that the largest zero of $f_N$ and of $f$ are exponentially equivalent at the scale $N^{\alpha/2}$. This allows us to conclude in Theorem \ref{approxexpobonne} that $(\mu_{N, \ee})_{N\in \NN, \ee>0}$, are an exponentially good approximations of $\lambda_{X_N}$ (in the sense of \cite[4.2.2]{Zeitouni}).

Then, in Section \ref{LDP}, we prove that $(\mu_{N, \ee})_{N\in \NN}$ satisfies a LDP for each $\ee>0$, and we deduce a LDP for $(\lambda_{X_N})_{N\in \NN}$. The key of the proof is Proposition \ref{entrees non nulles C}, which allows us to assume that the matrix $C^{\ee}$ has only a finite number of non-zero entries at the exponential scale $N^{\alpha/2}$. With this observation, the problem can be reduced to a finite-dimensional one. We define $\widetilde{\mathcal{E}}_r$ to be the set of equivalence classes of infinite Hermitian matrices with at most $r$ non-zero entries, under the action of permutation matrices. 
 In Proposition \ref{LDP nb entree fixe}, we establish a LDP for $C^{\ee}$, when seen as an element of $\widetilde{\mathcal{E}}_r$, with respect to the topology given by the distance
$$ \forall \tilde{A}, \tilde{B} \in \widetilde{\mathcal{E}}_r, \ \tilde{d}\left(\tilde{A}, \tilde{B}\right) = \min_{\sigma, \sigma' \in \mathcal{S}}\max_{i,j} \left|B _{\sigma(i), \sigma(j)} - A_{\sigma'(i),\sigma'(j)}\right|,$$
where $A$ and $B$ representatives of $\tilde{A}$ and $\tilde{B}$ respectively, and where $\mathcal{S}=\cup_{n\in \NN} \mathcal{S}_n$ is the union of the symmetric groups. The map which associates to any matrix of $\widetilde{\mathcal{E}}_r$, its largest eigenvalue is continuous with respect to $\tilde{d}$, and allows us to apply a contraction principle to get the large deviations principle for $(\mu_{N, \ee})_{N\in \NN}$, which is stated in Proposition \ref{pdgmuN}. We finally deduce a LDP for $(\lambda_{X_N})_{N\in \NN}$ in Theorem \ref{pgdXn}, with rate function 
$$J(x) = 
\begin{cases}
cG_{\sigma_{sc}}(x)^{-\alpha} & \text{if } x> 2,\\
0 & \text{if } x =2,\\
+\infty & \text{if } x< 2,\\
\end{cases}
$$
where  \begin{equation} \label{constc} c = \inf \left \{ b \sum_{i=1}^{+\infty} |A_{i,i}|^{\alpha}+a \sum_{i\neq j} |A_{i,j}|^{\alpha} : \lambda_A = 1, A \in \mathcal{D}\right\},\end{equation}
and $$\mathcal{D} = \left\{ A \in \cup_{n\geq 1} H_n(\CC) : \forall i\leq j , \ A_{i,j} = 0 \text{ or } \frac{A_{i,j}}{|A_{i,j}|} \in \mathrm{supp}(\nu_{i,j}) \right\},$$
where $\nu_{i,j}= \nu_1$ if $i=j$, and $\nu_2$ if $i< j$, and where $\mathrm{supp}(\nu_{i,j})$ denotes the support of the measure $\nu_{i,j}$.

In Section \ref{calculc}, we show that we can compute explicitly in certain cases the constant $c$ appearing in the rate function $J$.  In particular, in the case where the entries of $X_N$ are real, or when $\alpha \in (0,1]$, Proposition \ref{calcul expli} computes completely the constant $c$.

The optimization problem \eqref{constc} exhibits two different behaviors, when $\alpha \in (0,1]$
 and when $\alpha \in (1,2)$. When $\alpha \in (0,1]$, the infimum is achieved for matrices of sizes $1$ or $2$, and can computed for any choice of $\nu_1$ and $\nu_2$. When $\alpha \in (1,2)$, the picture is more complicated, and one cannot say much without some assumptions on the supports of $\nu_1$ and $\nu_2$. In particular, one can observe that when $b> \frac{a}{2}$ and $1\in \supp(\nu_1)\cap\supp(\nu_2)$, the infimum can be achieved for a matrix of size arbitrary large, when $\alpha$ gets arbitrary close to $2$. 

Moreover, the knowledge of the minimizers of \eqref{constc} is useful to derive the lower bound of the LDP. Indeed, it indicates which finite rank deformation one has to choose to get the lower bound on the deviations of $\lambda_{X_N}$, as explained in Section \ref{heuristics}.

\section{Concentration inequalities}\label{notation}
Throughout the rest of this paper, we fix a constant $\kappa>0,$ such that for all $t$ large enough,
\begin{equation} \label{tail distrib}\PP\left(|X_{1,1}|>t\right)\vee \PP\left(|X_{1,2}|>t\right)  \leq e^{-\kappa t^{\alpha}}.\end{equation}

With a slight adaptation of the concentration inequality from \cite[p. 239]{Massart}, for the largest eigenvalue of a random symmetric matrix with bounded entries, we get the following proposition.
\begin{Pro} \label{concentration vp max borne}
Let $H$ be a random Hermitian matrix with entries bounded by a constant $K>0$, such that $(H_{i,j})_{i\leq j}$ are independent variables and let $C$ be a deterministic Hermitian matrix. For all $t>0$, 
$$\PP\left( |\lambda_{H+C} -\EE(\lambda_{H+C})| >t\right) \leq 2\exp\left( -\frac{t^2}{32K^2}\right).$$
\end{Pro}

We state now a second concentration inequality we will use later in order to prove an isotropic-like property of the semi-circle law.
\begin{Pro}\label{concentration resolvante}
Let $u$ be a unit vector of $\CC^N$, and $\mu \in \RR$. Let $H$ be a random Hermitian matrix of size $N$, such that the entries $(H_{i,j})_{1\leq i\leq j\leq N}$ are independent and bounded by $K>0$. We denote by $\mathcal{C}$, the set of Hermitian matrices $X$ of size $N$, with top eigenvalue $\lambda_{X}$ strictly less that $\mu $. Let also $x \in (\mu, +\infty)$. 

(i). The function $f_u : \mathcal{C} \to \RR$ defined by
$$ f_u\left(X \right) = \left\langle u, \left(x-X\right)^{-1} u \right\rangle,$$
is convex and $1/(x-\mu)^2$-Lipschitz with respect to the Hilbert-Schmidt norm $||\ ||_{HS}$.

(ii). $f_u$ admits a convex extension to $H_N(\CC)$, denoted $\tilde{f}_u$ which is $1/(x- \mu)^2$-Lipschitz with respect to the Hilbert-Schmidt norm.
\\
Moreover, for all $x >\mu $, and all $t>0$,
$$\PP\left( \left|\tilde{f_u}(H) - \EE\left(\tilde{f_u}(H)\right)\right| > t  \right)  \leq  2\exp\left(-\frac{(x-\mu)^4t^2}{32K^2}\right).$$

\end{Pro}

\begin{proof}

(i). Let $x>\mu$. From \cite[p.117]{bhatia}, we know that $t\mapsto 1/t$ is operator convex on $(0,+\infty)$. Consequently, $t \mapsto (x-t)^{-1}$ is operator convex on $(-\infty, x)$, and in particular on $(-\infty, \mu)$. It means that the mapping $f_u$, defined on $\mathcal{C}$ by,
$$ f_u(X) = \left\langle u, \left(x-X\right)^{-1} u \right\rangle,$$
is convex. Since $x>\mu $, we have for all $X$, $Y$ in $\mathcal{C}$,
\begin{align*}
\left|f_u(X)-f_u(Y)\right| &= \left|\left\langle u, \left(\left(x-X\right)^{-1}-\left(x-Y\right)^{-1}\right) u \right\rangle\right| \\
& =\left| \left\langle u, \left(x-X\right)^{-1}\left(X-Y\right)\left(x-Y\right)^{-1} u \right\rangle \right|\\
& \leq \frac{1}{(x-\mu)^2}\left|\left| X-Y\right| \right |_{HS}.
\end{align*}
Thus, $f_u$ is convex and $1/(x-\mu)^2$-Lipschitz.

(ii). Since $f_u$ is differentiable, we can write for all $X \in \mathcal{C}$
$$ f_u(X)= \sup_{Y\in \mathcal{C}} \left(f_u(Y) + \left\langle \nabla f_u(Y),(X-Y) \right \rangle\right),$$
where $\langle \ , \ \rangle$ denotes the canonical Hermitian product on the space of Hermitian matrices of size $N$, denoted $H_N(\CC)$.
 Let $\widetilde{f}_u$ be defined for all $X\in H_N(\CC)$ by
$$\widetilde{f}_u(X) =  \sup_{Y\in \mathcal{C}} \left(f_u(Y) + \left\langle \nabla f_u(Y),(X-Y) \right \rangle\right).$$
For all $X\in H_N(\CC)$, $\tilde{f}_u(X) < +\infty$, since for all $Y \in \mathcal{C}$,
$$\left|\left| \nabla f_u(Y) \right | \right|_{HS} \leq \frac{1}{(x-\mu)^2}.$$
As a supremum of affine functions, $\widetilde{f}_u$ is convex and by the property above it is also $1/(x-\mu)^2$-Lipschitz.

We show now that $\widetilde{f}_u$ satisfies a bounded differences inequality in quadratic mean, in the sense of \cite[p.249]{Massart} (see in the Appendix Lemma \ref{concentration transport 1}) on the product space $H_N(\CC)$ of Hermitian matrices with entries bounded by $K$. Let $H$ and $H'$ be two Hermitian matrices with entries bounded by $K$. Let $\zeta(H)$ be a sub-differential of $\tilde{f}_u$ at the point $H$. Then we have,
\begin{align*}
\tilde{f}_u(H) - \tilde{f}_u(H') & \leq  \left\langle \zeta(H),(H-H') \right\rangle \\
& \leq \sum_{1\leq i \leq j \leq N}\Car_{H_{i,j} \neq H'_{i,j}} 4K\left|\zeta(H)_{i,j}\right|,
\end{align*}
where $\zeta(H)_{i,j}$ denote the $(i,j)$ coordinate of $\zeta(H)$. Since $\tilde{f}_u$ is $1/(x-\mu)^2$-Lipschitz we have,
$$ \left| \left| \zeta(H) \right| \right|_{HS} \leq \frac{1}{(x-\mu)^2}.$$
Using Lemma \ref{concentration transport 1} in the Appendix, it follows that for all $t>0$, 
$$\PP\left( \left|\widetilde{f}_u(H) - \EE\left(\widetilde{f}_u(H)\right)\right| > t  \right)  \leq  2\exp\left(-\frac{(x-\mu)^4t^2}{32K^2}\right).$$

\end{proof}

\section{Exponential tightness}\label{Exponential tightness}

The goal of this section is to prove that $(\lambda_{X_N})_{N\in \NN}$ is exponentially tight at the exponential scale $N^{\alpha/2}$. More precisely, we will prove the following.

\begin{Pro}\label{tension expo vp max} 
$$\lim_{t\to +\infty} \limsup_{N\to +\infty} \frac{1}{N^{\alpha/2}} \log \PP\left(\lambda_{X_N}>t\right) = -\infty.$$
\end{Pro}

\begin{proof}
According to Weyl's inequality (see Lemma \ref{Weyl} in the Appendix) we have, 
$$\lambda_{X_N} \leq \lambda_A +\lambda_{B^{\ee}} +\lambda_{C^{\ee}} +\lambda_{D^{\ee}},$$
where $A$, $B^{\ee}$, $C^{\ee}$, and $D^{\ee}$ are as in \eqref{cut}.
Therefore
\begin{align} \label{tension expo vp max 1}\PP\left(\lambda_{X_N}>4t\right)& \leq \PP\left(\lambda_A>t\right)+\PP\left(\lambda_{B^{\ee}}>t\right)\nonumber
\\
&+\PP\left(\lambda_{C^{\ee}}>t\right)+\PP\left(\lambda_{D^{\ee}}>t\right).\end{align}
We are going to estimate at the exponential scale $N^{\alpha/2}$ the probability of each of the events $\left\{\lambda_A>t\right\}$, $\left\{\lambda_{B^{\ee}}>t\right\}$, $\left\{\lambda_{C^{\ee}}>t\right\}$, and $\left\{\lambda_{D^{\ee}}>t\right\}$.

From the assumption \eqref{hypo1} on the tail distributions of the entries, we get the following lemma, which we state without proof.
\begin{Lem} \label{esperance queue distrib}
For $t>0$,
$$ \EE\left( \Car_{|X_{1,1}| >t} |X_{1,1}|^2 \right)\vee  \EE\left( \Car_{|X_{1,2}| >t} |X_{1,2}|^2 \right)  = O\left(e^{-\frac{\kappa}{2}t^{\alpha}}\right),$$
with $\kappa>0$ as in \eqref{tail distrib}.
\end{Lem}

We focus first on the event $\{\lambda_A>t\}$. Applying the result of Proposition \ref{concentration vp max borne}, we get the following corollary.
\begin{Cor}\label{Concentration}
For all $t>0$,
\begin{equation}\label{estimation concentration}\lim_{N\to +\infty} \frac{1}{N^{\alpha/2}} \log \PP\left(|\lambda_A - 2 |>t\right) = -\infty,
\end{equation}
where $A$ is the matrix with entries
$$ A_{i,j} = \frac{X_{i,j}}{\sqrt{N}}\Car_{\left|X_{i,j}\right| \leq (\log N)^d},$$
and $\lambda_A$ is the largest eigenvalue of $A$.
\end{Cor}

\begin{proof}

If we apply Proposition \ref{concentration vp max borne} to $A$, with $K = \frac{(\log N)^d}{\sqrt{N}}$ we get for any $t>0$,
$$\PP\left( \left| \lambda_A- \EE\left(\lambda_A\right) \right|>t/2 \right) \leq2 \exp\left( -\frac{t^2 N}{128 (\log N)^{2d}} \right).$$
 Since $\alpha<2$, we have 
\begin{equation}\label{approx expo concentration}\limsup_{N\to +\infty} \frac{1}{N^{\alpha/2}} \log \PP\left(|\lambda_A - \EE\left( \lambda_A\right) |>t/2\right) = -\infty.\end{equation}
We know from  \cite{Furedi} and  \cite{Guionnet}[2.1.27] that the largest eigenvalue of $X_N$ converges in mean to 2. Besides by Weyl's inequality (see Lemma \ref{Weyl} in the Appendix) we have,
\begin{align}
\EE\left|\lambda_A-\lambda_{X_N}\right|^2 &\leq\EE\left( \tr(A-X_N)^2\right)\nonumber\\
&= \frac{1}{N}\sum_{1\leq i,j\leq N} \EE\left(\left|X_{i,j}\right|^2\Car_{|X_{i,j}| > (\log N)^d}\right). \label{approx vp max A et X}
\end{align}
But from Lemma \ref{esperance queue distrib} we have,
$$ \EE\left(\Car_{|X_{i,j}|> \left(\log  N\right)^{d}} \left|X_{i,j}\right|^2 \right) = O\left(e^{-\frac{\kappa}{2}\left(\log N\right)^{d\alpha }} \right),$$
with $\kappa>0$ defined in \eqref{tail distrib}.
Putting the estimate above into \eqref{approx vp max A et X}, we get together with the fact that $d\alpha>1$,
$$\EE\left|\lambda_A-\lambda_{X_N}\right|^2\underset{N\to +\infty}{\longrightarrow} 0,$$
which implies 
\begin{equation} \label{conv vp max}\EE\left( \lambda_A \right ) \underset{N\to +\infty}{\longrightarrow} 2.\end{equation}
Putting together \eqref{approx expo concentration} and \eqref{conv vp max}, we get
$$\lim_{N\to +\infty} \frac{1}{N^{\alpha/2}} \log \PP\left(|\lambda_A - 2 |>t \right) = -\infty.$$

\end{proof}

We can deduce from Proposition \ref{Concentration} that for $t$ large enough, we have,
\begin{equation}\label{tensionA}\limsup_{N\to +\infty}\frac{1}{N^{\alpha/2}} \log \PP\left( \lambda_A >t\right) =-\infty.\end{equation}
For the second event $\PP\left(\lambda_{B^{\ee}} >t\right)$, we start by proving the following lemma.
\begin{Lem}\label{tension rayon spectral B}For all $t>0$, 
$$ \bornesup \PP\left(\tr\left(B^{\ee}\right)^2 >t\right) \leq -\frac{2^{\alpha/2}}{8}t \kappa\alpha \ee^{-2+\alpha},$$
with $\kappa>0$ as in \eqref{tail distrib}. 

\end{Lem}
\begin{proof}
We repeat here almost verbatim the argument used in the proof of Lemma 2.3 in \cite[p.7]{Bordenave}. We have
\begin{align*}
 \PP\left(\tr\left(B^{\ee}\right)^2> t\right)&= \PP\left(\sum_{i,j} \frac{|X_{i,j}|^2}{N}\Car_{(\log N)^d < |X_{i,j}|_{\infty}<\ee N^{1/2} } >t\right)\\
&\leq \PP\left(2\sum_{i\leq j} \frac{|X_{i,j}|^2}{N}\Car_{(\log N)^d < |X_{i,j}|_{\infty}< \ee N^{1/2} } >t \right)\\
&\leq \PP\left(\sum_{i\leq j} \frac{|X_{i,j}|^2}{N}\Car_{(\log N)^d < |X_{i,j}|< \sqrt{2} \ee N^{1/2} } >\frac{t}{2}\right),
\end{align*}
where we used in the last inequality $|X_{i,j}|_{\infty} \leq |X_{i,j}| \leq \sqrt{2} |X_{i,j}|_{\infty}$. 

Let now $\lambda>0$. By Chernoff's inequality, 
\begin{equation} \PP\left(\tr\left(B^{\ee}\right)^2> t\right)\leq e^{-\lambda \frac{t}{2}} \prod_{i\leq j}\EE\left(\exp\left(\lambda\frac{|X_{i,j}|^2}{N}\Car_{(\log N)^d <|X_{i,j}| < \sqrt{2}\ee N^{1/2}}\right)\right).\label{chernoff B}\end{equation}
We denote by $\Lambda_{i,j}$ be the Laplace transform of $\frac{|X_{i,j}|^2}{N}\Car_{(\log N)^d <|X_{i,j}| < \sqrt{2}\ee N^{1/2}}$, and by $\mu$ the distribution of $|X_{i,j}|$. Then, we have
$$\Lambda_{i,j}\left(\lambda\right)\leq 1 + \int_{(\log N)^d}^{\sqrt{2}\ee N^{1/2}} e^{\frac{\lambda x^2}{N}} \mathrm{d}\mu(x).$$
 Recall that for $\mu$ a probability measure on $\RR$, and $g\in C^1$, we have the following integration by parts formula:
$$ \int_a^b g(x) d\mu(x) = g(a)\mu\left[a, +\infty\right) - g(b) \mu\left(b, +\infty\right) + \int_a^b g'(x) \mu\left[x, +\infty\right)dx.$$
Thus,
$$\Lambda_{i,j}\left( \lambda \right) \leq 1 + \mu[(\log N)^d, +\infty)e^{\frac{\lambda(\log N)^{2d}}{N}}+ \int_{(\log N)^d}^{\sqrt{2} \ee N^{1/2}}\frac{2\lambda x}{N} e^{\frac{\lambda x^2}{N}} \mu[x,+\infty)dx. $$

We define $f(x) = \frac{\lambda x^2}{N}-\kappa x^{\alpha}$, with $\kappa$ as in \eqref{tail distrib}. For $N$ large enough we get,
\begin{align}
\Lambda_{i,j}\left( \lambda \right)
&\leq 1 + e^{f\left(\left(\log N\right)^d\right)}+\int_{(\log N)^d}^{\sqrt{2} \ee N^{1/2}} \frac{2\lambda}{N} x e^{f(x)} dx \nonumber\\
& \leq 1 + e^{f\left(\left(\log N\right)^d\right)} +4\lambda \ee^{2}\max_{[(\log N)^d, \sqrt{2} \ee N^{1/2}]} e^f. \label{majoration esperance}
\end{align}
Choose $\lambda = 2^{\alpha/2-2} \kappa\alpha \ee^{-2+\alpha}N^{\alpha/2}$. Observe that $f$ is decreasing until $x_0$ and increasing on $[x_0, +\infty)$, with $x_0$ given by 
$$x_0= \left(\frac{\kappa\alpha N}{2\lambda}\right)^{1/(2-\alpha)} = \left(2^{1- \alpha/2}N^{1-\alpha/2}\ee^{2-\alpha}\right)^{1/(2-\alpha)}  =\sqrt{2}\ee N^{1/2}.$$ 
Thus, the maximum of $e^f$ on $[(\log N)^d, \sqrt{2}\ee N^{1/2}]$ is achieved at $(\log N)^d$. Since $\alpha/2 <1$, we have for $N$ large enough, 
\begin{equation*} \label{majoration f} f\left(\left(\log N\right)^d\right) = 2^{\alpha/2-2} \kappa\alpha \ee^{-2+\alpha}N^{\alpha/2-1}(\log N)^{2d}-\kappa(\log N)^{d\alpha} \leq -\frac{\kappa}{2}(\log N)^{d\alpha}.\end{equation*}
From  \eqref{majoration esperance} and the inequality above, we get
$$\Lambda_{i,j}\left( \lambda \right)  \leq 1 + e^{-\frac{\kappa}{2}(\log N)^{d\alpha}}\left(1+2^{\alpha/2}\kappa\alpha\ee^{\alpha}N^{\alpha/2}\right).$$
Since $d\alpha>1$, we have for $N$ large enough 
$$\Lambda_{i,j}\left( \lambda \right)  \leq 1 + e^{-\frac{\kappa}{4}(\log N)^{d\alpha}} \leq \exp\left(e^{-\frac{\kappa}{4}(\log N)^{d\alpha}}\right).$$
Finally, putting this last estimate into \eqref{chernoff B} we get
\begin{equation}\label{borne tension rayon spectral B}\PP\left(\tr\left(B^{\ee}\right)^2 >t\right) \leq \exp\left(-\frac{2^{\alpha/2}}{8}t \kappa\alpha \ee^{-2+\alpha} N^{\alpha/2}\right)\exp\left(N^2e^{-\frac{\kappa}{4}(\log N)^{d\alpha}}\right),\end{equation}
which gives the claim.
\end{proof}
Coming back at the proof of Proposition \ref{tension expo vp max}, we observe that
$$\PP\left(\lambda_{B^{\ee}} >t\right) \leq \PP\left(\tr\left(B^{\ee}\right)^2>t^2\right).$$
Hence, 
\begin{equation}\label{tensionB}
\limsup_{N\to +\infty} \frac{1}{N^{\alpha/2}} \log \PP\left(\lambda_{B^{\ee}}>t\right) \leq -\frac{2^{ \alpha/2}}{8} t^2\kappa \alpha \ee^{-2+\alpha}.
\end{equation}

We focus now on the third event $\left\{\lambda_{C^{\ee}}>t\right\}$. The estimate is given by the following lemma.
\begin{Lem}\label{tensionC}For all $t>0$,
\label{tension vpmax C}\begin{equation} \label{tension expo vp max C}\bornesup \PP\left(\rho(C^{\ee})>t\right) \leq -\frac{\kappa}{4\sqrt{2}}t\ee^{\alpha+1},
\end{equation}
with $\kappa$ as in \eqref{tail distrib} and where $\rho(C^{\ee})$ denotes the spectral radius of $C^{\ee}$.
\end{Lem}

\begin{proof}
As $$\rho(C^{\ee})\leq \max_{1\leq i \leq N} \sum_{j=1}^N|C^{\ee}_{i,j}|,$$ we have
\begin{align}
\PP\left(\rho(C^{\ee})>t\right) &\leq N\PP\left(\sum_{j=1}^N |C^{\ee}_{1,j}|>t\right)\nonumber \\
&= N\PP\left(\sum_{j=1}^N |X_{1,j}|\Car_{\ee N^{1/2}\leq |X_{1,j}|_{\infty} \leq \ee^{-1} N^{1/2} } > t\sqrt{N}\right)\nonumber \\
&\leq N\PP\left(\sum_{j=1}^N |X_{1,j}|\Car_{\ee N^{1/2} \leq|X_{1,j}| \leq \sqrt{2}\ee^{-1} N^{1/2} } > t\sqrt{N}\right)\nonumber \\
&= N\PP\left(\sum_{j=1}^N Y_j >t\sqrt{N}\right), \label{1 ineg}
\end{align}
with $Y_j = |X_{1,j}| \Car_{\ee N^{1/2} \leq |X_{1,j}| \leq \sqrt{2}\ee^{-1} N^{1/2}}$.
But from Lemma \ref{esperance queue distrib} we deduce
$$\EE\left(Y_j\right)= O\left(e^{-\frac{\kappa}{2}\ee^{\alpha}N^{\alpha/2}}\right) = o\left(1/\sqrt{N}\right).$$
This yields for $N$ large enough,
\begin{equation}\PP\left(\sum_{j=1}^N Y_j >t\sqrt{N}\right) \leq \PP\left(\sum_{j=1}^N \left(Y_j-\EE\left(Y_j\right)\right) >\frac{t}{2}\sqrt{N}\right).\label{2 ineg}\end{equation}
But by Bennett's inequality (see in the Appendix  Lemma \ref{bennett}),  we have 
$$\PP\left(\sum_{j=1}^N \left(Y_j - \EE\left(Y_j\right)\right) >\frac{t}{2} \sqrt{N}\right) \leq \exp\left(-\frac{v}{2\ee^{-2}N}h\left(\frac{\ee^{-1} Nt}{\sqrt{2}v}\right)\right),$$
with $h(x) =(x+1)\log (x+1) - x$, and $v = \sum_{j=1}^N \EE\left(Y_j^2\right)$. Using again Lemma \ref{esperance queue distrib}, we find,
 \begin{equation} \label{estimation v} v = O\left( N e^{-\frac{\kappa}{2}\ee^{\alpha}N^{\alpha/2}}\right).\end{equation} 
As $h(x) \underset{x\to +\infty}{\sim} x\log x $,  we have for $N$ large enough,
 $$\PP\left( \sum_{j=1}^N \left(Y_j -\EE\left(Y_j\right)\right) >\frac{t}{2} \sqrt{N}\right) \leq \exp\left(-\frac{t}{2\sqrt{2}\ee^{-1}}\log\left(\frac{\ee^{-1}Nt}{\sqrt{2}v}\right)\right).$$
Using \eqref{estimation v}, we get
\begin{equation}\limsup_{N\to +\infty}\frac{1}{N^{\alpha/2}} \log \PP\left( \sum_{j=1}^N \left(Y_j -\EE\left(Y_j\right)\right) >\frac{t}{2} \sqrt{N}\right) \leq -\frac{\kappa}{4\sqrt{2}}t\ee^{\alpha+1}. \label{estim final}
\end{equation}
Putting together inequalities \eqref{1 ineg} and \eqref{2 ineg} with the last exponential estimate \eqref{estim final}, we get the claim 
$$\bornesup \PP\left(\rho(C^{\ee})>t\right) \leq -\frac{\kappa}{4\sqrt{2}}t\ee^{\alpha+1}.$$
\end{proof}
Finally, we now turn to the estimation of the last event $\PP\left( \lambda_{D^{\ee}} >t\right)$. It will directly fall from the following lemma.
\begin{Lem}\label{tension rayon spectral D} For all $t>0$,
$$\limsup_{N\to +\infty} \frac{1}{N^{\alpha/2}} \log \PP\left( \rho\left(D^{\ee}\right)>t\right) \leq -\frac{\kappa}{2}  \ee^{-\alpha}.$$
where $\rho\left(D^{\ee}\right)$ denotes the spectral radius $D^{\ee}$, and $\kappa$ is as in \eqref{tail distrib}. 
\end{Lem}
\begin{proof}Just as in the proof of Lemma \ref{tensionC}, we have
 $$\PP\left(\rho\left(D^{\ee}\right) > t\right)
 \leq N \PP\left(\sum_{j=1}^N \frac{|X_{1,j}|}{\sqrt{N}} \Car_{\ee^{-1}N^{1/2}< |X_{1,j}|} >t\right).$$
By Markov's inequality we get 
$$\PP\left(\rho\left(D^{\ee}\right) > t\right)
  \leq \frac{\sqrt{N}}{t}\sum_{j=1}^N \EE\left(|X_{1,j}| \Car_{\ee^{-1}N^{1/2} <|X_{1,j}|} \right).$$
From Lemma \ref{esperance queue distrib} we deduce
$$\EE\left(|X_{1,j}| \Car_{\ee^{-1}N^{1/2} < |X_{1,j}|} \right) = O\left(e^{-\frac{\kappa}{2} \ee^{-\alpha} N^{\alpha/2} } \right).$$
Therefore,
$$\PP\left(\rho\left(D^{\ee}\right) >t\right) = O\left(N \sqrt{N}e^{-\frac{\kappa}{2} \ee^{-\alpha} N^{\alpha/2}}\right),$$
which gives the claim.
\end{proof}
Putting together the different estimates \eqref{tensionA}, \eqref{tensionB}, \eqref{tension expo vp max C} and \eqref{tension rayon spectral D} , and using inequality \eqref{tension expo vp max 1}, we get
\begin{equation} \label{borne sup vp max temp}\bornesup \PP( \lambda_{X_N} >4t) \leq  - C_1 \min\left(t^2\ee^{-2+\alpha}, t\ee^{\alpha+1}, \ee^{-\alpha}\right),\end{equation}
where $C_1$ is some constant small enough.
Taking the limsup as $t$ goes to infinity, and then the limsup as $\ee$ goes to $0$, we get finally 
$$\limsup_{t\to +\infty} \bornesup \PP\left( \lambda_{X_N} >4t\right) \leq -\infty.$$

\end{proof}

We show now that at the exponential scale we consider, $C^{\ee}$ has a bounded number of non-zero entries. This will be crucial later when we will see $C^{\ee}$ as a finite rank perturbation of the matrix $A$.

\begin{Pro} \label{entrees non nulles C}
For all $\ee> 0$,
$$\lim_{r\to +\infty}\limsup_{N\to +\infty} \frac{1}{N^{\alpha/2}} \log \PP\left(\mathrm{Card}\{(i,j) : C_{i,j}^{\ee}\neq 0 \}>r\right) = -\infty.$$

\end{Pro}

\begin{proof}We follow here the argument of the proof of Lemma 2.2 in \cite[p. 6]{Bordenave}. We have,
\begin{align*}
\PP\left(\mathrm{Card}\{(i,j) : C_{i,j}^{\ee}\neq 0 \}>r\right) & =  \PP\left(\sum_{i,j} \Car_{C^{\ee}_{i,j}\neq 0}>r\right)\\
& \leq \PP\left( \sum_{i\leq j} \Car_{|X_{i,j}|_{\infty}\geq \ee N^{1/2}} >r/2\right)\\
& \leq  \PP\left( \sum_{i\leq j} \Car_{|X_{i,j}|\geq \ee N^{1/2}} >r/2\right).\\
\end{align*}
Let $p_{i,j} = \PP\left(|X_{i,j}|\geq \ee N^{1/2}\right)$. From \eqref{tail distrib}, we get that $p_{i,j} =o\left(1/N^2\right)$. Therefore it is enough to show that for any $r>0$,
$$\limsup_{r \to +\infty} \limsup_{N \to +\infty}\frac{1}{N^{\alpha/2}}\log \PP\left( \sum_{i\leq j} \left(\Car_{|X_{i,j}|\geq \ee N^{1/2}} - p_{i,j} \right) >r\right) = -\infty.$$
Using Bennett's inequality (see in the Appendix Proposition \ref{bennett}), we get
$$\PP\left(\sum_{i\leq j} \left(\Car_{|X_{i,j}| \geq \ee N^{1/2}} -p_{i,j}\right)>r\right)\leq \exp\left(-v h\left(\frac{r}{v}\right)\right),$$
with $h(x) = (x+1)\log (x+1) -x$, and $v = \sum_{i\leq j} p_{i,j}$. As $h(x) \underset{+\infty}{\sim} x\log x$, we have for $N$ large enough,
\begin{align}
\PP\left(\sum_{i\leq j}\left( \Car_{|X_{i,j}|\geq \ee N^{1/2}} -p_{i,j}\right) >r \right)&\leq\exp\left(-r \log\left(\frac{r}{v}\right)\right) \nonumber \\
& \leq \exp\left(r\log\left(rN^2\right)\right) \exp\left(-r\kappa\ee^{\alpha}N^{\alpha/2}\right),\label{majnbentree}
\end{align}
where we used in the last inequality the fact that $v \leq N^2 e^{-\kappa \ee^{\alpha}N^{\alpha/2} }$, with $\kappa$ as in \eqref{tail distrib}. Taking the limsup  at the exponential scale in \eqref{majnbentree}, we get the claim.

\end{proof}
As a consequence of the latter proposition, we get the following result.
\begin{Pro}\label{rangC}
For all $\ee> 0$,
$$\lim_{r\to +\infty}\limsup_{N\to +\infty} \frac{1}{N^{\alpha/2}} \log \PP\left(\rk\left(C^{\ee}\right)>r\right) = -\infty.$$
\end{Pro}

\begin{proof}
As the rank of a matrix is bounded by the number of non-zero entries, we see that Proposition \ref{entrees non nulles C} yields the claim.
\end{proof}

\section{Exponential equivalences}\label{Exponential equivalences}

\subsection{First step} \label{1Equiv}
We show here that we can neglect at the exponential scale $N^{\alpha/2}$, the contributions of the very large entries (namely those such that $|X_{i,j}|_{\infty}> \ee^{-1}\sqrt{N}$) and the intermediate entries (namely those such that $(\log N)^d < |X_{i,j}|_{\infty} <\ee \sqrt{N}$) to the deviations of the largest eigenvalue of $X_N$.

\begin{Pro}\label{entrées intermédiaires}For all $t>0$,
$$\lim_{\ee \to 0} \limsup_{N \to +\infty} \frac{1}{N^{\alpha/2}}\log \PP\left(|\lambda_{A+C^{\ee}} - \lambda_{X_N}|>t\right) = -\infty,$$
where $A$ and $C^{\ee}$ are as in \eqref{cut}.
In short, $(\lambda_{A+C^{\ee}})_{N\in \NN, \ee>0}$ are exponentially good approximations of $(\lambda_{X_N})_{N\in \NN}$.
\end{Pro}

\begin{proof} 

We have by Weyl's inequality (see Lemma \ref{Weyl} in the Appendix),
\begin{equation} \label{major}\PP\left(|\lambda_{A+C^{\ee}} - \lambda_{X_N}|>t \right) \leq \PP\left(\rho(B^{\ee}) >t/2 \right) +  \PP\left(\rho\left(D^{\ee}\right) > t/2\right).\end{equation}
But we know by Lemma \ref{tension rayon spectral D} and \ref{tension rayon spectral B}, that
$$\bornesup \PP\left(\rho\left(D^{\ee}\right) >\frac{t}{2}\right) \leq  -\frac{\kappa}{2} \ee^{-\alpha},$$
and
$$ \bornesup \PP\left(\tr\left(B^{\ee}\right)^2 >\frac{t}{2}\right) \leq -\frac{2^{\alpha/2}}{16}t \kappa\alpha \ee^{-2+\alpha},$$
with $\kappa$ as in \eqref{tail distrib}.
Thus, taking the limsup at the exponential scale $N^{\alpha/2}$ in \eqref{major}, and then the limsup as $\ee$ goes to $0$, recalling that $\alpha < 2$, we get the claim.

\end{proof}

\subsection{Second step }\label{2Equiv}
We now show that in the study of the deviations of $\lambda_{A +C^{\ee}}$, we can consider $A$ and $C^{\ee}$ to be independent. We will prove the following result.

\begin{The} \label{equivalence expo}We denote by $P_N$ the law of $X_{1,1}$ conditioned on the event $\{|X_{1,1}|_{\infty}\leq (\log N)^d\}$ and by $Q_N$ the law of $X_{1,2}$ conditioned on the event $\{|X_{1,2}|_{\infty} \leq (\log N)^d \}$. Let $H$ be a random Hermitian matrix independent of $X$ such that $(H_{i,j})_{1\leq i\leq j \leq N}$ are independent, and for $1\leq i \leq N$, $H_{i,i}$ has law $P_N$, and for all $i<j$,  $H_{i,j}$ has law $Q_N$. We denote by $H_N$ the normalized matrix $H/\sqrt{N}$.

 We have for all $t>0$,
$$\lim_{\ee \to 0}\bornesup \PP\left(\left|\lambda_{X_N}- \lambda_{H_N+C^{\ee}}\right|>t\right) = -\infty.$$ 
\end{The}

With a similar argument as in the proof of Proposition \ref{entrees non nulles C}, we get the following lemma.
\begin{Lem}\label{tension1}
Let $I = \{(i,j) : |X_{i,j}|_{\infty} > (\log N)^d \}$. For all $t>0$,
$$\lim_{N\to +\infty} \frac{1}{N^{\alpha/2}} \log \PP(|I|> t N^{\alpha/2}) = -\infty.$$
\end{Lem}

\begin{proof}[Proof of Theorem \ref{equivalence expo}]
Due to Proposition \ref{entrées intermédiaires}, it is enough to prove for any $\ee>0$ and any $t>0$,
$$\bornesup \PP\left(\left|\lambda_{A+C^{\ee}}- \lambda_{H_N+C^{\ee}}\right|>t\right) = -\infty.$$ 
We will follow the same coupling argument to remove the dependency between $A$ and $C^{\ee}$, as in the proof of Proposition 2.1 in \cite{Bordenave}. 

Let $I = \{(i,j) : |X_{i,j}|_{\infty}> (\log N)^d\}$. Let $A'$ be the $N\times N$ matrix with $(i,j)$-entry,
$$A'_{i,j} = \Car_{(i,j) \notin I } A_{i,j} + \Car_{(i,j) \in I}\frac{H_{i,j}}{\sqrt{N}}.$$
Let $\mathcal{F}$ be the $\sigma$-algebra generated by the random variables $X_{i,j}$ such that $(i,j) \in I$.
Then $A'$ and $H_N$ are independent of $\mathcal{F}$ and have the same law. By Weyl's inequality (see Lemma \eqref{Weyl} in the appendix),
\begin{align}
\left|\lambda_{A+C^{\ee}}-\lambda_{A'+C^{\ee}}\right|^2 
&\leq  \tr\left(A-A'\right)^2\nonumber\\
& = \sum_{i,j} \left|A_{i,j}-A'_{i,j}\right|^2\nonumber\\
& = \frac{1}{N} \sum_{i,j}\left( \Car_{(i,j)\in I} \left|H_{i,j}\right|^2\right)\nonumber\\
&\leq |I| \frac{(\log N)^{2d}}{N}\label{control cond}.
\end{align}
Let $t>0$. Define the event $F = \left\{|I| < t^2 N/(\log N)^{2d} \right\}$. Then, by Lemma \ref{tension1} we have, 
\begin{equation}\label{approx}\lim_{N\to +\infty} \frac{1}{N^{\alpha/2}} \log \PP\left(F^c\right) = -\infty.\end{equation}
But according to \eqref{control cond},
\begin{equation}\label{prox}\Car_F \left|\lambda_{A+C^{\ee}}- \lambda_{A'+C^{\ee}}\right|\leq t.\end{equation}
Thus,
$$ \lim_{N\to +\infty} \frac{1}{N^{\alpha/2}} \log \PP\left(\left|\lambda_{A+C^{\ee}} - \lambda_{A'+C^{\ee}} \right|> t\right ) = -\infty.$$
But $C^{\ee}$ is $\mathcal{F}$-measurable, and conditioned by $\mathcal{F}$, $A'$ is a random Hermitian matrix with up-diagonal entries independent and bounded by $(\log N)^d /\sqrt{N}$. According to Proposition \ref{concentration vp max borne},  we have
$$ \bornexpo \PP\left( \left| \lambda_{A'+C^{\ee}} - \EE_{\mathcal{F}}\left( \lambda_{A'+C^{\ee}}\right) \right|>t \right) = -\infty,$$
where $\EE_{\mathcal{F}}$ denotes the conditional expectation with respect to $\mathcal{F}$.
Applying again Proposition \ref{concentration vp max borne} to $H_N$ and $C^{\ee}$, we get 
$$ \bornexpo \PP\left( \left| \lambda_{H_N+C^{\ee}} - \EE_{\mathcal{F}}\left( \lambda_{H_N+C^{\ee}}\right) \right|>t \right) = -\infty.$$
But $A'$ and $H_N$ are independent of $\mathcal{F}$ and have the same law. Therefore,
$$ \EE_{\mathcal{F}}\left( \lambda_{A'+C^{\ee}}\right) = \EE_{\mathcal{F}}\left( \lambda_{H_N+C^{\ee}}\right).$$
Thus by triangular inequality,
$$ \lim_{N\to +\infty} \frac{1}{N^{\alpha/2}} \log \PP\left(\left|\lambda_{A+C^{\ee}} - \lambda_{H_N+C^{\ee}} \right|> 3t\right ) = -\infty,$$
which ends the proof.
\end{proof}

\subsection{Exponential approximation of the equation of eigenvalues outside the bulk}\label{equation vpmax}
As a consequence of the LDP for the empirical spectral measure proved in \cite{Bordenave}, we show in the next proposition that the deviations at the left of $2$ have an infinite cost at the exponential scale $N^{\alpha/2}$. This result will allow us to focus only on understanding the deviations of the largest eigenvalue at the right of $2$. 
\begin{Pro}\label{pgdinf2}
$$\forall x<2, \quad \limsup_{N \to +\infty} \frac{1}{N^{\alpha/2}} \log \PP\left(\lambda_{X_N} \leq x\right) =- \infty.$$
\end{Pro}

\begin{proof}
According to \cite{Bordenave}, we know that the empirical spectral measure $\mu_{X_N}$ satisfies a LDP with speed  $N^{1+\alpha/2}$, and with good rate function $I$ which achieves  $0$ only for the semicircular law $\sigma_{sc}$. Let $x<2$ and $h$ be a bounded continuous function whose support is in $(x, 2)$, and such that $ \sigma_{sc}( h) =1$. We have
$$\PP\left( \lambda_{X_N} \leq x \right)  \leq \PP\left(  \mu_{X_N} (h) = 0\right) .$$ 
But $F =\left \{ \mu \in \mathcal{M}_1(\RR) : \mu ( h)  = 0 \right\}$ is a closed set with respect to the weak topology and it does not contain $\sigma_{sc}$. Then
$$ \limsup_{N\to +\infty}\frac{1}{N^{1+\alpha/2}} \log \PP\left(  \mu_{X_N}(h) = 0\right) = - \inf_F I.$$
Since  $\sigma_{sc}\notin F$, $\inf_F I>0$. Thus,
$$ \limsup_{N\to +\infty}\frac{1}{N^{\alpha/2}} \log \PP\left( \lambda_{X_N} \leq x\right)  = - \infty.$$
\end{proof}

In the view of Theorem \ref{equivalence expo}, Proposition \ref{controlesp}, and Proposition \ref{pgdinf2}, we are reduced to understand the deviations in $(2,+ \infty)$, at the exponential scale $N^{\alpha/2}$, of the largest eigenvalue of the perturbed matrix  $H_N+C^{\ee}$, where $C^{\ee}$ can be assumed, due to Proposition \ref{rangC} to be a finite rank matrix. We will use here the same approach as in many papers on finite rank deformations of Wigner matrices (see for example \cite{Benaych} or \cite{Maida}) to determine the behavior of the extreme eigenvalues outside the bulk of a perturbed Wigner matrix. This approach is based on a determinant computation, stated here without proof, in the following lemma. It is a direct consequence of Frobenius formula (see Proposition \ref{formule Frobenius} in the Appendix).  
\begin{Lem}\label{equation vp}
Let $H$ and $C$ be two Hermitian matrices of size $N$. Denote by $k$ the rank of $C$, by $\theta_1,...,\theta_k$ the non-zero eigenvalues of $C$ in nondecreasing order and $u_1,...,u_k$ orthonormal eigenvectors associated with these eigenvalues. Let $Sp(H)$ be the spectrum of $H$. If $\lambda_{H+C} \notin Sp(H)$, then it is the largest zero of $f_N$, where $f_N$ is defined for all $z \notin Sp(H)$ by
$$ f_N(z) = \det\left(M_N(z)\right), \text{ where } M_N(x) = I_k - \left(\theta_i\langle u_i, \left(x-H\right)^{-1} u_j \rangle\right)_{1\leq i,j \leq k}.$$

\end{Lem}

To make this strategy works, we need a control on the spectrum of $H_N$ which will allow us to assume that the spectrum of $H_N$ is nearly included $(-\infty,2]$ at the exponential scale we consider. 
As a consequence of Proposition \ref{concentration vp max borne}, and arguing similarly as in the proof of Corollary \ref{Concentration}, we get the following proposition.
\begin{Pro}[Control on the spectrum of $H_N$]\label{controlesp}
 Let $\delta>0$. Define $$C_{\delta}= \left\{X \in H_N(\CC) : \lambda_{X} <2 +\delta\right\}.$$ Then,
 $$\lim_{N\to +\infty} \frac{1}{N^{\alpha/2}} \log \PP\left(H_N \notin C_{\delta}\right) = -\infty,$$
with $H_N$ is as in Theorem \ref{equivalence expo}.

\end{Pro}

The goal of this section is to prove an exponential approximation of the equation of the eigenvalues of the perturbed matrix on every compact subset of $(2,+\infty)$. We will prove the following result.
\begin{The} \label{approx equation au vp} Let $H_N$ be as in Theorem \ref{equivalence expo} and let $C_N$ be an independent random Hermitian matrix. Let $k$ be the rank of $C_N$, $\theta_1,...,\theta_k$ the non-zero eigenvalues  in non-decreasing order of $C_N$ and $u_1,...,u_k$ orthonormal eigenvectors of $C_N$ associated with those eigenvalues. 

Let $\delta>0$, $\rho>0$, and $r\in \NN$.  Define the event
\begin{equation} \label{def W} W =\left\{ \rk(C_N) = r, \ \rho(C_N) \leq \rho,\  \lambda_{H_N} \leq 2+\delta  \right\},\end{equation}
where $\rho(C_N)$ is the spectral radius of $C_N$.
 For any $t>0$, and any compact subset $K$ of $(2+\delta, + \infty)$,
$$\bornesup \PP\left(\left\{\sup_{ x \in K} \left|f_N(x) - f(x)\right| >t\right\}\cap W \right) =-\infty,$$
where $f_N$ is defined for any $x \notin Sp(H_N)$ by
$$f_N(x) = \det\left(M_N(x)\right), \text{ with }  M_N(x) = I_k - \left(\theta_i\langle u_i, \left(x-H_N\right)^{-1} u_j \rangle\right)_{1\leq i,j \leq k},$$
$f$ is defined for any $x>2$ by $f(x) = \det\left(M(x)\right)$, with 
$$M(x) = I_k - 
  \left(
     \raisebox{0.5\depth}{%
       \xymatrixcolsep{1ex}%
       \xymatrixrowsep{1ex}%
       \xymatrix{
         \theta_1G_{\sigma_{sc}}(x) \ar @{.}[ddddrrrr]& 0 \ar @{.}[rrr] \ar @{.}[dddrrr] &  & & 0  \ar @{.}[ddd]  \\
         0 \ar@{.}[ddd] \ar@{.}[dddrrr]& & & & \\
         &&&& \\
         &&&& 0 \\
        0 \ar@{.}[rrr] & & & 0 & \theta_k G_{\sigma_{sc}}(x)
       }%
     }
   \right),
$$
where we denote by $G_{\sigma_{sc}}(x)$ the Stieltjes transform of the semi-circular law.

\end{The}
\subsection{First step}
We start by showing that $M_N$ is close to its conditional expectation given $C_N$. 
As a consequence of Proposition \ref{concentration resolvante}, we get the following  concentration result.
\begin{Pro}\label{concentration forme quadra}
Let $u,v$ be two unit vectors. 
Define for all $x>2+\delta$,
$$b_N(u,v) = \Car_{H_N \in C_{\delta}}\left\langle u, \left(x-H_N\right)^{-1} v \right\rangle,$$
where $H_N$ is as in Theorem \ref{equivalence expo}, and $C_{\delta} = \{ X \in H_N(\CC) : \lambda_X < 2 +\delta\}$. For any $t>0$,
$$ \lim_{N\to +\infty} \frac{1}{N^{\alpha/2}} \log \sup_{||u|| = ||v|| =1 }\PP\left( \left|b_N(u,v) -\EE\left(b_N(u,v)\right) \right|>t \right) = -\infty.$$
\end{Pro}

\begin{proof}Since $b_N$ is a bilinear form, by the polarization formula we see that we only need to prove,
$$ \lim_{N\to +\infty} \frac{1}{N^{\alpha/2}} \log \sup_{||u|| =1 }\PP\left( \left|b_N(u,u) -\EE\left(b_N(u,u)\right) \right|>t \right) = -\infty.$$
By assumption, $H_N$ has its entries bounded by $(\log N)^d/\sqrt{N}$. Applying Proposition \ref{concentration resolvante} with $\mu = 2+\delta$, we get that for any $x>2+\delta$,
 \begin{equation} \label{concentration f tilde} \lim_{N\to +\infty} \frac{1}{N^{\alpha/2}} \log \sup_{||u||=1} \PP\left(\left| \tilde{f_u}(H_N) - \EE \left( \tilde{f_u}(H_N) \right) \right| > t \right) = -\infty,\end{equation}
where $\tilde{f_u}$ is a convex extension of $f_u$ which is defined on $C_{\delta}$ by 
$$f_u(Y) = \left\langle u, \left(x-Y\right)^{-1} u \right\rangle.$$
Furthermore, $\widetilde{f}_u$ is $1/(x-2-\delta)^2$-Lipschitz, with respect to the Hilbert-Schmidt norm.
We have for all $t>0$,
\begin{equation} \label{controle proba} \PP\left(\left| \tilde{f_u}(H_N) - b_N(u,u) \right | > t \right) \leq \PP\left( \lambda_{H_N}\notin C_{\delta}\right), \end{equation}
which, invoking Proposition \ref{controlesp} yields,
\begin{equation} \label{approx forme quadra} \lim_{N\to +\infty} \frac{1}{N^{\alpha/2}} \log \sup_{||u||=1} \PP\left(\left| \tilde{f_u}(H_N) - b_N(u,u) \right | > t \right) = -\infty.\end{equation}
Moreover,
$$\left| \tilde{f_u}(H_N) - b_N(u,u) \right | \leq  \Car_{\lambda_{H_N}\notin C_{\delta}}\sup_{\mathcal{K}_N} \left|\tilde{f_u}\right|,$$
where the supremum is taken over the set $\mathcal{K}_N$ of Hermitian matrices of size $N$ with entries bounded by $(\log N)^d/\sqrt{N}$. 
Thus,
\begin{equation} \label{esperance} \EE\left| \tilde{f_u}(H_N ) - b_N(u,u)\right| \leq  \sup_{\mathcal{K}_N} \left|\tilde{f_u}\right| \PP\left( \lambda_{H_N}\notin C_{\delta} \right).\end{equation}
It only remains to show that
\begin{equation} \label{convergence esperance}\sup_{||u||=1}\EE\left| \tilde{f_u}(H_N ) - b_N(u,u)\right|  \underset{N\to +\infty}{\longrightarrow} 0.\end{equation}
Indeed, putting together \eqref{concentration f tilde} with \eqref{approx forme quadra} and the claim above, we will get by the triangular inequality, 
$$ \lim_{N\to +\infty} \frac{1}{N^{\alpha/2}} \log \sup_{||u||=1} \PP\left(\left|b_N(u,u) - \EE\left(b_N(u,u) \right) \right| >2t \right) = -\infty.$$
We now show \eqref{convergence esperance}. Since $x>2+\delta$, we have for all $H' \in C_{\delta}$,
$$ \left|f_u(H')\right| \leq \frac{1}{\eta},$$
with $\eta = x-(2+\delta)$.
Let $H$ be a Hermitian matrix with entries bounded by $(\log N)^d/\sqrt{N}$. We have,
$$ \left |\tilde{f_u}(H) \right | \leq \left |\tilde{f}_u(H) -\tilde{f}_u\left(\frac{H}{\left|\left|H\right|\right|+1}\right)\right | + \left| \tilde{f}_u\left(\frac{H}{\left|\left|H\right|\right|+1}\right) \right|.$$
But $H/(\left|\left| H\right| \right| +1)$ is in $C_{\delta}$, thus $\left|f_u\left(H/(\left|\left| H\right| \right| +1)\right)\right| \leq \frac{1}{\eta}$. Besides $\tilde{f}_u$ is $1/\eta^2$-Lipschitz with respect to the Hilbert-Schmidt norm. Therefore,
\begin{align*}
 \left |\tilde{f}_u(H) \right |&  \leq  \frac{1}{\eta^2}||H||_{HS} + \frac{1}{\eta}\\
& \leq \frac{\sqrt{N}(\log N)^d}{\eta^2} + \frac{1}{\eta}\leq \frac{2\sqrt{N}(\log N)^d}{\eta^2}.
\end{align*}
We deduce that $$\sup_{\mathcal{K}_N} \left|\tilde{f}_u\right| \leq \frac{2\sqrt{N}(\log N)^d}{\eta^2}.$$
From Proposition \ref{controlesp} we get, 
$$\EE\left| \tilde{f}_u(H_N ) - b_N(u,u)\right| \underset{N\to +\infty}{\longrightarrow} 0,$$
which ends the proof of the claim.
\end{proof}

We are now ready to prove that $M_N$, restricted to the event that the spectrum of $H_N$ is in $(-\infty, 2+\delta)$ for some $\delta>0$, is exponentially equivalent to its conditional expectation given $C_N$, uniformly on any compact subset of $(2+\delta, +\infty)$.
\begin{Pro}[Concentration in the equation of eigenvalues outside the bulk]\label{concentration equation}Let $H_N$ be as in Theorem \ref{equivalence expo}, and let $C_N$ be an independent random Hermitian matrix.
Let $k$ be the rank of $C_N$, $\theta_1,...,\theta_k$ the non-zero eigenvalues in non-decreasing order, and $u_1,...,u_k$ orthonormal eigenvectors associated with these eigenvalues. For all $x> 2 +\delta $,  we define
$$\widetilde{M}_N(x) =  I_k - \left( \theta_i\left\langle u_i, \Car_{H_N \in C_{\delta}}\left(x-H_N\right)^{-1} u_j \right \rangle \right)_{1\leq i,j \leq k},$$
where $C_{\delta} = \{ X\in H_N(\CC) : \lambda_X < 2+\delta\}$, and where $H_N$ is as in Theorem \ref{equivalence expo}.

Let $t>0$ and $\rho>0$. For any compact subset $K$ of $(2+\delta, +\infty)$,
$$\lim_{N\to+\infty} \frac{1}{N^{\alpha/2}}\log \PP\left(\left\{\sup_{x\in K} \left|\widetilde{M}_N(x)  -\EE_{C_N}\left(\widetilde{M}_N(x)\right)\right|_{\infty} >t \right\} \cap V \right) = -\infty,$$
where
$$V = \left\{ \rk(C_N)=r,\ \rho(C_N) \leq \rho\right\},$$
and  $\EE_{C_N}$ denotes the conditional expectation given $C_N$, and where for any matrix $M$,
$\left|M\right|_{\infty} = \sup_{i,j } \left|M_{i,j}\right|$.

\end{Pro}

\begin{proof}
 Fix $x$ in $(2+\delta, +\infty)$ and $i,j \in \{1,...,r\}$. We will denote by $\PP_{C_N}$ the conditional probability given $C_N$. We have,
$$\Car_{V} \PP_{C_N}\left(\left| \widetilde{M}_N(x)_{i,j} - \EE_{C_N}\left(\widetilde{M}_N(x)_{i,j}\right)\right| >t \right)
\leq \sup_{||u||=||v||=1}\PP\left(\rho \left|b_N(u,v) - \EE \left(b_N(u,v)\right)\right| >t\right),$$
where $b_N(u,v)$ is as in Proposition \ref{concentration forme quadra}.
Thus, from Proposition \ref{concentration forme quadra}, we get
$$\lim_{N\to +\infty} \frac{1}{N^{\alpha/2}} \log \PP\left( \left\{\left|\widetilde{M}_N(x)_{i,j}  -\EE_{C_N}\left(\widetilde{M}_N(x)_{i,j}\right)\right| >t\right\} \cap V\right) =-\infty.$$
Taking the union over all the $i,j$ in $\{1,...,r\}$, we get for any $x  \in (2+\delta, +\infty)$,
$$\lim_{N\to +\infty} \frac{1}{N^{\alpha/2}} \log \PP\left( \left\{\left|\widetilde{M}_N(x) -\EE_{C_N}\left(\widetilde{M}_N(x)\right)\right|_{\infty} >t\right\} \cap V\right) =-\infty.$$
We now use a $\ee$-net argument to extend this exponential equivalence uniformly in $z$ in a given compact subset $K$ of $(2+\delta, +\infty)$. Let $n\in \NN$. Since $K$ is compact, there are a finite number of points in $\{ x \in K : nx \in \ZZ\}$. Taking the union bound, we deduce that for any $t>0$,
\begin{equation} \label{eenet} \lim_{N\to +\infty} \frac{1}{N^{\alpha/2}} \log\PP\left( \left\{\sup_{\underset{nx \in \ZZ}{x \in K}} \left|\widetilde{M}_N(x) - \EE_{C_N}\left(\widetilde{M}_N(x)\right)\right|_{\infty} > t\right\}\cap V\right) = - \infty.\end{equation}
Note that provided $\rho(C^{\ee}) \leq \rho$, we have for any $x,y\in K$,
$$ \left|  \widetilde{M}_N(x) - \widetilde{M}_N(y)\right|_{\infty} \leq \rho|x-y| \Car_{H_N \in C_{\delta}} ||(x-H_N)^{-1}||. ||(y-H_N)^{-1}||\leq \frac{\rho}{\eta^2} |x-y|,$$
where $\eta = \inf K - (2+\delta)$. Therefore, on the event $V$, the function $x\in K\mapsto \widetilde{M}_N(x)$ is $\rho/\eta^2$-Lipschitz with respect to the norm $|  \ |_{\infty}$, and we have,
$$\sup_{ x \in K} \left|\widetilde{M}_N(x) - \EE_{C_N}\left(\widetilde{M}_N(x)\right)\right|_{\infty} \leq \sup_{\underset{nx \in \ZZ}{ x \in K}} \left|\widetilde{M}_N(x) - \EE_{C_N}\left(\widetilde{M}_N(x)\right)\right|_{\infty}+\frac{2\rho}{n\eta^2}.$$
Taking $n$ large enough, we get from \eqref{eenet} and the inequality above, that for any $t>0$,
$$\lim_{N\to +\infty} \frac{1}{N^{\alpha/2}} \log\PP\left(\left\{\sup_{ x \in K} \left |\widetilde{M}_N(x) -\EE_{C_N}\left(\widetilde{M}_N(x)\right)\right|_{\infty} >t \right\}\cap V\right)=-\infty.$$

\end{proof}

The second step of the proof of Theorem \ref{approx equation au vp} will be to prove an isotropic-like property of the semicircular law. This will be made possible due to the results on estimates of the coefficients of the resolvent of Wigner matrices in \cite{resolvantentries}. This is where our assumption on the independence between the real and imaginary parts of the entries of our Wigner matrix $X$ plays its role.
\begin{The}\label{conv unif resolvante}
For any compact subset $K$ of $(2+\delta, +\infty)$,
$$\sup_{ x \in K}\sup_{||u|| = ||v|| = 1} \left| \left \langle u, \EE \left(\Car_{H_N \in C_{\delta}} \left(x-H_N\right)^{-1}\right) v \right\rangle - \left\langle u, v\right\rangle G_{\sigma_{sc}}(x) \right|\underset{N \to +\infty}{ \longrightarrow } 0,$$
where $C_{\delta} = \{X \in H_N(\CC) : \lambda_X < 2+\delta \}$, and where $H_N$ is as in Theorem \ref{equivalence expo}.
\end{The}
\begin{proof}
Let $u$ and $v$ be two unit vectors. Let $K$ be a compact subset of $(2+\delta, +\infty)$. Set $\eta= \inf K - (2+\delta)$. To ease the notation, we denote for any $z \notin Sp(H_N)$,  the resolvent of $H_N$,  $R(z)= \left(z-H_N\right)^{-1}$. 
Let $y>0$ and $x\in K$. We write $z = x +i y$. We have,
$$\Car_ {H_N \in C_{\delta}}\left| \left\langle u, R(x) v \right\rangle  - \left\langle u, R(z) v \right\rangle \right|  \leq \Car_ {H_N \in C_{\delta}} \left| \left | \left(x-H_N \right)^{-1}(z-x) \left(z-H_N\right)^{-1} \right| \right|\leq \frac{y}{\eta^2}.$$
Thus,
$$\EE\left|\Car_ {H_N \in C_{\delta}} \left\langle u, R(x) v \right\rangle- \left\langle u, R(z) v \right\rangle \right| \leq  \frac{y}{\eta^2}+ \frac{1}{y}\PP\left(H_N \notin C_{\delta}\right).$$
Take $y = 1/ \log N$.  From Proposition \ref{controlesp}, we get uniformly for $x$ in $K$,
\begin{equation}\sup_{||u||=||v||=1} \label{approx hors de R}\EE\left|\Car_ {H_N \in C_{\delta}} \left\langle u, R(x) v \right\rangle- \left\langle u, R\left(x+\frac{i}{\log N}\right) v \right\rangle \right| \underset{N\to +\infty}{\longrightarrow} 0.\end{equation}
Thus, we only need to show,
$$\sup_{||u||=||v||=1}\left| \EE \left(\left\langle u, R\left(x+\frac{i}{\log N}\right) v \right\rangle\right)  - \left\langle u, v\right\rangle G_{\sigma_{sc}}(x) \right| \underset{N\to +\infty}{\longrightarrow} 0,$$
uniformly for $x\in K$.

Expanding the scalar product and using the exchangeability of the entries of $H_N$, we get
\begin{align*}
\left\langle u, \EE R(z) v \right\rangle & = \sum_{1\leq i,j \leq N} \overline{u_i}\EE R_{i,j}(z) v_j\nonumber\\
& =\left \langle u, v \right \rangle \EE R_{1,1}(z)+ \sum_{i\neq j} \overline{u_i} v_j \EE R_{1,2}(z) \\
& = \left\langle u,v\right \rangle \frac{1}{N}\EE \tr R(z)+ \sum_{i\neq j} \overline{u_i} v_j \EE R_{1,2}(z) .
\end{align*}
 Since $u$ and $v$ are unit vectors, 
\begin{equation} \label{isotropie}\left|\left\langle u, \EE R(z)  v \right\rangle - \left \langle u, v \right \rangle \EE\left(\frac{1}{N} \tr R(z)\right) \right| \leq N\left|\EE R_{1,2}(z)\right|.\end{equation}
But since the entries of $X$ have finite fifth moment and their real and imaginary parts are independent, we have according to Proposition 3.1 in \cite{resolvantentries}, 
\begin{equation}\EE R_{1,2}(X_N)(z) = O\left(\frac{P_9\left(1/\left|\Im(z)\right|\right)}{ N^{3/2}}\right),\label{estimation coeff non diag}\end{equation}
uniformly for $z \in \CC\setminus \RR$, where we denote by $R(X_N)$ the resolvent of $X_N$, and where $P_9$ is a polynomial of degree $9$. 
But recall from the proof of Proposition \ref{equivalence expo} that $H_N$ has the same law as the matrix $A'$, where $A'$ is the $N \times N$ matrix such that
$$ A'_{i,j} = \frac{X_{i,j}}{\sqrt{N}}\Car_{|X_{i,j}|_{\infty} \leq (\log N)^d } + \frac{H_{i,j}}{\sqrt{N}}\Car_{|X_{i,j}|_{\infty} > (\log N)^d }.$$
Thus,
\begin{equation} \EE R_{1,2}(z) = \EE R(A')_{1,2}(z) \label{egalite espe},\end{equation}
where $R(A')$ denotes the resolvent of $A'$.
Using the resolvent equation we get, 
\begin{equation}N\left|\EE R(A')_{1,2}(z) -\EE R(X_N)_{1,2}(z) \right| \leq 
 N\left(\log N\right)^2\EE\left|\left|A'-X_N\right|\right|_{HS}\label{comparaison resolvante},\end{equation}
where $||.||_ {HS}$ denote the Hilbert-Schmidt norm. But it is easy to see that
$$\EE\left|\left|A'-X_N\right|\right|_{HS} = o\left( \frac{1}{N\left( \log N\right)^2} \right),$$
since we know from Lemma \ref{esperance queue distrib} that
$$\EE\left(\left|X_{i,j}\right|\Car_{\left|X_{i,j}\right|>\left(\log N\right)^d } \right) = O\left(e^{-\frac{\kappa}{2}\left(\log N\right)^{d\alpha} }\right),$$
with $\kappa$ as in \eqref{tail distrib} and $d\alpha>1$. 
Thus, the latter estimate, together with \eqref{comparaison resolvante} and \eqref{egalite espe}, yields,
$$N\left|\EE R_{1,2}\left(x+\frac{i}{\log N}\right) -\EE R(X_N)_{1,2} \left(x+\frac{i}{\log N}\right)\right| \underset{N\to +\infty}{\longrightarrow} 0,$$
uniformly in $x \in K$.
Using \eqref{estimation coeff non diag}, we get 
\begin{equation} \label{coeffoffdiag}N\EE R_{1,2}\left(x+\frac{i}{\log N}\right) \underset{N \to+\infty}{\longrightarrow} 0,\end{equation}
uniformly in $x\in K$.

By the same coupling argument as above, one can show that 
$$  \EE\left( \frac{1}{N} \tr R\left(X_N\right)\left(x+\frac{i}{\log N}\right) \right) - \EE\left( \frac{1}{N}\tr R\left(x+ \frac{i}{\log N}\right) \right)  \underset{N\to +\infty}{\longrightarrow } 0,$$
uniformly for $x$ in $K$. 

But according to \cite[Proposition 3.1]{Soshnikov}, we have also
$$\EE\left( \frac{1}{N}\tr R(X_N)(z) \right) = G_{\sigma_{sc}}(z) +O\left( \frac{1}{\left|\Im(z) \right|^6 N } \right),$$
uniformly on bounded subsets of $\CC\setminus \RR$. We deduce that,
\begin{equation}\label{conv tr pres axe R}\EE\left(\frac{1}{N} \tr R\left(x+\frac{i}{\log N}\right)\right) \underset{N\to +\infty}{\longrightarrow} G_{\sigma_{sc}}(x),\end{equation}
uniformly for $x$ in  $K$.
Thus, putting \eqref{conv tr pres axe R} , \eqref{coeffoffdiag} together with \eqref{isotropie}, we get
$$\sup_{||u||=||v||=1}\left|\left\langle u, \EE R\left(x+\frac{i}{\log N}\right)  v \right\rangle - \left \langle u, v \right \rangle G_{\sigma_{sc}}(x)\right| \underset{N\to +\infty}{\longrightarrow}0,$$
uniformly for $x$ in  $K$, which completes the proof.
\end{proof}

As a consequence of Proposition \ref{concentration equation}, and the isotropic property of Proposition \ref{conv unif resolvante}, with the control on the spectrum of $H_N$ proved in Proposition \ref{controlesp}, we get the following exponential equivalent for $M_N$.
\begin{Pro}\label{approxequavp}
Let $H_N$ be as in Theorem \ref{equivalence expo} and $C_N$ be a random Hermitian matrix independent of $H_N$. Let $k$ be the rank of $C_N$, $\theta_1, \theta_2,...,\theta_k$ the non-zero eigenvalues of $C_N$ in non-decreasing order, and $u_1,u_2,...,u_k$ orthonormal eigenvectors associated with these eigenvalues. 
We define for $x\notin Sp(H_N)$,
$$M_N(x) = I_k - \left(\theta_i \left\langle u_i, \left(x-H_N\right)^{-1} u_j \right \rangle \right)_{1\leq i,j \leq k},$$
and for all $x > 2$,$$M(x) = I_k - 
  \left(
     \raisebox{0.5\depth}{%
       \xymatrixcolsep{1ex}%
       \xymatrixrowsep{1ex}%
       \xymatrix{
         \theta_1G_{\sigma_{sc}}(x) \ar @{.}[ddddrrrr]& 0 \ar @{.}[rrr] \ar @{.}[dddrrr] &  & & 0  \ar @{.}[ddd]  \\
         0 \ar@{.}[ddd] \ar@{.}[dddrrr]& & & & \\
         &&&& \\
         &&&& 0 \\
        0 \ar@{.}[rrr] & & & 0 & \theta_k G_{\sigma_{sc}}(x)
       }%
     }
   \right).
$$
Let $\delta>0$ and $\rho>0$. For any compact subset $K$ of $(2+\delta, +\infty)$ and $t>0$, we have
$$\lim_{N\to +\infty} \frac{1}{N^{\alpha/2}} \log \PP\left(\left\{\sup_{ x\in K} \left|M_N(x) -M(x)\right|_{\infty} >t\right\}\cap W  \right) =-\infty,$$
with $$W  = \left\{ \rk(C_N) = r, \rho(C_N) \leq \rho, \lambda_{H_N} \leq 2+\delta  \right\}.$$
\end{Pro}

\begin{proof}
By triangular inequality, we have
\begin{align*}
&\PP\left(\left\{\sup_{ x \in K} \left|M_N(x) -M(x)\right|_{\infty} >t\right\}\cap W  \right)\\ &\leq \PP\left(\left\{\sup_{ x \in K} \left|\widetilde{M}_N(x) - \EE_{C_N}\left(\widetilde{M}_N(x)\right) \right|_{\infty} >t/2\right\} \cap V\right) \\
& + \PP\left( \left\{\sup_{x\in K } \left|\EE_{C_N}\left(\widetilde{M}_N(x)\right) - M(x)\right|_{\infty} >t/2\right\} \cap  V \right),
 \end{align*}
with $$V = \left\{ \rk\left(C_N\right) = r, \rho(C_N) \leq \rho \right\}. $$
From Theorem \ref{conv unif resolvante}, we know that 
$$\sup_{ x\in K} \Car_{V}  \left|\EE_{C_N}(\widetilde{M}_N(x)) - M(x)\right|_{\infty} \overset{L^{\infty}}{\underset{N\to +\infty}{\longrightarrow }} 0,$$
where the convergence takes place in the space of essentially  bounded functions.
Thus, for $N$ large enough,
$$\PP\left(\left\{\sup_{x\in K} \left|M_N(x) -M(x)\right|_{\infty} >t\right\}\cap W  \right) \leq  \PP\left(\left\{\sup_{x\in K} \left|\widetilde{M}_N(x) - \EE_ {C_N}\left(\widetilde{M }_N(x)\right) \right|_{\infty} >t/2\right\} \cap V \right),$$
which, applying Proposition \ref{concentration equation}, ends the proof.
\end{proof}
We are now ready to give the proof of Theorem \ref{approx equation au vp}.
\begin{proof}[Proof of Theorem \ref{approx equation au vp}]
Let $K$ be compact subset of $(2+\delta, +\infty)$. Assuming $W$ occurs, we see that for all $x$ in $K$, the matrices $M_N(x)$ and $M(x)$ have their spectral radii bounded by 
 $$1+ \rho\max\left(1, \frac{1}{d(2+\delta, K)}\right),$$
where $d(2+\delta, K)$ is the distance of $2+\delta$ from $K$. 
Therefore $M(x)$ and $M_N(x)$ remain in a compact set of $\M_r(\CC) $. As the determinant function is uniformly continuous on compact sets of $\M_r(\CC)$, Theorem \ref{approxequavp} yields the claim.

\end{proof}

\subsection{Exponential equivalence of the largest solutions of the eigenvalue equation and the limit equation.}\label{equiv expo zero}

We are interested here in finding simple exponentially good approximations of $(\lambda_{X_N})_{N\in \NN}$, which will allow us to derive a large deviation principle for $\lambda_{X_N}$. To this end, define for all $N\in \NN$ and $\ee>0$,
\begin{equation} \label{def mu ee}
\mu_{N, \ee} = 
\begin{cases}
G_{\sigma_{sc}}^{-1}\left(1/\lambda_{C^{\ee}}\right) & \text{if }  \lambda_{C^{\ee}} \geq 1, \\
2  & \text{if } \lambda_{C^{\ee}}< 1.
\end{cases}
\end{equation}
We will show in this section the following result.
\begin{The}\label{approxexpobonne} For all $t>0$
$$\lim_{\ee \to 0}\limsup_{N\to +\infty} \frac{1}{N^{\alpha/2}} \log \PP\left(\left |\lambda_{X_N} - \mu_{N, \ee}\right|>t \right) = -\infty.$$
In other words, $(\mu_{\ee, N})_{ N\in \NN, \ee>0}$ are exponentially good approximations of $(\lambda_{X_N})_ {N\in \NN}$ at the exponential scale $N^{\alpha/2}$.
\end{The}

Since we know from Theorem \ref{equivalence expo} that $(\lambda_{H_N+C^{\ee}})_{N\in \NN, \ee>0}$ are exponentially good approximations of $(\lambda_{X_N})_{N\in \NN}$, we only need to prove Theorem \ref{approxexpobonne} with $\lambda_{H_N+C^{\ee}}$ instead of $\lambda_{X_N}$. For sake of clarity, we will focus first on finding an exponential equivalent of $\lambda_{H_N + C_N}$ where $C_N$ is a general random Hermitian matrix independent of $H_N$, and then we will apply our result to the matrix $C^{\ee}$ to get Theorem \ref{approxexpobonne}.

 We know by Lemma \ref{equation vp}, that provided $\lambda_{H_N +C_N}$ is outside the spectrum of $H_N$,  it is the largest zero of $f_N$ defined for all $z \notin Sp\left(H_N\right)$ by
$$f_N(z) = \det\left(I_r - \left(\theta_i \left\langle u_i, \left(z-H_N\right)^{-1} u_j \right\rangle \right)_{1\leq i,j \leq k} \right),$$
with $k$ the rank of $C_N$, $\theta_1,\theta_2,...,\theta_k$ are the non-zero eigenvalues of $C_N$ in non-decreasing order and $u_1,u_2,...,u_k$ are orthonormal eigenvectors associated with those eigenvalues.
But from Theorem \ref{approx equation au vp}, we know that this function is arbitrary close to a certain limit function $f$ on every compact subset of $(2,+\infty)$ with an exponentially high probability, with $f$ defined for all $x\notin (-2,2)$ by
\begin{equation} \label{def f}f(x) = \prod_{i=1}^k \left(1-\theta_i G_{\sigma_{sc}}(x)\right).\end{equation}
 Therefore, one can hope that the largest zero of $f_N$, which is the top eigenvalue of $H_N + C_N$, is arbitrary close to the largest zero of $f$. But since 
$$\forall x\geq 2, \ G_{\sigma_{sc}}(x) = \frac{x-\sqrt{x^2-4}}{2} ,$$
(see \cite[p.10]{Guionnet} for the computation), we see that $G_{\sigma_{sc}}$ is decreasing on $[2,+\infty)$ taking its values in $(0,1]$, and that 
$$\forall x \in (0,1], \  G_{\sigma_{sc}}^{-1}(x) = x +\frac{1}{x}.$$
Thus, $f$ admits a zero only when $\theta_k>1$, in which case its largest zero is $G_{\sigma_{sc}}^{-1}(1/\theta_k)$, which is also equal to
$G_{\sigma_{sc}}^{-1}\left(1/\lambda_{C_N}\right) $.

\begin{Pro}\label{equivexpo gene} Let $H_N$ be as in Theorem \ref{equivalence expo}, and let $C_N$ be a random Hermitian matrix independent of $H_N$. Let $\delta>0$ and $l\geq 2+2\delta$.
For all $t>0$ and $r\in \NN$,
$$\bornexpo\PP\left(\left|\lambda_{H_N+C_N} - \mu_N\right|>t,  \mu_N \geq 2+2\delta, \lambda_{H_N+C_N} \leq l, C_N \in V_{r,l}\right)= -\infty,$$
where $$\mu_N = \begin{cases}
G_{\sigma_{sc}}^{-1}\left(1/\lambda_{C_N}\right) & \text{if }  \lambda_{C_N} \geq 1, \\
2  & \text{if } \lambda_{C_N}< 1.
\end{cases}
$$
and $$ V_{r,l}  = \left\{ C \in H_N(\CC) : \rk(C) = r, \rho(C) \leq 1/G_{\sigma_{sc}}(l) \right\}.$$

\end{Pro}

\begin{proof} 
 We start by reducing the problem to the case where $C_N$ has its top eigenvalue simple and bounded away from its last-but-one eigenvalue. Let  $u$ be an eigenvector associated with the largest eigenvalue of $C_N$. Let $\gamma>0$. We denote by $C_{N}^{(\gamma)}$ the matrix defined by,
$$ C_{N}^{(\gamma)} = C_N+  \gamma u u^*.$$
By definition, the largest eigenvalue of $C_N$ is bounded away from its last-but-one eigenvalue by $\gamma$. Provided that $\lambda_{C_N} \geq 1$, we define 
$$\mu_{N}^{(\gamma)} =
 G_{\sigma_{sc}}^{-1}\left(1/\lambda_{C_{N,\gamma}}\right) = G_{\sigma_{sc}}^{-1}\left(1/(\lambda_{C_{N}}+\gamma )\right).$$
Weyl's inequality (see Lemma \ref{Weyl}) yields,
$$ \left| \lambda_{H_N+C_{N}^{(\gamma)}} - \lambda_{H_N+C_N} \right| \leq \gamma.$$
As for all $x \in (0,1]$, $G_{\sigma_{sc}}^{-1}(x) = x+\frac{1}{x}$, easy computation yields
$$ \left| \mu_{N}^{(\gamma)} - \mu_N \right| \leq   2\gamma.$$
Thus, we see that it is sufficient to prove the statement in Proposition \ref{equivexpo gene} but with $V_{r,l}^{(\gamma)}$ instead of $V_{r,l}$, where
$$V_{r,l}^{(\gamma)}  =\left\{ C \in H_N(\CC) : \rk(C) = r, \ \rho(C) \leq 1/G_{\sigma_{sc}}(l),  \ \theta_r(C) - \theta_{r-1}(C) \geq \gamma\right\},$$
where $\theta_r(C)$, and $\theta_{r-1}(C)$ denote respectively the largest and the second largest eigenvalue of $C$.

We know from Theorem \ref{approx equation au vp} that the functions $f_N$ and $f$ are arbitrary close on any compact subset of $(2,+\infty)$, with exponentially high probability. Since we cannot make the error on the distance between $f_N$ and $f$ in Theorem  \ref{approx equation au vp} depend on $C_N$, we need now a kind of uniform continuity property of the largest zero of continuous functions belonging to a certain compact set, to get that their largest zeros are close with exponentially high probability. This is the object of the following lemma. 

\begin{Lem} \label{continuite zero}
Let $K'\subset K$ be two compact subsets of $\RR$, such that there is some open set $U$ such that $K' \subset U \subset K$. Let $\mathcal{K}$ a compact subset of $C(K)$, the space of continuous functions on $K$ taking real values. We assume that any $f \in \mathcal{K}$ admits at least one zero in $K$, its largest zero, z(f), lies in $K'$, and $f$ changes sign at $z(f)$. Then, for all $t>0$, there is some $s
>0$, such that for all $f\in \mathcal{K}$ and $g \in C(K)$, such that 
$$ || f -g || <t,$$
 $g$ admits at least one zero in $K$, and  its largest zero $z(g)$, satisfies
$$ \left| z(f) - z(g) \right| < s.$$ 
\end{Lem}
\begin{proof}
As an consequence of the intermediate values theorem, the function $\phi$, defined for all $g\in C(K)$ by,
$$ \phi(g) = \begin{cases}
z_{\max}(g) & \text{ if } g \text{ admits a zero in } K,\\
 \dagger & \text{ otherwise,}
\end{cases}$$
is continuous at each $f \in \mathcal{K}$. As the set $\mathcal{K}$ is compact, we get the claim.

\end{proof}

We come back now at the proof of Proposition \ref{equivexpo gene}. Observe that if $C_N \in V_{r,l}^{(\gamma)}$, then $\mu_N \leq l$. Let $K$ be a compact set such that there is an open set $U$ satisfying $[2+2\delta, l] \subset U \subset K \subset (2+\delta, +\infty)$. Note that the subset 
$$ \mathcal{K}^{(\gamma)} = \left\{ x \in K \mapsto \prod_{i=1}^r \left(1-\theta_i G_{\sigma_{sc}}(x) \right) : (\theta_1,...,\theta_r) \in \Theta_{\gamma}  \right\},$$
where $$ \Theta^{(\gamma)}  = \left\{(\theta_1,...,\theta_r) \in \RR^r: -\rho \leq \theta_1\leq ... \leq \theta_{r-1} \leq \theta_r - \gamma,  \ 1\leq \theta_r \leq \rho, \  G_{\sigma_{sc}}^{-1}(1/\theta_r) \in K'\right\},$$
is a compact subset of $C(K)$.
Applying Lemma \ref{continuite zero} with $K'=[2+2\delta, l]$ and $K$, we get for any $t>0$,  that there is $s>0$, such that
\begin{align*}
\PP&\left( \left|\lambda_{H_N+C_N} - \mu_N\right|>t,  \mu_N \in K',  \lambda_{H_N + C_N}\leq l,  C_N \in V_{r,l}^{(\gamma)}\right) \\
&\leq \PP\left( \left\{\sup_{x\in K} \left|f_N(x) - f(x) \right| >s\right\} \cap W \right)+\PP\left( \lambda_{H_N}>2+\delta \right),
\end{align*}
with $$ W= \left\{ \rk(C_N) = r, \ \rho(C_N) \leq 1/G_{\sigma_{sc}}(l),\ \lambda_{H_N} \leq 2+\delta \right\}.$$
By Theorem \ref{approx equation au vp} and Proposition \ref{controlesp}, we deduce that,
$$ \lim_{N\to +\infty} \frac{1}{N^{\alpha/2}} \log \PP\left( \left|\lambda_{H_N+C_N} - \mu_N\right|>t,  \mu_N \geq 2+2\delta,  \lambda_{H_N + C_N}\leq l,  C_N \in V_{r,l}^{(\gamma)}\right) = -\infty,$$
which ends the proof of Proposition \ref{equivexpo gene}.
\end{proof}

We are now ready to give the proof of Theorem \ref{approxexpobonne}.

\begin{proof}[Proof of Theorem \ref{approxexpobonne} ]
According to  Proposition \ref{pgdinf2}, we only need to prove that for $\delta>0$ small enough,
$$\lim_{\ee \to 0}\limsup_{N\to +\infty} \frac{1}{N^{\alpha/2}} \log \PP\left(\left |\lambda_{X_N} - \mu_{N,\ee }\right|>t, \lambda_{X_N} >2-\delta \right) = -\infty.$$
Taking $\delta < t/3$, we see that it is actually sufficient to show
\begin{equation} \label{goal} \lim_{\ee \to 0}\limsup_{N\to +\infty} \frac{1}{N^{\alpha/2}} \log \PP\left(\left |\lambda_{X_N} - \mu_{N,\ee}\right|>t, \mu_{N,\ee} \geq 2+2\delta \right) = -\infty.\end{equation}
Using Proposition \ref{equivexpo gene}, but with $C^{\ee}$ instead of $C_N$, we get for any $l\geq2+2\delta$, and $k\in \NN$,
$$ \bornesup\PP\left(\left|\lambda_{H_N+C^{\ee}} - \mu_{N,\ee}\right|>t,  \mu_{N,\ee} \geq 2+2\delta, \lambda_{H_N+C^{\ee}} \leq l, C^{\ee} \in V_{k,l}\right)= -\infty,$$
 where $\mu_{ N,\ee}$ is defined as in \eqref{def mu ee}, and  where $ V_{k,l}$ is defined in Proposition \ref{equivexpo gene}.
Let $V_{\leq r, l} = \cup_{k=0}^r V_{k,l}$. Since $V_{\leq r, l}$ is a finite union of the $V_{k,l}$'s, we get
$$ \bornesup\PP\left(\left|\lambda_{H_N+C^{\ee}} - \mu_{N,\ee}\right|>t,  \mu_{N,\ee} \geq 2+2\delta, \lambda_{H_N+C^{\ee}} \leq l, C^{\ee} \in V_{\leq r,l}\right)= -\infty.$$
As a consequence of Lemma \ref{tensionC} and Proposition \ref{rangC}, we deduce that for any $\ee>0$,
$$\lim_{r,l \to +\infty} \bornesup \PP\left(C^{\ee} \notin V_{\leq r, l} \right) = -\infty.$$
Thus,
$$ \bornesup\PP\left(\left|\lambda_{H_N+C^{\ee}} - \mu_{N,\ee}\right|>t,  \mu_{N,\ee} \geq 2+2\delta, \lambda_{H_N+C^{\ee}} \leq l\right)= -\infty.$$
Using the fact that according to Theorem \ref{equivalence expo}, $(\lambda_{H_N+C^{\ee}})_{N\in \NN,\ee>0}$ are exponentially good approximations of $(\lambda_{X_N})_{N\in \NN}$, we get,
$$ \lim_{\ee \to 0}\bornesup\PP\left(\left|\lambda_{X_N} - \mu_{N,\ee}\right|>t,  \mu_{N,\ee} \geq 2+2\delta, \lambda_{X_N} \leq l\right)= -\infty.$$
But $(\lambda_{X_N})_{N\in \NN}$ is exponentially tight according to Proposition \ref{tension expo vp max}, thus we can conclude that,
$$ \lim_{\ee \to 0}\bornesup\PP\left(\left|\lambda_{X_N} - \mu_{N,\ee}\right|>t,  \mu_{N,\ee} \geq 2+2\delta \right)= -\infty,$$
which ends the proof.
\end{proof}

\section{Large deviations principle for the largest eigenvalue of $X_N$} \label{LDP}

Our aim here is to prove for each $\ee>0$, a large deviations principle for $(\mu_{\ee, N})_{N\in \NN}$. Since $(\mu_{N, \ee})_{N\in \NN, \ee >0}$ are exponentially good approximations of the largest eigenvalue of $X_N$, we will get a large deviations principle for $(\lambda_{X_N})_{N\in\NN}$. 

For every $r\in \NN$, we define
$$ \mathcal{E}_r = \{A\in \cup_{n\geq 1}H_n(\CC) : \mathrm{Card}\{(i,j) : A_{i,j} \neq 0\} \leq r \}.$$
For any $n\in \NN$, let $\mathcal{S}_n$ be the symmetric group on the set $\{1,...,n\}$. We denote by $\mathcal{S}$, the group $\cup_{n\in \NN}\mathcal{S}_n$. We denote by $\widetilde{\mathcal{E}}_r$ the set of equivalence classes of $\mathcal{E}_r$ under the action of $\mathcal{S}$, which is defined by
$$ \forall \sigma \in \mathcal{S}, \forall  A \in \mathcal{E}_r, \ \sigma.A = M_{\sigma}^{-1}AM_{\sigma} = \left(A_{\sigma(i), \sigma(j)}\right)_{i,j},$$
where $M_{\sigma}$ denotes the permutation matrix associated with the permutation $\sigma$ i.e $M_{\sigma} =(\delta_{i,\sigma(j)})_{i,j}$.

Let $H_r(\CC)/\mathcal{S}_r$ be the set of equivalence classes of $H_r(\CC)$ under the action of the symmetric group $\mathcal{S}_r$. Note that any equivalence class of the action of $\mathcal{S}$ on $\mathcal{E}_r$ has a representative in $H_r(\CC)$. This defines an injective map from $\widetilde{\mathcal{E}}_r$ into $H_r(\CC)/  \mathcal{S}_r$. Identifying $\widetilde{\mathcal{E}}_r$ to a subset of $H_r(\CC) / \mathcal{S}_r$, we equip $\widetilde{\mathcal{E}}_r$ of the quotient topology of $H_r(\CC) / \mathcal{S}_r$.
This topology is metrizable by the distance $\tilde{d}$ given by
\begin{equation} \label{definition distance topo quotient}\forall \tilde{A}, \tilde{B} \in \widetilde{\mathcal{E}}_r, \ \tilde{d}\left(\tilde{A}, \tilde{B}\right) = \min_{\sigma, \sigma' \in \mathcal{S}}\max_{i,j} \left|B _{\sigma(i), \sigma(j)} - A_{\sigma'(i),\sigma'(j)}\right|,\end{equation}
where $A$ and $B$ are two representatives of $\tilde{A}$ and $\tilde{B}$ respectively.
Since the application which associates to a matrix of $H_r(\CC)$ its largest eigenvalue is continuous and is invariant by conjugation, we can define this application on $H_r(\CC)/ \mathcal{S}_r$ and it will still be continuous. Therefore, the application which associates to a matrix of $\widetilde{\mathcal{E}}_r$ its largest eigenvalue is continuous for the topology we defined above. 
This fact will be crucial later when we will apply a contraction principle to derive a large deviations principle for $(\mu_{\ee, N})_{\ee>0, N\in \NN}$.

Let $\ee>0$. Let $\PP^{\ee}_{N,r}$ be the law of $C^{\ee}$, with $C^{\ee}$ as in \eqref{cut}, conditioned on the event $\{C^{\ee} \in \mathcal{E}_r\}$, and $\widetilde{\PP}^{\ee}_{N,r}$ the push forward of $\PP^{\ee}_{N,r}$ by the projection $\pi : \mathcal{E}_r \to \widetilde{\mathcal{E}}_r$.

\begin{Pro} \label{LDP nb entree fixe}
Let $r\in \NN$ and $\ee>0$. Then $(\widetilde{\PP}^{\ee}_{N,r})_{N\in \NN}$ satisfies a large deviations principle with speed $N^{\alpha/2}$, and good rate function $I_{\ee, r}$ defined  for all $\tilde{A}\in \widetilde{\mathcal{E}}_r$ by,
\begin{equation} I_{\ee, r}\left(\tilde{A}\right) = 
\begin{cases}
b\sum_{i\geq 1}\left|A_{i,i}\right|^{\alpha} + \frac{a}{2}\sum_{i\neq j} \left|A_{i,j}\right|^{\alpha} & \text{ if } A \in \mathcal{D}_{\ee,r },\\
+\infty & \text{ otherwise,}
\end{cases}
\label{fntaux}\end{equation} 
where $A$ is a representative of the equivalence class $\tilde{A}$ and 
$$\mathcal{D}_{\ee,r} =\left \{ A \in \mathcal{E}_r : \forall i\leq j, \ A_{i,j} = 0 \text{ or } \ \ee\leq \left|A_{i,j}\right| \leq \ee^{-1}, \text{ and } A_{i,j}/|A_{i,j}|  \in \mathrm{supp}(\nu_{i,j}) \right\},$$ 
with $\nu_{i,j} = \nu_1$ if $i=j$, and $\nu_{i,j} = \nu_{2}$ if $i<j$, where $\nu_1$ and $\nu_2$ are defined in \ref{hypo1}.
\end{Pro}

We recall here a Lemma from \cite{Bordenave}[p.2478], which will be very useful in the proof of Proposition \ref{LDP nb entree fixe}.
\begin{Lem}\label{minoration proba boules}
For all $\gamma>0$, and all $x\neq 0$ with $x/|x| \in \mathrm{supp}(\nu_1)$, there is a sequence $(b_{N})_{N\in \NN}$ which converges to $b$, such that for $N$ large enough,
$$\PP\left(X_{1,1}/\sqrt{N} \in [x-\gamma, x+\gamma]\right) \geq e^{-b_N|x|^{\alpha}N^{\alpha/2}}.$$
Similarly, for all $z\neq 0$ such that $z/|z| \in \mathrm{supp}(\nu_2)$, and all $0<\gamma<|z|$, there is a sequence $(a_N)_{N\in \NN}$ which converges to $a$, such that for $N$ large enough,
$$\PP\left(X_{1,2}/\sqrt{N} \in B_{\CC}(z, \gamma)\right) \geq e^{-a_N|z|^{\alpha}N^{\alpha/2}}.$$
\end{Lem}

\begin{proof}[Proof of Proposition \ref{LDP nb entree fixe}]

\textbf{Property of the rate function: } The function $\phi$ defined on $H_r(\CC)$ by,
$$\phi(A) =  b\sum_{i=1}^r \left|A_{i,i}\right|^{\alpha} + \frac{a}{2} \sum_{1\leq i \neq j \leq r} \left|A_{i,j} \right|^{\alpha},$$
has compact level sets. Thus, we can deduce, by definition of the topology we equipped $\widetilde{\mathcal{E}}_r$, that the rate function $I_{\ee,r}$ has also compact level sets. 

\textbf{Exponential tightness: }

Let $\gamma>0$. We define,
$$K_{\gamma} = \left\{\tilde{A} \in\widetilde{\mathcal{E}}_r : \sum_{i,j\in \NN} \left|A_{i,j}\right|^{\alpha} \leq \gamma  \right\},$$
where $A$ denotes a representative of $\tilde{A}$. Since the set
$$ \left\{ A\in H_r(\CC) : \sum_{1\leq i , j\leq r} \left|A_{i,j}\right|^{\alpha} \leq \gamma \right\},$$
is a compact subset of $H_r(\CC)$ and invariant under the action of $\mathcal{S}_r$, we can deduce, by the choice of the topology we equipped $\widetilde{\mathcal{E}}_r$, that $\widetilde{K}_{\gamma}$ is a compact subset of $\widetilde{\mathcal{E}}_r$.
Then, by definition of $\widetilde{\PP}^{\ee}_{N,r}$, we have
\begin{equation} \label{tension exponentielle matrice sparse} \widetilde{\PP}^{\ee}_{N,r}\left(K_{\gamma}^c\right) = \PP\left(\sum_{1\leq i,j \leq N} \left|C_{i,j}^{\ee}\right|^{\alpha} > \gamma \  \arrowvert  \ C^{\ee} \in \mathcal{E}_r  \right).\end{equation}
But $\Car_{\sum_{i,j} |C_{i,j}^{\ee}| >\gamma}$ and $\Car_{C^{\ee} \in \mathcal{E}_r}$  are respectively nondecreasing and nonincreasing with respect to the absolute value of each entry of $C^{\ee}$. Therefore, Harris' inequality yields,
\begin{align*}
\PP\left(\sum_{1\leq i,j \leq N} \left|C_{i,j}^{\ee}\right|^{\alpha} > \gamma \  \arrowvert  \ C^{\ee} \in \mathcal{E}_r  \right)& \leq \PP\left( \sum_{1\leq i,j\leq N} \left|C_{i,j}^{\ee}\right|^{\alpha} >\gamma\right)\\
& \leq \PP\left(\sum_{i=1}^N |C^{\ee}_{i,j}|^{\alpha} >\gamma/2\right)+\PP\left(\sum_{1\leq i\neq j \leq N} |C^{\ee}_{i,j}|^{\alpha} >\gamma/2\right).
\end{align*}
Now choose $a_1$ such that $0<2a_1<a$, and $b_1$ such that $0<b_1<b$. By Chernoff's inequality we have,
\begin{align}  \widetilde{\PP}^{\ee}_{N,r}\left(K_{\gamma}^c\right)& \leq e^{-b_1N^{\alpha/2}\gamma/2}\EE\left(e^{b_1|X_{1,1}|^{\alpha}\Car_{\ee N^{1/2}\leq |X_{1,1}|\leq \ee^{-1}N^{1/2}}} \right)^N\nonumber \\
&+e^{-a_1N^{\alpha/2}\gamma/2}\EE\left(e^{2a_1|X_{1,2}|^{\alpha}\Car_{\ee N^{1/2}\leq|X_{1,2}|_{\infty} \leq\ee^{-1}N^{1/2}}} \right)^{N(N-1)/2}.\label{tension exponentielle matrice sparse 2}\end{align}
Let $b_2 \in (b_1, b)$. For $t$ large enough we have,
$$\PP\left(|X_{1,1}|>t\right) \leq e^{-b_2 t^{\alpha}}.$$
Thus, integrating by part just as in the proof of Lemma \ref{tension rayon spectral B} we get, for $N$ large enough, 
\begin{align}
\EE\left(e^{b_1 |X_{1,1}|^{\alpha}\Car_{\ee N^{1/2}\leq |X_{1,1}|\leq\ee^{-1}N^{1/2}}}\right)\leq\exp\left( \frac{b_2}{b_2-b_1}e^{-(b_2-b_1)\ee^{\alpha}N^{\alpha/2}}\right). \label{estimation transfo laplace 1}
\end{align}
Similarly,  for $N$ large enough and with $a_2$ such that $2a_2 \in (2a_1, a)$ we have,
\begin{equation} \label{estimation transfo laplace 2}\EE\left(e^{2a_1 |X_{1,2}|^{\alpha}\Car_{\ee N^{1/2}\leq |X_{1,2}|_{\infty}\leq \ee^{-1}N^{1/2}}}\right) \leq \exp\left( \frac{a_2}{a_2-a_1}e^{-2(a_2-a_1)\ee^{\alpha}N^{\alpha/2}}\right).\end{equation}
Therefore, putting together \eqref{estimation transfo laplace 1} and \eqref{estimation transfo laplace 2} into \eqref{tension exponentielle matrice sparse 2}, we get,
$$ \bornesup \widetilde{\PP}^{\ee}_{N,r}\left(\widetilde{K}_{\gamma}^c\right)\leq - \frac{\gamma}{2} a_1 \vee b_1,$$
which proves that $(\widetilde{\PP}_{N,r}^{\ee})_{N\in \NN}$ is exponentially tight.

\textbf{Lower bound:} Let $A \in H_r(\CC)$. Without loss of generality, we can assume that $I_{\ee,r}(\tilde{A}) < +\infty$, that is $A\in \mathcal{D}_{\ee,r }$. Moreover, we make the assumption that for all $1\leq  i,j \leq r$,
$$ A_{i,j} =0 \text{ or } \ee < |A_{i,j}| < \ee^{-1 }.$$
 Let $\delta>0$ be such that 
 $$\delta<  \min\left(\min_{A_{i,j} \neq 0}|A_{i,j}| - \ee, \ee^{-1} - \max_{1\leq i,j \leq r}\left|A_{i,j}\right|, \ee \right).$$
Let
$$\tilde{B}\left(\tilde{A}, \delta\right)  = \left\{ \tilde{X} \in \widetilde{\mathcal{E}}_r : \tilde{d}\left(\tilde{A},\tilde{X}\right) < \delta \right\},$$
with $\tilde{d}$ being the distance defined in \eqref{definition distance topo quotient}. We have
$$ \widetilde{\PP}^{\ee}_{N,r}\left( \tilde{B}\left(\tilde{A}, \delta\right) \right) = \PP\left(\min_{\sigma \in \mathcal{S}} \max_{i,j} \left|C_{\sigma(i), \sigma(j)}^{\ee}-A_{i,j}\right|<\delta \ | \ C^{\ee} \in \mathcal{E}_r\right).$$
Let 
$$B_{\infty, N}\left(A, \delta\right) = \left\{X\in H_N\left(\CC\right) : \max_{1\leq i,j \leq N} \left| X_{i,j} - A_{i,j}\right| < \delta \right\}.$$
Since $\delta< \ee$, and since all the non-zero entries of $C^{\ee}$ are in $ \{z \in \CC : \ee \leq |z| \leq \ee^{-1}\}$, we see that if $C^{\ee} \in B_{\infty, N}(A,\delta)$, then $C^{\ee} \in \mathcal{E}_r$. Thus, 
 \begin{align}\widetilde{\PP}^{\ee}_{N,r}\left( \tilde{B}\left(\tilde{A}, \delta\right) \right)&\geq 
\PP\left(C^{\ee} \in B_{\infty, N}\left(A, \delta\right) | C^{\ee} \in \mathcal{E}_r\right)\nonumber \\
& = \frac{1}{\PP\left(C^{\ee} \in \mathcal{E}_r\right)}\PP\left(C^{\ee} \in B_{\infty, N}\left(A, \delta\right)\right). \label{minoration proba boule} \end{align}
But by independence, we have
\begin{equation} \label{proba finale}
\PP\left(C^{\ee}\in B_{\infty, N}\left(A, \delta\right)\right)= \prod_{i=1}^N \PP\left(|C^{\ee}_{i,i}-A_{i,i}|<\delta\right) \prod_{i<j} \PP\left(|C^{\ee}_{i,j}-A_{i,j}|<\delta\right).
\end{equation}
Since 
 $$\delta<  \min\left(\min_{A_{i,j} \neq 0}|A_{i,j}| - \ee, \ee^{-1 } - \max_{1\leq i,j \leq r}\left|A_{i,j}\right| \right),$$
we have
$$\PP\left(|C^{\ee}_{i,i}-A_{i,i}|<\delta\right)\geq \PP\left( \left|\frac{X_{i,i}}{\sqrt{N}}-A_{i,i}\right|<\delta\right)\Car_{A_{i,i}\neq0} + \PP\left(C^{\ee}_{i,i} = 0\right) \Car_{A_{i,i} = 0}.$$
Thus, according to Lemma \ref{minoration proba boules}, there is a sequence $(b_N)_{N\in \NN}$ converging to $b$ such that,
\begin{align*} 
\PP\left(|C^{\ee}_{i,i}-A_{i,i}|<\delta\right)& \geq e^{-b_N|A_{i,i}|^{\alpha} N^{\alpha/2}} \Car_{A_{i,i}\neq 0} + \left(1-\PP\left(|C^{\ee}_{i,i}|\neq 0\right)\right)\Car_{A_{i,i}=0} \\
&\geq e^{-b_N|A_{i,i}|^{\alpha} N^{\alpha/2}} \Car_{A_{i,i}\neq 0} + \left(1-\PP\left(|X_{i,i}|\geq \ee N^{1/2}\right)\right)\Car_{A_{i,i}=0}.
\end{align*}
For $N$ large enough we get, with $\kappa$ defined in \eqref{notation}, we get
\begin{align}
\PP\left(|C^{\ee}_{i,i}-A_{i,i}|<\delta\right)& \geq  e^{-b_N|A_{i,i}|^{\alpha} N^{\alpha/2}} \Car_{A_{i,i}\neq 0} +\left (1-e^{-\kappa\ee^{\alpha}N^{\alpha/2}}\right)\Car_{A_{i,i}=0} \nonumber\\
&\geq e^{-b_N|A_{i,i}|^{\alpha} N^{\alpha/2}} \left(1-e^{-\kappa\ee^{\alpha}N^{\alpha/2}}\right). \label{estimation 1}
\end{align}
Similarly for  $i\neq j $, we have,
\begin{equation} \label{estimation 2}\PP\left(|C^{\ee}_{i,j}-A_{i,j}|<\delta\right)\geq e^{-a_N|A_{i,j}|^{\alpha} N^{\alpha/2}} \left(1-e^{-\kappa\ee^{\alpha}N^{\alpha/2}}\right),\end{equation}
where $(a_N)_{N\in \NN}$ is a sequence converging to $a$.
Putting \eqref{estimation 1} and \eqref{estimation 2} into \eqref{proba finale}, we get,
$$\PP\left(C^{\ee}\in B_{\infty, N}\left(A, \delta\right)\right) \geq  e^{-b_N\sum_{i\geq1} |A_{i,i}|^{\alpha} N^{\alpha/2}}  e^{-a_N\sum_{i< j}|A_{i,j}|^{\alpha} N^{\alpha/2}} \left(1-e^{-\kappa\ee^{\alpha}N^{\alpha/2}}\right)^{N^2}.$$
Hence at the exponential scale,
$$\liminf_{N\to +\infty} \frac{1}{N^{\alpha/2}} \log \PP\left(C^{\ee} \in B_{\infty, N}\left(A, \delta\right)\right) \geq- b\sum_{i\geq1} |A_{i,i}|^{\alpha} - a\sum_{i<j} |A_{i,j}|^{\alpha}.$$
Besides by Proposition \ref{entrees non nulles C} and Borel-Cantelli Lemma, we have 
$$ \PP\left(C^{\ee} \in \mathcal{E}_r\right) \underset{N\to +\infty}{\longrightarrow} 1.$$
Putting these estimates into \eqref{minoration proba boule}, we get
\begin{equation} \label{borneinfrestr} \borneinf \widetilde{\PP}^{\ee}_{N,r}\left( \tilde{B}\left(\tilde{A}, \delta\right) \right) \geq - b\sum_{i=1}^r |A_{i,i}|^{\alpha} - a\sum_{1\leq i<j \leq r}  |A_{i,j}|^{\alpha}.\end{equation}
Observe that since the rate function $I_{\ee,r}$ is continuous on its domain $\pi\left(\mathcal{D}_{\ee,r}\right)$, we have also the bound \eqref{borneinfrestr} for any $A\in \mathcal{D}_{\ee,r}$. This concludes the proof of the lower bound.

\textbf{Upper bound: }
From our assumption \ref{hypo1}, we deduce that for $N$ large enough, the support of $\widetilde{\PP}^{\ee}_{N,r}$ is included in the domain of $I_{\ee,r}$, that is $\pi\left(\mathcal{D}_{\ee,r}\right)$. Thus, we see that whenever $I_{\ee,r}(\tilde{A}) = +\infty$ for $\tilde{A} \in \widetilde{\mathcal{E}}_r$, 
$$\lim_{\delta \to 0} \bornesup  \widetilde{\PP}^{\ee}_{N,r}\left( \tilde{B}\left(\tilde{A}, \delta\right) \right) = -\infty.$$
Let $A \in H_r(\CC)$ be such that $A\in \mathcal{D}_{\ee,r}$.
Since $X\in H_r(\CC) \mapsto \sum_{i=1}^r |X_{i,i}|^{\alpha}$ and $X\in H_r(\CC) \mapsto \sum_{1\leq i\neq j\leq r}|X_{i,j}|^{\alpha}$ are continuous, then by definition of the topology we equipped $\widetilde{\mathcal{E}}_r$ the functions
$\tilde{X}\in \widetilde{\mathcal{E}}_r  \mapsto \sum_{i\geq 1}|X_{i,i}|^{\alpha}$ and   $\tilde{X}\in \widetilde{\mathcal{E}}_r  \mapsto \sum_{i \neq j}|X_{i,j}|^{\alpha}$ are continuous. Then, we can find a nonnegative function $h$, such that $h(\delta) \to 0$ as $\delta \to 0$, and such that
$$\tilde{P}_{N,r}^{\ee}\left(\tilde{B}\left(\tilde{A}, \delta\right) \right)\leq \PP\left(\sum_{i \geq 1} |C^{\ee}_{i,i}|^{\alpha} \geq \sum_{i \geq 1} |A_{i,i}|^{\alpha}-h\left(\delta\right), \sum_{i \neq j} |C^{\ee}_{i,j}|^{\alpha} \geq \sum_{i \neq j} |A_{i,j}|^{\alpha}-h\left(\delta\right) | \ C^{\ee} \in \mathcal{E}_r \right).$$
But the sets
$$\left\{\sum_{i \geq 1} |C^{\ee}_{i,i}|^{\alpha} \geq \sum_{i \geq 1} |A_{i,i}|^{\alpha}-h\left(\delta\right)\right\} \text{ and } \left\{\sum_{i \neq j} |C^{\ee}_{i,j}|^{\alpha} \geq \sum_{i \geq 1} |A_{i,j}|^{\alpha}-h\left(\delta\right) \right\},$$
 are nondecreasing with respect to the absolute value of each entry of $C^{\ee}$, and $\{C^{\ee} \in \mathcal{E}_r \}$ is nonincreasing with respect to the absolute value of each entry of $C^{\ee}$. Using Harris' inequality and the independence of the entries,
\begin{align}
\tilde{P}^{\ee}_{N,r}\left(\tilde{B}\left(\tilde{A}, \delta\right) \right)&\leq \PP\left(\sum_{i \geq 1} |C^{\ee}_{i,i}|^{\alpha} \geq \sum_{i \geq 1} |A_{i,i}|^{\alpha}-h\left(\delta\right), \sum_{i \neq j} |C^{\ee}_{i,j}|^{\alpha} \geq \sum_{i \neq j} |A_{i,j}|^{\alpha}-h\left(\delta\right)\right)\nonumber\\
& = \PP\left(\sum_{i \geq 1} |C^{\ee}_{i,i}|^{\alpha} \geq \sum_{i \geq 1} |A_{i,i}|^{\alpha}-h\left(\delta\right) \right) \PP\left(\sum_{i \neq j} |C^{\ee}_{i,j}|^{\alpha} \geq \sum_{i \neq j} |A_{i,j}|^{\alpha}-h\left(\delta\right)\right).\label{majoration proba boule}
\end{align}
Let $N\geq r$. By Chernoff's inequality we get, with $0< b_1<b$,
 $$\PP\left( \sum_{i=1}^N|C_{i,i}|^{\alpha}\geq  \sum_{i=1}^N|A_{i,i}|^{\alpha} +h(\delta)\right)\leq e^{-N^{\alpha/2}b_1\left(\sum_{i=1}^N|A_{i,i}|^{\alpha}+ h(\delta)\right)}\EE\left(e^{b_1 |X_{1,1}|^{\alpha}\Car_{\ee N^{1/2}\leq |X_{1,1}|\leq \ee^{-1}N^{1/2}}}\right)^N.$$
But we know from \eqref{estimation transfo laplace 1} that for any $b_2 \in (b_1,b)$ and $N$ large enough,
$$\EE\left(e^{b_1|X_{1,1}|^{\alpha}\Car_{\ee N^{1/2}\leq |X_{1,1}|\leq \ee^{-1}N^{1/2}} }\right) \leq\exp\left( \frac{b_2}{b_2-b_1}e^{-(b_2-b_1)\ee^{\alpha}N^{\alpha/2}}\right).$$
Hence,
$$\limsup_{N\to +\infty} \frac{1}{N^{\alpha/2}} \log \PP\left( \sum_{i=1}^N|C_{i,i}|^{\alpha} \geq \sum_{i=1}^N|A_{i,i}|^{\alpha} +h(\delta))\right) \leq -b_1\sum_{i\geq1}|A_{i,i}|^{\alpha}+ h(\delta).$$
As this inequality is true for all $b_1<b$, letting $b_1$ go to $b$, we get,
$$\limsup_{N\to +\infty} \frac{1}{N^{\alpha/2}} \log \PP\left( \sum_{i=1}^N|C_{i,i}|^{\alpha} \geq \sum_{i= 1}^N |A_{i,i}|^{\alpha} +h(\delta)\right) \leq -b\left(\sum_{i\geq1}|A_{i,i}|^{\alpha}+ h(\delta)\right).$$
Similarly one can show,
$$\limsup_{N\to +\infty} \frac{1}{N^{\alpha/2}} \log \PP\left( \sum_{i\neq j}|C_{i,j}|^{\alpha} \geq \sum_{i\neq j}|A_{i,j}|^{\alpha} +h(\delta)\right) \leq -\frac{a}{2}\left(\sum_{i\neq j}|A_{i,j}|^{\alpha}+ h(\delta)\right).$$
Putting these two last estimates into \eqref{majoration proba boule}, we get
$$\limsup_{\delta\to 0} \limsup_{N \to +\infty} \frac{1}{N^{\alpha/2}} \log \tilde{P}^{\ee}_{N,r}\left(\tilde{B}\left(\tilde{A}, \delta\right) \right) \leq  -b\sum_{i\geq1}|A_{i,i}|^{\alpha} -\frac{a}{2}\sum_{i\neq j}|A_{i,j}|^{\alpha}.$$
\end{proof}

The idea now, is to use the fact that $C^{\ee}$ has with exponentially large probability at most $r$ non-zero entries, by Proposition \ref{entrees non nulles C}, to release the conditioning on the event $\{C^{\ee} \in \mathcal{E}_r\}$. Then, as the largest eigenvalue map is continuous on $\widetilde{\mathcal{E}}_r$, the contraction principle will give us a LDP for $(\mu_{\ee,N})_{ N \in \NN}$.

\begin{Pro}\label{pdgmuN}Recall that for any $N\in \NN$ and $\ee>0$, we define
\begin{equation*}
\mu_{\ee,N} = 
\begin{cases}
G_{\sigma_{sc}}^{-1}\left( 1/\lambda_{C^{\ee}}\right) & \text{ if } \lambda_{C^{\ee}} \geq 1\\
2   & \text{ otherwise,}
\end{cases}
\end{equation*} 
where $\lambda_{C^{\ee}}$ denotes the largest eigenvalue of $C^{\ee}$, and $C^{\ee}$ is as in \eqref{cut}.

For all $\ee>0$, $(\mu_{\ee,N})_{N\in \NN}$ follows a large deviations principle with speed $N^{\alpha/2}$, and good rate function $J_{\ee}$, defined  by
$$J_{\ee}(x) =
\begin{cases}
\inf\{ I_{\ee}(A) : A \in \cup_{n\geq 1} H_n(\CC), \ \lambda_A  = 1/G_{\sigma_{sc}}(x) \} & \text{if } x\geq 2,\\
0 & \text{ if } x=2,\\
+\infty & \text{if } x<2,
\end{cases}
$$
where $\lambda_A$ denotes the largest eigenvalue of any Hermitian matrix $A$ and 
 $$I_{\ee}(A) = 
\begin{cases}
b\sum_{i\geq 1}\left|A_{i,i}\right|^{\alpha} + a \sum_{i<j} \left|A_{i,j}\right|^{\alpha} & \text{ if } A \in \mathcal{D}_{\ee} , \\
+\infty & \text{otherwise.}
\end{cases}
$$ 
with
$$\mathcal{D}_{\ee} =\left \{ A \in \cup_{n\in \NN} H_n(\CC) : \forall i\leq j, \ A_{i,j} = 0 \text{ or } \ \ee\leq \left|A_{i,j}\right| \leq \ee^{-1}, \text{ and } A_{i,j}/|A_{i,j}|  \in \mathrm{supp}(\nu_{i,j}) \right\},$$ 
with $\nu_{i,j} = \nu_1$ if $i=j$, and $\nu_{i,j} = \nu_{2}$ if $i<j$, where $\nu_1$ and $\nu_2$ are defined in \ref{hypo1}.

\end{Pro}

\begin{proof}Note that by Lemma \ref{tensionC}, we already know that $(\mu_{\ee, N})_{N\in \NN}$ is exponentially tight. Therefore, we only need to show that $(\mu_{\ee, N})_{N\in \NN}$ satisfies a weak LDP.
Let $f : \cup_{n\geq 1}H_n(\CC) \to \RR$ be defined by,
$$ f(A) = \begin{cases}
G_{\sigma_{sc}}^{-1}\left(1/\lambda_{A} \right) &  \text{ if } \lambda_A \geq 1,\\
2 & \text{ otherwise.}
\end{cases}
$$ 
 Since the largest eigenvalue of a Hermitian matrix is invariant by conjugation, $f$ can be defined on $\widetilde{\mathcal{E}}_r$ for any $r\in\NN$. Because of the topology we put on $\widetilde{\mathcal{E}}_r$, $f$ is continuous on $\widetilde{\mathcal{E}}_r$. Therefore, by the contraction principle (see \cite[p.126]{Zeitouni}),  the push-forward of $\widetilde{\PP}^{\ee}_{N,r}$ by $f$, denoted $\widetilde{\PP}^{\ee}_{N,r} \circ f^{-1 }$, satisfies a LDP with speed $N^{\alpha/2}$, and good rate function $J_{\ee, r}$, defined for any $x\in \RR$ by
$$J_{\ee, r}(x) =
\inf \left\{  I_{\ee, r}(\tilde{A}) : f(\tilde{A}) = x, \tilde{A} \in \widetilde{\mathcal{E}}_r\right\},$$
where $I_{\ee, r}$ is as in \eqref{fntaux}. Since $G_{\sigma_{sc}}^{-1}(x) \geq 2$, for all $x\in (0,1]$, we can re-write this rate function as,
$$J_{\ee, r}(x) =\begin{cases}
\inf \left\{ I_{\ee}(A) : \lambda_A = 1/G_{\sigma_{sc}}(x), A \in \mathcal{D}_{\ee} \right\}& \text{ if }  x> 2\\
0 & \text{ if } x=2\\
+\infty &  \text{ if } x<2,
\end{cases}$$
where $I_{\ee}$ and $\mathcal{D}_{\ee}$ are defined in Proposition \ref{pdgmuN}. 
Observe that $\widetilde{\PP}^{\ee}_{N,r} \circ f^{-1}$ is in fact the law of $\mu_{\ee, N}$ conditioned on the event $\{C^{\ee} \in \mathcal{E}_r \}$. We will show that $(\widetilde{\PP}^{\ee}_{N,r}\circ f^{-1})_{N, r \in \NN}$ are exponentially good approximations of $(\mu_{\ee,N})_{N\in \NN}$. Let $\nu_{r,N}$ be an independent random variable with the same law as of $\mu_{\ee,N}$ conditioned on the event $\{C^{\ee} \in \mathcal{E}_r\}$. Define
$$ \tilde{\nu}_{r, N} = \mu_{\ee,N}\Car_{C^{\ee} \in \mathcal{E}_r} + \nu_{r,N}\Car_{C^{\ee} \notin \mathcal{E}_r}.$$
Then,  $\tilde{\nu}_{r, N}$ and $\nu_{r,N}$ have the same law $\widetilde{\PP}^{\ee}_{N,r} \circ f^{-1}$. Let $\delta>0$. We have
$$ \PP\left( \left| \tilde{\nu}_{r, N} -   \mu_{\ee,N}\right| >\delta \right) \leq \PP\left(C^{\ee} \notin \mathcal{E}_r\right).$$
By Proposition \ref{entrees non nulles C}, we get 
$$ \lim_{r\to +\infty} \bornesup \PP\left( \left| \tilde{\nu}_{r, N} -   \mu_{\ee,N}\right| >\delta \right)=-\infty.$$
We can apply \cite[Theorem 4.2.16]{Zeitouni} and deduce that $(\mu_{\ee, N})_{N\in \NN}$ satisfies a weak LDP with speed $N^{\alpha/2}$, and rate function defined for all $x\in \RR$ by
$$ \Psi_{\ee}(x)=  \sup_{\delta>0} \liminf_{r \to +\infty} \inf_{|x-y|<\delta} J_{\ee,r}(y).$$
But $J_{\ee, r}$ is nonincreasing in $r$. Thus,
$$ \Psi_{\ee}(x)=  \sup_{\delta>0}  \inf_{r >0} \inf_{|x-y|<\delta} J_{\ee,r}(y) = \sup_{\delta>0}  \inf_{|x-y|<\delta} \inf_{r >0}  J_{\ee,r}(y) =  \sup_{\delta>0}  \inf_{|x-y|<\delta} J_{\ee}(y),$$
where $J_{\ee}$ is defined in Proposition \ref{pdgmuN}. To conclude that $\Psi_{\ee} = J_{\ee}$, we need to show that $J_{\ee}$ is lower semicontinuous. We will in fact show that $J_{\ee}$ has compact level sets.  
Let $\tau >0$ and $x\in\RR$. If 
$$\inf\{I_{\ee}\left(A\right) :f(A) =x, A \in \cup_{n\in \NN} H_n(\CC)\} \leq \tau,$$
where $I_{\ee}$ is defined in Proposition \ref{pdgmuN}, then
$$\inf\{I_{\ee}\left(A\right) :f(A) =x\}  = 
\inf\{I_{\ee}\left(A\right) : f(A) =x, \ I_{\ee}\left(A\right) \leq 2\tau\}$$
But if $A\in \cup_{n\geq 1} H_n(\CC)$ is such that $I_{\ee}(A) \leq 2\tau$, then
$$b\wedge \frac{a}{2} \sum_{i,j}\ee^{\alpha} \Car_{A_{i,j} \neq 0} \leq I_{\ee}(A) \leq  2\tau.$$
 Let $r \geq \frac{2\tau}{\ee^{\alpha}(b\wedge a/2) }$. We deduce by the above observation that, 
$$ \inf\{I_{\ee}\left(A\right) :f(A) =x\}  = \inf\{I_{\ee}\left(A\right) : f(A) =x, I_{\ee}\left(A\right) \leq 2\tau, A \in \mathcal{E}_r\}.$$
Therefore,
$$\inf\left\{I_{\ee}\left(A\right) :f(A) =x \right\} =  \inf\left\{I_{\ee}\left(A\right) : f(A) =x, A \in \mathcal{E}_r\right\}.$$
Thus,
$$\inf\left\{I_{\ee}\left(A\right) :f(A) =x \right\} = \inf\left\{I_{\ee,r}\left(\tilde{A}\right):  f(\tilde{A}) =x, \tilde{A} \in \tilde{\mathcal{E}}_r\right\},$$
with $I_{\ee,r}$  being defined in Proposition \ref{LDP nb entree fixe}.
Since $I_{\ee,r}$ is a good rate function and $f$ is continuous on $\tilde{\mathcal{E}_r}$, we have
$$\left\{x \in \RR : J_{\ee}\left(x\right) \leq \tau  \right\} =  \left\{ f(\tilde{A}) : I_{\ee,r}\left(\tilde{A}\right)\leq \tau\right \}.$$
Thus, the $\tau $-level set of $J_{\ee}$ is compact, which concludes the proof.

\end{proof}

\begin{The}\label{pgdXn}
The sequence $(\lambda_{X_N})_{N\in \NN}$ follows a LDP with speed $N^{\alpha/2}$, and good rate function defined by,
$$J(x) = 
\begin{cases}
cG_{\sigma_{sc}}(x)^{-\alpha} & \text{if } x> 2,\\
0 & \text{if } x =2,\\
+\infty & \text{if } x< 2,\\
\end{cases}
$$
where  $$c = \inf \left \{ I(A) : \lambda_A = 1, A \in \mathcal{D}\right\},$$
where $I$ is defined for any $A\in \cup_{n\geq 1} H_n(\CC)$, by
$$ I\left( A\right) = b \sum_{i=1}^{+\infty} \left|A_{i,i} \right|^{\alpha} + a \sum_{i<j} \left| A_{i,j} \right|^{\alpha},$$
and 
$$\mathcal{D} = \left\{ A \in \cup_{n\geq 1} H_n(\CC) : \forall i\leq j , \ A_{i,j} = 0 \text{ or } \frac{A_{i,j}}{|A_{i,j}|} \in \mathrm{supp}(\nu_{i,j}) \right\},$$
where $\nu_{i,j}= \nu_1$ if $i=j$, and $\nu_2$ if $i< j$, and where $\mathrm{supp}(\nu_{i,j})$ denotes the support of the measure $\nu_{i,j}$. 
\end{The}

\begin{proof}

We already know by Proposition \ref{tension expo vp max}  that $(\lambda_{X_N})_{N\in \NN}$ is exponentially tight. Thus, it is sufficient to prove that $(\lambda_{X_N})_{N\in \NN}$ satisfies a weak LDP.
Since we know from Theorem \ref{approxexpobonne} that $(\mu_{\ee, N})_{ N \in \NN,\ee
>0}$ are exponentially good approximations of $(\lambda_{X_N})_{N\in \NN}$, and that for each $\ee>0$, $(\mu_{\ee, N})_{N\in \NN}$ follows a LDP with rate function $J_{\ee}$, then by \cite[Theorem 4.2.16]{Zeitouni}, we deduce that $(\lambda_{X_N})_{N\in \NN}$, satisfies a weak LDP with rate function, 
$$ \Phi(x) = \sup_{\delta>0} \liminf_{\ee\to 0}\inf_{|y-x|< \delta} J_{\ee}(y),$$
As $J_{\ee}$ is nondecreasing in $\ee$, we get 
\begin{align}
 \Phi(x) & = \sup_{\delta>0}\inf_{\ee >0}\inf_{|y-x|< \delta} J_{\ee}(y) = \sup_{\delta>0}\inf_{|y-x|< \delta}\inf_{\ee >0} J_{\ee}(y)\nonumber \\
&= \sup_{\delta>0} \inf_{|y-x|< \delta} J(x), \label{rate func}\end{align}
with \begin{equation} \label{ratefunc}J(x) = \begin{cases}
 \inf \left\{ I(A) : A\in \cup_{n\geq1} H_n(\CC), \lambda_A = G_{\sigma_{sc}}(x)^{-1}, A\in \mathcal{D} \right\} & \text{ if } x>,2\\
0 & \text{ if } x=2,\\
+\infty & \text{ if } x<2.
\end{cases}
\end{equation}
As for any $t >0$, and $A\in H_n(\CC)$, $I(tA) = t^{\alpha} I(A)$ and $\lambda_{tA} = t \lambda_A$, and furthermore $\mathcal{D}$ is a cone, we have for any $x>2$,
$$J(x) = G_{\sigma_{sc}}(x)^{-\alpha}J(1).$$ 
As $G_{\sigma_{sc}}$ is non-increasing from $[2,+\infty)$ to $(0,1]$. This yields that $J$ has compact level sets. Therefore, from \eqref{rate func}, we get that $\Phi = J$, which concludes the proof.
\end{proof}

\section{Computation of $J(1)$}\label{calculc}
In this section, we compute the constant $c$ in Theorem \ref{pgdXn} explicitly, assuming certain conditions on the supports of the limiting angle distributions of the diagonal and off-diagonal entries (in the sense of \ref{hypo1}). In particular, when the entries are real random variables, or when $\alpha \in (0,1]$, the following proposition together with Theorem \ref{pgdXn}, gives an explicit formula for the rate function. 
\begin{Pro}\label{calcul expli}
With the notations of Theorem \ref{pgdXn}, we have the following :\\
(a). If $0<\alpha \leq 1$, then
$$c =  \begin{cases}
\min(b, a)& \text{ if } 1 \in \mathrm{supp}(\nu_1),\\
a & \text{ otherwise.}\end{cases}$$
(b). If $1<\alpha<2$ and $1\in\mathrm{supp}(\nu_1)$, and $b\leq \frac{a}{2}$, then $c = b$.\\
(c). If $1< \alpha <2$, $1\in \mathrm{supp}(\nu_1)\cap \mathrm{supp}(\nu_2)$ and $b> \frac{a}{2}$, then
$$c =\min \Big\{ I\Big(B^{(k)}\Big(\left(\frac{1}{b}\right)^{\frac{1}{\alpha-1}}, \left( \frac{2}{a} \right)^{\frac{1}{\alpha-1}} \Big)\Big) : n \in \NN \Big\},$$
where $B^{(n)}(s,t)$ denotes for any $(s,t)\neq (0,0)$, and $n\in \NN$, the following matrix of size $n \times n$,
\begin{equation} \label{defmat} B^{(n)}(s,t) = \frac{1}{s +(n-1)t}\left(
     \raisebox{0.5\depth}{%
       \xymatrixcolsep{1ex}%
       \xymatrixrowsep{1ex}%
       \xymatrix{
        s \ar @{.}[ddddrrrr]&t \ar @{.}[rrr] \ar @{.}[dddrrr] &  & &t \ar @{.}[ddd]  \\
         t \ar@{.}[ddd] \ar@{.}[dddrrr]& & & & \\
         &&&& \\
         &&&& t\\
         t \ar@{.}[rrr] & & & t & s
       }%
     }
   \right). \end{equation}
Equivalently,
$$  c = \min\left(\psi\left(\lfloor t_0 \rfloor \right), \psi\left(\lceil t_0 \rceil \right)\right),$$
where $\lfloor t_0 \rfloor$ and $\lceil t_0 \rceil$ denote respectively the lower and upper integer parts of $t_0$, and with $\psi$ and $t_0$ being defined by
\begin{equation} \label{defpsi} \forall t \geq 1, \ \psi(t) = \frac{t}{ \left( \left( \frac{1}{b}\right)^{\frac{1}{\alpha-1}} +\left(t-1\right) \left( \frac{2}{a} \right)^{\frac{1}{\alpha-1}}\right)^{(\alpha-1)}}, \ t_0 = \frac{1}{2-\alpha}\Big(1- \left( \frac{a}{2b} \right)^{\frac{1}{\alpha-1}}\Big).\end{equation}
(d). If $1< \alpha <2$,  $1\in \mathrm{supp}(\nu_1)$, and $\mathrm{supp}(\nu_2) = \{-1\}$ and $b> \frac{a}{2}$, then,
$$ c = \min\Big(b, \frac{2}{\left(\left(\frac{1}{b} \right)^{\frac{1}{\alpha-1}}+ \left( \frac{2}{a}\right)^{\frac{1}{\alpha-1}} \right)^{\alpha-1}} \Big).$$
(e). If $1< \alpha < 2$, $\mathrm{supp}(\nu_1) = \{-1\}$ and $1\in \mathrm{supp}(\nu_2)$, then
$$c=\min \left\{ I\left( B^{(n)}\left( 0,1\right)\right) : n\geq 2 \right\} =\min\left(\phi\left(\lfloor t_1 \rfloor \right), \phi\left(\lceil t_1 \rceil \right)\right),$$
where 
$$\forall t\geq 2, \ \phi(t) = \frac{t}{(t-1)^{\alpha-1}}, \quad t_1 = \frac{1}{2-\alpha}.$$
(f). If $1<\alpha<2$, and $\mathrm{supp}(\nu_1) = \mathrm{supp}(\nu_2) = \{-1\}$, then $ c =a.$\\

%

\end{Pro}

\begin{proof}

 
(a). Let $0< \alpha \leq 1$ and $1 \in \mathrm{supp}(\nu_1)$. Note that for any $A\in H_n(\CC)$ such that $\lambda_A = 1$, we have, for all $1\leq i,j \leq n$, $|A_{i,j}|\leq 1$. As $0< \alpha \leq 1$, we get,
\begin{align*}
c & \geq \left(b \wedge a \right) \inf\left\{ \sum_{i\geq 1} \left|A_{i,i}\right| + \sum_{i< j} \left| A_{i,j} \right| : \lambda_A =1, A\in \cup_{n\geq 1} H_n(\CC) \right\} \\
 &\geq \left(b \wedge a \right)\inf\left\{ \frac{1}{2}\left| \tr (A) \right| +\frac{1}{2} \sum_{i, j} \left| A_{i,j} \right| : \lambda_A =1, A\in \cup_{n\geq 1} H_n(\CC) \right\},
\end{align*}
where used the triangular inequality in the last inequality. But we know from \cite{Zhan}[Theorem 3.32], that for any $A \in H_n(\CC)$,
$$ \sum_{i,j} |A_{i,j}|\geq \sum_{i=1}^n \left|\lambda_i\right|,$$
where $\lambda_1,...,\lambda_n$ are the eigenvalues of $A$. Therefore, 
$$  c \geq\frac{1}{2} \left(b \wedge a \right) \inf_{n\geq 1} \inf \left\{  \left|1+ \sum_{i=1}^{n-1} \lambda_i \right| +\left( 1 +\sum_{i=1}^{n-1} \left| \lambda_i\right| \right): \lambda_1,..., \lambda_{n-1} \in \RR \right\}.$$
But, for all $\lambda_1,...,\lambda_{n-1}\in \RR$,
$$\left|1+ \sum_{i=1}^{n-1} \lambda_i \right| +\left( 1 +\sum_{i=1}^{n-1} \left| \lambda_i\right| \right)\geq 2+ \sum_{i=1}^{n-1} \left( \lambda_i +\left|\lambda_i\right| \right)\geq 2,$$
with equality for $\lambda_1=\lambda_2=...=\lambda_{n-1}=0$.

Therefore, $c \geq  \min(b, a)$.
But, as $1\in \mathrm{supp}(\nu_1)$,
\begin{equation} \label{bornesupfntaux}c \leq \min\left( I\left(1\right), I \left(
\begin{array}{cc}
0 & e^{i\theta}  \\
e^{-i\theta} & 0 
\end{array}
\right)\right),
\end{equation}
 with some $\theta \in \mathrm{supp}(\nu_{2})$. Thus, $c = \min(b,a)$.\\

 Let $0< \alpha \leq 1$, but assume $\mathrm{supp}(\nu_1) = \{-1\}$. Then,
\begin{align*}c&  \geq \inf\left\{ b \sum_{i\geq 1} \left|A_{i,i}\right| + a \sum_{i<j} \left|A_{i,j} \right| : A \in \cup_{n\geq 1} H_n(\CC), A_{i,i} \leq 0,  \forall i \in \NN, \lambda_A =1 \right\}\\
& = \inf \left\{ \left( b- \frac{a}{2} \right) \left| \sum_{i\geq1} A_{i,i} \right| + \frac{a}{2} \sum_{i,j} \left| A_{i,j} \right| : A \in \cup_{n\geq 1} H_n(\CC), A_{i,i} \leq 0,  \forall i\in \NN, \lambda_A =1 \right\}\\
& \geq \inf \left\{ \left( b- \frac{a}{2} \right) \left| \sum_{i\geq1} A_{i,i} \right| + \frac{a}{2} \sum_{i,j} \left| A_{i,j} \right| : A \in \cup_{n\geq 1} H_n(\CC), \tr A\leq 0,  \forall i\in \NN, \lambda_A =1 \right\}.
\end{align*}
Using again the fact that $\sum_{i,j} |A_{i,j} | \geq \sum_{i=1}^n \left|\lambda_i \right|$, where $A\in H_n(\CC)$, and  $\lambda_1,...,\lambda_n$ are the eigenvalues of $A$, we get
$$c \geq  \inf_{n\geq 1} \inf \left\{ \left( b- \frac{a}{2} \right) \left| 1 + \sum_{i=1}^{n-1} \lambda_i \right| + \frac{a}{2} \left( 1+ \sum_{i=1}^{n-1} \left| \lambda_{i} \right| \right): \lambda_1,...,\lambda_{n-1} \in \RR,  \sum_{i=1}^{n-1} \lambda_i \leq -1 \right\}.$$
But if $1+\sum_{i=1}^{n-1} \lambda_i \leq 0$, for $\lambda_1,...,\lambda_{n-1} \in \RR$, then
\begin{align*}
\left( b- \frac{a}{2} \right) \left| 1 + \sum_{i=1}^{n-1} \lambda_i \right| + \frac{a}{2} \left( 1+ \sum_{i=1}^{n-1} \left| \lambda_{i} \right| \right) &  = - \left( b -\frac{a}{2} \right) \left( 1+ \sum_{i=1}^{n-1} \lambda_i \right) + \frac{a}{2} \left( 1 +  \sum_{i=1}^{n-1} \left|\lambda_i\right| \right) \\
& = a-b\left( 1 +\sum_{i=1}^{n-1} \lambda_i\right)+\frac{a}{2} \sum_{i=1}^{n-1}\left( \left| \lambda_i \right| + \lambda_i \right)\\
&\geq a.
\end{align*}
Thus, $c \geq a$. But, $c \leq a$ by the same argument as in \eqref{bornesupfntaux}. Thus $c=a$.\\

(b). Let $1<\alpha <2$ and assume $1\in \supp(\nu_1)$ and $b\leq \frac{a}{2}$. Due to \cite{Zhan}[Theorem 3.32], we have for any $A\in H_n(\CC)$,
 $$I(A) \geq b\sum_{1\leq i,j\leq n} \left| A_{i,j} \right|^{\alpha} \geq b  \sum_{i=1}^n \left| \lambda_i\right|^{\alpha},$$
where $\lambda_1,...,\lambda_n$ are the eigenvalues of $A$. As $\lambda_A =1$, we get  $I(A) \geq b$. Therefore, $c \geq b$. As $1\in \supp(\nu_1)$, we also have $c \leq I((1))=b$, which ends the proof. \\

(c). Let $1< \alpha <2$, $b> \frac{a}{2}$ and assume $1 \in \supp(\nu_1) \cap \supp(\nu_2)$. We have the bound
$$ c\geq \inf_{n\geq 1} \inf\left\{ I(A) : A\in H_n(\CC), \lambda_A =1 \right\}.$$
Let $n\geq2$. We consider the minimization problem
$$\inf \left\{ I(A) : A \in H_n(\CC), \forall i \in \NN, \lambda_A =1 \right\}.$$
As $I$ is continuous and the constraints set is compact, the infimum is achieved at some $A\in H_n(\CC)$. If $1$ is an eigenvalue of $A$ of multiplicity greater that $2$, then denoting $\lambda_1,...,\lambda_n$ the eigenvalues of $A$, we have by \cite{Zhan}[Theorem 3.32], 
$$ I(A) \geq \frac{a}{2} \sum_{i=1}^n |\lambda_i|^{\alpha}\geq a.$$
As $A$ is a minimizer, and $1 \in \supp(\nu_1)\cap \supp(\nu_2)$,
 $$I(A) \leq I\left( 
\begin{array}{cc}
p & (1-p)  \\
(1-p) & p 
\end{array}
\right),
$$
where $p =\left(1+\left(\frac{2b}{a} \right)^{1/(\alpha-1)}\right)^{-1}$. As $2bp^{\alpha-1}= a(1-p)^{\alpha-1}$, 
\begin{align*}
 I\left( 
\begin{array}{cc}
p & (1-p)  \\
(1-p) & p 
\end{array}
\right)& = 2b p^{\alpha}+a (1-p)^{\alpha} \\
&= a(1-p)^{\alpha-1} p+ a(1-p)^{\alpha}\\
&=a(1-p)^{\alpha-1} <a,
 \end{align*}
where we used in the last inequality the fact that $ \alpha> 1$. This yields a contradiction. 

Therefore, $1$ must be a simple eigenvalue of $A$.
From the multipliers rule (see \cite{Clarke}[Theorem 10.48]), there exist $\eta, \gamma \in \RR$, $(\eta, \gamma) \neq 0$, such that $\eta =0, \text{ or } 1$, and
\begin{equation} \label{statcondsimpl} 0 \in \eta \{ \nabla I (A) \} - \gamma \partial \lambda(A), \end{equation}
where the gradient of $f$, and the subdifferential of $\lambda$, denoted $\partial\lambda$, are taken with respect to the canonical Hermitian product on $H_n(\CC)$. As a corollary of Danskin's formula (see \cite{Clarke}[Theorem 10.22]), we have the following lemma.
\begin{Lem}\label{subdifflambda} Let $\lambda : H_n(\CC) \to \RR$ be the largest eigenvalue function. The subdifferential of $\lambda$ at $A$, taken with respect to the canonical Hermitian product, is 
$$ \partial \lambda \left(A\right) =  \left\{ X \in H_n(\CC) :  0 \leq X \leq \Car_{E_{\lambda_A}(A)}, \tr X =1  \right\},$$
where $\Car_{E_{\lambda_A}(A)}$ denotes the projection on the eigenspace $E_{\lambda_A}(A)$ of $A$ associated with the largest eigenvalue of $A$, and $\leq $ is the order structure on $H_n(\CC)$.
\end{Lem}
%

As $1$ is a simple eigenvalue of $A$, we get from Lemma \ref{subdifflambda} that there is some unit eigenvector of $A$, $x$, associated with the eigenvalue $1$, such that 
$$ \eta \nabla I(A) = \gamma x x^*.$$
 We deduce that for any $i\neq j$,
\begin{equation} \label{condstatoffdiag} \eta \frac{a}{2}\alpha A_{i,j} \left|A_{i,j} \right|^{\alpha-2} = \gamma x_i\overline{x_j},\end{equation}
and for any $1\leq i\leq n$,
\begin{equation} \label{condstatdiag}  \eta b \alpha A_{i,i} \left|A_{i,i}\right|^{\alpha-2} = \gamma \left|x_i\right|^2,\end{equation}
with the convention that $z |z|^{\alpha-2} = 0$ when $z=0$. 
Multiplying the two equations above by $\overline{A_{i,j}}$ and $A_{i,i}$ respectively, and summing over all $ i,j\in \{1,...,n\}$, we get 
\begin{equation}\label{valeurI} \eta I(A) = \gamma.\end{equation}
As $(\eta,\gamma) \neq (0,0)$, this shows that $\eta =1$. 
Furthermore, the stationary condition yields for all $i\neq j$,
$$ A_{i,j} =  \left(\frac{2\gamma}{a \alpha}\right)^{\frac{1}{\alpha-1}} x_i \overline{x_j}\left|x_i x_j \right|^{\frac{1}{\alpha-1}-1},$$
and for all $1\leq i\leq n$,
$$ A_{i,i} = \left(\frac{\gamma}{b \alpha}\right)^{\frac{1}{\alpha-1}} \left|x_i\right|^{\frac{2}{\alpha-1}}.$$
Due to the eigenvalue equation $Ax=x$, we have for all $1\leq i\leq n$,
\begin{equation} \label{eigeneq} \left(\frac{\gamma}{b \alpha}\right)^{\frac{1}{\alpha-1}}  \left|x_i\right|^{\frac{2}{\alpha-1}} x_i + \left(\frac{2\gamma}{a \alpha}\right)^{\frac{1}{\alpha-1}}\sum_{j\neq i} x_i\left|x_i \right|^{\frac{1}{\alpha-1}-1} \left|x_j \right|^{\frac{1}{\alpha-1}+1} = x_i.\end{equation}
At the price of permuting the coordinates of $x$ and conjugating $A$ by a permutation matrix, we can assume $x=(x_1,...,x_k,0,...,0)$, with $x_1\neq0,...,x_k\neq 0$. Dividing by $x_i |x_i|^{\frac{1}{\alpha-1}-1}$ in \eqref{eigeneq}, we get
\begin{equation} \label{eigeneq2}B y = \left(\frac{\gamma}{\alpha}\right)^{-\frac{1}{\alpha-1}} y^{-\frac{2-\alpha}{\alpha}},\end{equation}
where $y  \in \RR^k$ is such that $y_i = |x_i|^{1+\frac{1}{\alpha-1}}$ for all $i\in \{1,...,k\}$, and where the power on the right-hand side must be understood entry-wise, and
 $$B = \left(
     \raisebox{0.5\depth}{%
       \xymatrixcolsep{1ex}%
       \xymatrixrowsep{1ex}%
       \xymatrix{
         \left(\frac{1}{b}\right)^{\frac{1}{\alpha-1}} \ar @{.}[ddddrrrr]&\left(\frac{2}{a}\right)^{\frac{1}{\alpha-1}} \ar @{.}[rrr] \ar @{.}[dddrrr] &  & &  \left(\frac{2}{a}\right)^{\frac{1}{\alpha-1}}  \ar @{.}[ddd]  \\
         \left(\frac{2}{a}\right)^{\frac{1}{\alpha-1}} \ar@{.}[ddd] \ar@{.}[dddrrr]& & & & \\
         &&&& \\
         &&&&\left(\frac{2}{a}\right)^{\frac{1}{\alpha-1}} \\
         \left(\frac{2}{a}\right)^{\frac{1}{\alpha-1}} \ar@{.}[rrr] & & &   \left(\frac{2}{a}\right)^{\frac{1}{\alpha-1}}& \left(\frac{1}{b}\right)^{\frac{1}{\alpha-1}}
       }%
     }
   \right)\in H_k(\CC).$$
As $x$ is a unit vector, we have $\sum_{i=1}^k y_i^{2(\alpha-1)/\alpha}=1$. Taking the scalar product with $y$ in \eqref {eigeneq2}, yields
$$\left(\frac{\gamma}{\alpha}\right)^{-\frac{1}{\alpha-1}} = \left \langle By,y\right\rangle.$$
As $I(A) = \frac{\gamma}{\alpha}$ by \eqref{valeurI}, we deduce that

\begin{equation} \label{minorationc} c \geq \left(\sup_{k\geq 1} \sup\left\{  \left \langle By,y\right\rangle : y\in [0,+\infty)^k, \sum_{i=1}^k y_i^{2(\alpha-1)/\alpha}\right\}\right)^{-(\alpha-1)}.\end{equation}
In the next lemma, we compute the maximum of certain quadratic forms, like the one given by the matrix $B$, on the unit $\ell^{\delta}$-sphere, intersected with $[0,+\infty)^n$, where $\delta \in (0,1)$.
\begin{Lem}\label{maxfqnonpos}
Let $\lambda, \mu \in \RR$ such that $0\leq \lambda < \mu$. Let $\delta \in (0,1)$. Define for any $n\in \NN$,
$$ B  = \left(
     \raisebox{0.5\depth}{%
       \xymatrixcolsep{1ex}%
       \xymatrixrowsep{1ex}%
       \xymatrix{
        \lambda \ar @{.}[ddddrrrr]& \mu \ar @{.}[rrr] \ar @{.}[dddrrr] &  & &\mu  \ar @{.}[ddd]  \\
         \mu \ar@{.}[ddd] \ar@{.}[dddrrr]& & & & \\
         &&&& \\
         &&&&\mu \\
         \mu \ar@{.}[rrr] & & & \mu& \lambda
       }%
     }
   \right)\in H_n(\CC). $$
 It holds
\begin{equation} \label{maxfq} \sup \left\{ \left\langle By,y \right\rangle :  y\in [0,+\infty)^n, \sum_{i=1}^n y_i^{\delta} =1\right\} = \max_{1\leq k\leq n} (\lambda+(k-1)\mu) k^{1-2/\delta}.\end{equation}
\end{Lem}

\begin{proof}Let $n \in \NN$.
By continuity and compactness arguments, we see that the supremum 
$$ \sup\left\{ \left\langle By,y \right\rangle :  y\in [0,+\infty)^n, \sum_{i=1}^n y_i^{\delta} =1\right\}, $$
 is achieved at some $y\in \RR^n$. At the price of re-ordering the coordinates of $y$, we can assume that $y = (z_1,...,z_m,0...,0)$, with $z_1>0,...,z_m>0$, for some $m \in \{1,...,n\}$. Then, the vector $z = (z_1,...,z_m) \in \RR^m$ is a solution of the optimization problem
$$ \sup \left\{ \left\langle Bz,z \right\rangle :  \forall i\in\{1,...,k\}, z_i \geq 0, \sum_{i=1}^m z_i^{\delta} =1\right\},$$
which lies in the interior of $[0,+\infty)^m$. 
 The multipliers rule (\cite{Clarke}[Theorem 9.1]) asserts that there is some $(\eta, \gamma) \neq (0,0)$, with $\eta =0$ or $1$, such that 
\begin{equation} \label{stateqfq} \eta Bz = \gamma z^{\delta-1},\end{equation}
where the power on the right-hand side has to be understood entry-wise. Taking the scalar product with $z$ in \eqref{stateqfq} yields
$$ \eta \left\langle Bz , z \right \rangle = \gamma.$$
We deduce that $\eta =1$. 
Moreover, for any $i\in\{1,...,m\}$, we have
\begin{equation} \label{reformstat} \mu \sum_{j=1}^m z_j =  \gamma z_i^{\delta-1} + \left(\mu - \lambda \right) z_i.\end{equation}
But then, the function 
$$ \forall s \in (0,+\infty), \ f(s) = \gamma s^{\delta-1} + (\mu - \lambda)s,$$
is decreasing on $(0,s_0]$, and increasing on $(s_0,+\infty)$, where
$$ s_0 = \left( \frac{\mu-\lambda}{\gamma(1-\delta)} \right)^{\frac{1}{2-\delta}}.$$
Thus, \eqref{reformstat} yields that $z$ has at most two distinct coordinates. Without loss of generality, we can assume that there are some $k,l \in \NN$, $k+l\leq m$, and $s,t \geq 0$, such that $ks^{\delta}+lt^{\delta} =1$, so that
$$ \forall i \in \{1,...,m\}, \ z_i = \Car_{i\leq k} s + \Car_{k+1 \leq i\leq k+l} t.$$
But then,
\begin{align*}
\left \langle Bz, z \right\rangle& = \lambda k s^2 + \mu k(k-1)s^2+ \lambda l t^2 + \mu l(l-1)t^2+ 2\mu kl st\\
&= k\left( \lambda +(k-1) \mu \right)s^2 + l\left( \lambda + (l-1)\mu\right)t^2 + 2\mu kl st.
\end{align*}
We can deduce that
$$ \max  \left\{\left\langle By,y \right\rangle : y\in [0,+\infty)^n, \sum_{i=1}^n z_i^{\delta} =1\right\} = \max_{\underset{k,l\in \NN}{k+l\leq n}} \max_{ \underset{ks^{\delta}+lt^{\delta}=1}{z =(s,t)}} \left\langle B z , z \right\rangle .$$
Let $k,l\in \NN$, $k+l\leq n$. Let $s,t \geq 0$, such that $ks^{\delta} +lt^{\delta} =1$. Setting $z =(s,t)$, we have
\begin{align*}
 \left \langle Bz, z \right\rangle = k^{1-2/\delta}\left( \lambda +(k-1)\mu\right) x^{2/\delta} + l^{1-2/\delta} \left( \lambda +(l-1)\mu \right)(1-x)^{2/\delta}&\\
+2\mu (kl)^{1-1/\delta} x^{1/\delta}(1-x)^{1/\delta},
\end{align*}
where $x = ks^{\delta}$. 
Define 
$$\forall x \in (0,+\infty), \ \phi(x) =  x^{1-2/\delta}(\lambda +(x-1)\mu).$$
Note that $\phi$ is increasing on $(0,x_0]$ and decreasing on $[x_0, +\infty)$, where
\begin{equation}\label{defx0} x_0 =  \frac{\left( \frac{2}{\delta}-1\right)}{\frac{2}{\delta}-2}\left( 1 -\frac{\lambda}{\mu} \right).\end{equation}
With this definition, we have
$$ \left \langle Bz, z \right\rangle = \phi(k)x^{2/\delta} + \phi(l)(1-x)^{2/\delta} + 2\mu(kl)^{1-1/\delta}x^{1/\delta}(1-x)^{1/\delta},$$
and 
$$ \max  \left\{\left\langle Bz,z \right\rangle :  \forall i\in\{1,...,k\}, z_i \geq 0, \sum_{i=1}^m z_i^{\delta} =1\right\} = \max_{\underset{k,l\in \NN}{k+l\leq n}} \max_{x\in [0,1]} f_{k,l}(x),$$
with 
$$  \forall x \in [0,1], \ f_{k,l}(x) = \phi(k)x^{2/\delta} + \phi(l)(1-x)^{2/\delta} + 2\mu(kl)^{1-1/\delta}x^{1/\delta}(1-x)^{1/\delta}.$$
Let $m \in \NN$, be such that $\phi(m) = \max\{ \phi(k) : k\in \NN^*\}$. Since $\phi$ is increasing on $(0,x_0]$ and decreasing on $[x_0,+\infty)$, we have $m \in \{ \lfloor x_0 \rfloor, \lceil x_0 \rceil\}$. Moreover $\phi$, restricted on $\NN\setminus\{0\}$, is increasing on $\{1,...,m\}$, and decreasing on $\{m,m+1,...,n\}$. As $\delta \in (0,1)$, we have for any $k,l\in \NN$, and $x\in [0,1]$,
$$ f_{k,l}(x) \leq \phi(k\wedge m)x^{2/\delta}+\phi(l\wedge m)(1-x)^{2/\delta}+ 2\mu ((k\wedge m)(l\wedge m))^{1-1/\delta}x^{1/\delta}(1-x)^{1/\delta}.$$
Therefore,
\begin{equation}\label{reduction} \max  \left\{\left\langle By,y \right\rangle : y \in [0,+\infty)^n,  \sum_{i=1}^n y_i^{\delta} =1\right\} = \max_{\underset{k,l\leq m }{k+l\leq n}} \max_{x\in [0,1]} f_{k,l}(x).\end{equation}

We are reduced to study the maximum of certain functions $f_{k,l}$ on the interval $[0,1]$.  The variations of those functions are given by the following lemma.

\begin{Lem}\label{var}
Let $a,b,c\geq0$, $a,c\neq 0$. Let also $\delta \in (0,1)$. Define
$$ \forall x \in[0,1], \ f(x) = ax^{2/\delta} + 2bx^{1/\delta}(1-x)^{1/\delta} + c(1-x)^{2/\delta}.$$
Then one of the following holds :\\
(a). There is some $x_1 \in (0,1)$, such that $f$ is decreasing on $[0,x_1]$, and increasing on $[x_1,1]$.\\
(b). There are some $0<x_1<x_2<x_3<1$, such that $f$ is decreasing on $[0,x_1]$, increasing on $[x_1,x_2]$, decreasing on $[x_2,x_3]$, and increasing on $[x_3,1]$.
\end{Lem}

\begin{proof}
We have, for all $x\in (0,1)$,
$$ \frac{\delta}{2}f'(x) = ax^{\frac{2}{\delta}-1}+bx^{\frac{1}{\delta}-1}(1-x)^{\frac{1}{\delta}} -bx^{\frac{1}{\delta}}(1-x)^{\frac{1}{\delta}-1}-c(1-x)^{\frac{2}{\delta}-1}.$$
We write
\begin{equation} \label{derive} \frac{\delta}{2}f'(x)  = x^{\frac{2}{\delta}-1} \left( a +bs^{\frac{1}{\delta}}-bs^{\frac{1}{\delta}-1} -cs^{\frac{2}{\delta}-1}\right),\end{equation}
where $s = \frac{1-x}{x}$. Set for all $s\in (0, +\infty)$, $g(s) =   a +bs^{\frac{1}{\delta}}-bs^{\frac{1}{\delta}-1} -cs^{\frac{2}{\delta}-1}$.
Then, for all $s\in (0,+\infty)$
$$g'(s) = \frac{b}{\delta} s^{\frac{1}{\delta}-1}-b\left( \frac{1}{\delta}-1\right) s^{\frac{1}{\delta}-2}-c\left( \frac{2}{\delta}-1\right)s^{\frac{2}{\delta}-2} = s^{\frac{1}{\delta}-2}h(s),$$
with $h(s) = \frac{b}{\delta}s-b(\frac{1}{\delta}-1)-c(\frac{2}{\delta}-1)s^{\frac{1}{\delta}}$. Deriving once more, we get for any $s\in (0,+\infty)$,
$$h'(s) = \frac{b}{\delta}-\frac{c}{\delta}\left( \frac{2}{\delta}-1\right)s^{\frac{1}{\delta}-1}.$$
As $\delta \in (0,1)$,  we see that $h'$ is decreasing. This entails that $f$ has at most three changes of variations. As $f'(0) < 0$, and $f'(1)<0$, we deduce that $f$ is either decreasing on $[0,x_1]$, and increasing on $[x_1,1]$, for some $x_1\in [0,1]$, or there are some $x_1<x_2<x_3$ such that $f$ is decreasing on $[0,x_1]$ and $[x_2, x_3]$, and increasing on $[x_1,x_2]$ and $[x_3, 1]$. 

\end{proof}

Let $k,l \in \NN$, $1\leq k \leq l\leq m$. If $k=l$, then the graph of $f_{k,l}$ is symmetric with respect to $1/2$. By the previous Lemma \ref{var}, this entails that if $f_{k,l}$ has a local maximum in $(0,1)$, then it must be at $1/2$. One can easily check that $f_{k,k}(1/2) = \phi(2k)$. Thus,
$$ \max_{x \in [0,1]} f_{k,k}(x) =  \max\left( \phi(2k), \phi(k)\right).$$
Assume now $1\leq k<l\leq m$. We will show that the maximum of $f_{k,l}$ is achieved at either $0$, $k/(k+l)$ or $1$. We can write for any $x \in [0,1]$,
$$\frac{\delta}{2}f_{k,l}'(x)  = (x(1-x))^{\frac{1}{\delta}-\frac{1}{2}} \left( \phi(k)s^{-\frac{1}{\delta}+\frac{1}{2}} +2(kl)^{1-1/\delta}\mu\left( s^{-\frac{1}{2}}-s^{\frac{1}{2}}\right) -\phi(l)s^{\frac{1}{\delta}-\frac{1}{2}}\right),$$
with $s = \frac{1-x}{x}$. Let $y = \frac{k}{l}s$. We can write
$$\frac{\delta}{2}f_{k,l}'(x)  = \left(\frac{x(1-x)}{kl}\right)^{\frac{1}{\delta}-\frac{1}{2}} g_{k,l}(y),$$
with $g_{k,l}(y) = (\lambda +(k-1)\mu)y^{-\frac{1}{\delta}+\frac{1}{2}} +2\mu\left( k y^{-\frac{1}{2}}-ly^{\frac{1}{2}}\right) - (\lambda+(l-1)\mu) y^{\frac{1}{\delta}-\frac{1}{2}}$. Note that $ g_{k,l}(1)=0$, so that $f'_{k,l}(\frac{k}{k+l})=0$. Observe that $y$ is a decreasing function of $x$. Thus, to show that $k/(k+l)$ is a local maximum of $f_{k,l}$, we need to show that $g'_{k,l}(1)>0$.
But
\begin{align*}
g'_{k,l}(1) &= -\left(\frac{1}{\delta}-\frac{1}{2}\right) \left( 2\left( \lambda -\mu \right)+(k+l)\mu \right) - \mu\left( k+l\right) \\
&=  \left(\frac{2}{\delta}-1\right)\left( \mu-\lambda\right) - \mu\left( \frac{1}{\delta}+\frac{1}{2} \right)(k+l).
\end{align*}
Thus,
$$
g'_{k,l}(1) >  0\Longleftrightarrow \frac{k+l}{2} <\frac{\left( \frac{2}{\delta}-1\right)}{\frac{2}{\delta}-2}\left( 1 -\frac{\lambda}{\mu} \right)
\Longleftrightarrow \frac{k+l}{2} <x_0,$$
with $x_0$ as in \eqref{defx0}.
If $m =\lfloor x_0 \rfloor$, then 
$$ \frac{k+l}{2}< \lfloor x_0 \rfloor \leq x_0,$$
so that $g'_{k,l}(1)>0$. This yields that $\frac{k}{k+l}$ is a local maximum of $f_{k,l}$. By Lemma \ref{var}, we deduce that the maximum of $f_{k,l}$ is achieved at either $0$, $k/(k+l)$, or $1$.  Moreover, one can check that $f_{k,l}\left( \frac{k}{k+l}\right) = \phi(k+l)$.
Therefore,
$$ \max_{[0,1]} f_{k,l} = \max\left( \phi(k), \phi(l), \phi(k+l)\right).$$
Assume now $m= \lceil x_0 \rceil$. We can assume without loss of generality that $m>1$. Whenever $k+l<2x_0$, we can use the same argument as above to identify the maximum of $f_{k,l}$. Thus, we are reduced to find the maximum of $f_{m,m-1}$. We have for any $x\in [0,1]$,
$$f_{m,m-1}(x) \leq \phi(m)x^{2/\delta}+2\mu(m(m-1))^{1-1/\delta}x^{1/\delta}(1-x)^{1/\delta} + \phi(m)(1-x)^{2/\delta},$$
since $\phi$ is increasing on $\{1,...,m\}$. As the function on the left-hand side, which we denote by $f$, is symmetric with respect to $1/2$, we get by Lemma \ref{var} that its maximum is achieved at $0$, $1$ or $1/2$. Thus,
$$\max_{x\in [0,1]} f_{m,m-1}(x) \leq \max\Big( \phi(m), f\Big( \frac{1}{2}\Big)\Big).$$
But,
$$\phi(m) \geq  f\Big( \frac{1}{2}\Big) \Longleftrightarrow \Big(1-\frac{1}{m}\Big)^{1-1/\delta}\leq \frac{1-2^{1-2/\delta}}{2^{1-2/\delta}}\Big(1-\frac{1}{m}(1-\frac{\lambda}{\mu})\Big).$$
As $\delta \in (0,1)$ and $m\geq 2$,
$$\frac{1-2^{1-2/\delta}}{2^{1-2/\delta}}\Big(1-\frac{1}{m}(1-\frac{\lambda}{\mu})\Big)- \Big(1-\frac{1}{m}\Big)^{1-1/\delta}\geq \frac{1-2^{1-2/\delta}}{2^{2-2/\delta}}\Big(1+\frac{\lambda}{\mu}\Big)- 2^{-1+1/\delta}.$$
The same argument as in the case where $ m =\lfloor x_0 \rfloor$ shows that 
$$ \max_{[0,1]} f_{k,l} = \max\left( \phi(k), \phi(l), \phi(k+l)\right).$$
We conclude from \eqref{reduction} that 
$$\max  \left\{\left\langle By,y \right\rangle :  y\in [0,+\infty)^n, \sum_{i=1}^n y_i^{\delta} =1\right\} = \max_{1\leq k \leq n} \phi(k).$$
\end{proof}
We come back now at the proof of case (c).
As $\alpha \in (1,2)$, we have that $2(\alpha-1)/\alpha \in (0,1)$. From Lemma \ref{maxfqnonpos} and \eqref{minorationc}, we get
\begin{align*}
 c&\geq \Big\{ \max_{k\geq 1}\left( \Big( \frac{1}{b}\Big)^{\frac{1}{\alpha-1}} +(k-1) \Big( \frac{2}{a}\Big)^{\frac{1}{\alpha-1}}\right)k^{-(\alpha-1)} \Big\}^{-(\alpha-1)}\\
& = \min_{k\geq 1 } \psi(k),
\end{align*}
where $\psi$ is defined in the statement of Proposition \ref{calcul expli}. As $1\in \supp(\nu_1) \cap\supp(\nu_2)$, the matrix $B^{(k)}\left(\left(\frac{1}{b}\right)^{\frac{1}{\alpha-1}}, \left( \frac{2}{a} \right)^{\frac{1}{\alpha-1}} \right)$ defined in \eqref{defmat}, is in the domain $\mathcal{D}$, and 
$$I\Big(B^{(k)}\Big(\left(\frac{1}{b}\right)^{\frac{1}{\alpha-1}}, \left( \frac{2}{a} \right)^{\frac{1}{\alpha-1}} \Big)\Big) =  \psi(k),$$
which gives the first part of the claim in case (c).

Easy computations show that the function $\psi$ defined in \eqref{defpsi} is decreasing on $[0,t_0]$ and increasing on $[t_0, 1]$, with
$$t_0 = \frac{1}{2-\alpha}\left(1- \left( \frac{2b}{a} \right)^{-  \frac{1}{\alpha-1}}\right).$$
Thus,
$$c = \min\left( \psi\left(\lfloor t_0 \rfloor \right) , \psi\left(\lceil t_0 \rceil \right)\right).$$

(d). Let $1<\alpha <2$ and assume $1 \in \supp(\nu_1)$, $\supp(\nu_2) =\{-1\}$ and $b>\frac{a}{2}$. Then, 
$$ c = \inf_{n\geq 1} \inf\left\{ I(A) : A\in S_n(\RR), \lambda_A =1, A_{i,j} \leq 0, \forall i\neq j \right\},$$
where $S_n(\RR)$ denotes the set of real symmetric matrix of size $n$.

Let $n\geq 1$. We consider the minimization problem 
$$  \inf\left\{ I(A) : A\in S_n(\RR), \lambda_A =1, A_{i,j} \leq 0, \forall i\neq j \right\}.$$
The same argument as in case (c) justifies that the infimum is achieved at some $A\in S_n(\RR)$ for which $1$ is a simple eigenvalue. 
As in case (c), the multipliers rule (see \cite{Clarke}[Theorem 9.1]) asserts that there are some $( M, \gamma) \in  S_n(\RR) \times \RR$ such that $( M, \gamma)\neq (0,0)$, and
$$  \forall i\neq j,  M_{i,j} \geq 0,  M_{i,j} A_{i,j} =0, \text{ and } M_{i,i} =0, \forall 1\leq i\leq n, ,$$
satisfying
$$ \nabla I(A) + M = \gamma \! x {}^t{x},$$
where $x$ is a unit eigenvector associated with the eigenvalue $1$. We deduce that for any $i\neq j$,
\begin{equation} \label{condstatoffdiagreel}  \frac{a}{2}\alpha A_{i,j} \left|A_{i,j} \right|^{\alpha-2} +M_{i,j} = \gamma x_ix_j,\end{equation}
and for any $1\leq i\leq n$,
\begin{equation} \label{condstatdiagreel}   b \alpha A_{i,i} \left|A_{i,i}\right|^{\alpha-2} = \gamma x_i^2.\end{equation}
The same argument as in case (c), shows that
\begin{equation} \label{valueI}  \alpha I(A) = \gamma.\end{equation}
Without loss of generality, we can assume $x$ is of the form $x = (x_1,...,x_k, x_{k+1},...,x_{k+l}, 0,...0)$, with $x_1>0,..., x_k>0$, and $x_{k+1}<0,...,x_{k+l}<0$. 

Note that as $A_{i,j}M_{i,j} =0$, $M_{i,j}\geq 0$, and $A_{i,j}\leq 0$, for any $i\neq j$, we get from \eqref{condstatoffdiagreel}, that for any $i\neq j$, $A_{i,j}\neq 0$ if and only if $x_i x_j <0$. Thus, for all $i\neq j$,  $A_{i,j} \neq 0$, if and only if $(i,j)$ or $(j,i) \in \{1,...,k\}\times\{k+1,...,k+l\}$. 

Let $(i,j) \in \{1,...,k\}\times\{k+1,...,k+l\}$. From \eqref{condstatoffdiagreel}, we have
$$ A_{i,j} = - \left(\frac{2\gamma}{a \alpha}\right)^{\frac{1}{\alpha-1}} \left|x_i x_j \right|^{\frac{1}{\alpha-1}},$$
and for all $i \in \{ 1,...,k+l\}$, we get by \eqref{condstatdiagreel},
$$ A_{i,i} = \left(\frac{\gamma}{b \alpha}\right)^{\frac{1}{\alpha-1}} \left|x_i\right|^{\frac{2}{\alpha-1}}.$$
The eigenvalue equation $Ax =x$, yields, for any $i\in \{1,...,k\}$,
$$ \left( \frac{\gamma}{b \alpha}\right)^{\frac{1}{\alpha-1}} \left|x_i\right|^{\frac{2}{\alpha-1}+1}  + \left(\frac{2\gamma}{a \alpha}\right)^{\frac{1}{\alpha-1}} \sum_{k+1 \leq j\leq k+l}\left|x_i \right|^{\frac{1}{\alpha-1}}\left|x_j \right|^{\frac{1}{\alpha-1}+1} =\left|x_i\right|,$$
as $x_j <0$ for $j\in \{k+1,..., k+l\}$, and $x_i>0$ for $i\in \{ 1,..., k\}$. 

Similarly, for any $i\in \{k+1,...,k+l\}$,
$$  -\left(\frac{\gamma}{b \alpha}\right)^{\frac{1}{\alpha-1}} \left|x_i\right|^{\frac{2}{\alpha-1}+1}  - \left(\frac{2\gamma}{a \alpha}\right)^{\frac{1}{\alpha-1}} \sum_{1 \leq j\leq k}\left|x_i\right|^{\frac{1}{\alpha-1}} \left| x_j \right|^{\frac{1}{\alpha-1}+1} =-\left|x_i\right|.$$
Dividing in the two equations above by $|x_i|^{\frac{1}{\alpha-1}}$, we get
\begin{equation} \label{eqvp2} B y = \left(\frac{\gamma}{\alpha}\right)^{-\frac{1}{\alpha-1}}y^{-\frac{2-\alpha}{\alpha}},\end{equation}
with $y \in \RR^{k+l}$, such that $y_i = |x_i|^{\frac{1}{\alpha-1}+1}$, for all $i \in \{1,...,k+l\}$, and 
 $$B=\left(
\begin{array}{c|c}
\left(\frac{1}{b} \right)^{\frac{1}{\alpha-1}} I_k  & \left( \frac{2}{a} \right)^{\frac{1}{\alpha-1}}U_{k,l} \\ \hline 
 \left( \frac{2}{a} \right)^{\frac{1}{\alpha-1}}{}^tU_{k,l} &\left(\frac{1}{b} \right)^{\frac{1}{\alpha-1}}I_l 
\end{array}\right)\in S_{k+l}(\RR),$$
where $U_{k,l}$ is the matrix of size $k\times l$ whose entries are all equal to $1$. As $x$ is a unit vector, we have $\sum_{i=1}^{k+l} y^{\frac{2(\alpha-1)}{\alpha}}=1$.  We deduce from \eqref{eqvp2}, that
$$ \left( \frac{\gamma}{\alpha}\right)^{-\frac{1}{\alpha-1}} =\left\langle By,y\right \rangle.$$
Using \eqref{valueI} and the fact that $A$ is a minimizer, we get
\begin{equation}\label{egalpbdual}  c = \left(\sup_{k,l\in \NN} \sup \left\{ \left\langle By,y\right \rangle : \sum_{i=1}^{k+l} y_i^{2(\alpha-1)/\alpha} =1, y \in [0,+\infty)^{k+l} \right\}\right)^{-(\alpha-1)}.\end{equation}
In the following lemma, we compute the supremum of the left-hand side of the above inequality.

\begin{Lem}\label{maxfq2}
Let $\delta \in (0,1)$. Let $k,l \in \NN$, such that $(k,l) \neq (0,0)$. Let $\lambda,\mu \in \RR_+$, and define 
 $$B=\left(
\begin{array}{c|c}
\lambda I_k  & \mu U_{k,l} \\ \hline 
\mu {}^tU_{k,l} &\lambda I_l 
\end{array}\right)\in S_{k+l}(\RR),$$
where $U_{k,l}$ is the matrix of size $k\times l$ whose entries are all equal to $1$.
We have,
$$\sup \left\{ \left\langle By,y\right \rangle : \sum_{i=1}^{k+l} y_i^{\delta} =1, y \in [0,+\infty)^{k+l} \right\} = \max\left( \lambda, (\lambda+\mu)2^{1-2/\delta}\right).$$

\end{Lem}
\begin{proof}
With the same arguments as in the proof of Lemma \ref{maxfqnonpos}, the supremum of the quadratic form defined by $B$ on 
$$\left\{ y \in [0,+\infty)^{k+l} : \sum_{i=1}^{k+l} y_i^{\delta} = 1\right\},$$
is achieved at some $y$ such that,
$$ \forall i \in\{1,...,k+l\}, \ y_i = s_i\Car_{i\leq k'} + t_{k'+i} \Car_{1\leq i\leq l'},$$ with $s_1>0,...,s_{k'}>0$,  and $t_{k'+1}>0, ....,t_{k'+l'}>0$, for some $k'\leq k$ and $l'\leq l$, such that the vector $z = (s_1,...,s_{k'},t_{k'+1},...,t_{k'+l'}) \in \RR^{k'+l'}$, satisfies
for some $\gamma \in \RR$,
$$ \tilde{B}z = \gamma z^{\delta-1},$$
where
$$\tilde{B}=\left(
\begin{array}{c|c}
\lambda I_{k'}  & \mu U_{k',l'} \\ \hline 
\mu {}^tU_{k',l'} &\lambda I_{l'} 
\end{array}\right)\in S_{k'+l'}(\RR).$$
Comparing the $i^{\text{th}}$ and $j^{\text{th}}$ coordinates of $Bz$, for $1\leq i,j\leq k'$, we get
$$ \lambda \left( s_i-s_j\right) = \gamma \left( s_i^{\delta-1} - s_j^{\delta-1}\right).$$
If $\lambda =0$, then it immediately yields $s_i=s_j$. If $\lambda \neq 0$, as $\delta \in (0,1)$, we see that if $s_i \neq s_j$, the terms on the left-hand side, and the right-hand side must have opposite signs. Thus $s_i = s_j$  for any $i,j \in \{1,...,k'\}$. 
 Comparing the $i^{\text{th}}$ and $j^{\text{th}}$ coordinates of $By$, for $i,j \in \{k'+1,...,k'+l'\}$, yields that $t_i = t_j $, for all $i,j\in \{k'+1,...,k'+l'\}$, with the same argument. We can write
$$ \forall i\in \{1,...,k'+l'\}, \ z_i = s\Car_{i\leq k'} +  t\Car_{k'+1\leq i\leq k'+l'},$$
for some $s,t \in (0,+\infty)$. As $\sum_{i=1}^{k'+l'} z_i^{\delta} = 1$,
$$ k's^{\delta} + l't^{\delta} =1.$$
Let $v = (k'^{1/\delta}s, l'^{1/\delta} t )$.
Then,
$$  \left\langle \tilde{B}z, z\right\rangle= \lambda(k's^2+l't^2) +2\mu k'l'ts = \left\langle M(k',l') v, v \right\rangle,$$
where
$$ M(k',l') = \left(
\begin{array}{cc}
\lambda k'^{1-2/\delta} &\mu (k'l')^{1-1/\delta} \\
\mu (k'l')^{1-1/\delta} & \lambda l'^{1-2/\delta}
\end{array}
\right).$$
Thus,
\begin{align*}
\sup \left\{ \left\langle By,y\right \rangle : \sum_{i=1}^{k+l} y_i^{\delta} =1, y \in [0,+\infty)^{k+l} \right\} &=\sup_{k'\leq k, l'\leq l} \sup_{\underset{s^{\delta} + t^{\delta}=1, s,t\geq 0}{v=(s,t)} } \left\langle M(k',l')v,v\right\rangle\\
& = \sup_{\underset{s^{\delta} + t^{\delta}=1, s,t\geq 0}{v=(s,t)}}   \sup_{k'\leq k, l'\leq l}  \left\langle M(k',l')v,v\right\rangle.
\end{align*}
But for fixed $v\in \RR^2$, as $\delta \in (0,1)$, we easily see that the maximum of $\left\langle M(k',l')v,v\right\rangle$ is achieved at $k'=l'=1$. Thus,
$$\sup \left\{ \left\langle By,y\right \rangle : \sum_{i=1}^{k+l} y_i^{\delta} =1, y \in [0,+\infty)^{k+l} \right\} = \sup_{\underset{s^{\delta} + t^{\delta}=1, s,t\geq 0}{v=(s,t)}}  \left\langle M(1,1)v,v\right\rangle.$$
From Lemma \ref{maxfqnonpos}, we get
$$\sup_{\underset{s^{\delta} + t^{\delta}=1, s,t\geq 0}{v=(s,t)}}  \left\langle M(1,1)v,v\right\rangle = \max\left( \lambda, (\lambda+\mu) 2^{1-2/\delta} \right),$$
which yields the claim.
\end{proof}
We come back now to the proof of case (d). By Lemma \ref{maxfq2} and \eqref{egalpbdual}, we get
$$ c =\max\Big( b, \frac{2}{\left( \left(\frac{1}{b} \right)^{\frac{1}{\alpha-1}} + \left( \frac{2}{a} \right)^{\frac{1}{\alpha-1}}\right)^{\alpha-1} }  \Big), $$
which gives the claim.

(e). Let $1< \alpha <2$, and assume $1\in \supp(\nu_2)$ and $\supp(\nu_1) = \{-1\}$. Then,
$$ c \geq \inf_{n\geq 2} \inf \left\{ I(A) : A \in H_n(\CC), A_{i,i} \leq 0,  \forall i \in \NN, \lambda_A =1 \right\}.$$
Let $n\geq 2$. We consider the minimization problem
$$\inf \left\{ I(A) : A \in H_n(\CC), A_{i,i} \leq 0,  \forall i \in \NN, \lambda_A =1 \right\}.$$
The same arguments as in case (c) and (d) show that the infinmum is achieved at some $A$ such that $A_{i,i} = 0$ for all $1\leq i\leq n$, and such that for any $i\neq j$,
$$ A_{i,j} = \left( \frac{2\gamma}{ a\alpha} \right)^{\frac{1}{\alpha-1}} X_{i,j} \left|X_{i,j} \right|^{\frac{1}{\alpha-1}-1},$$
where $\gamma = \alpha I(A)$, and $X\in H_n(\CC)$ is such that $0 \leq X \leq \Car_{E_{1}(A)}$, and  $\tr X =1$. We deduce that $\tr AX =1$. This yields,
$$ \left( \frac{2\gamma}{ a\alpha} \right)^{\frac{1}{\alpha-1}} \sum_{i\neq j} \left| X_{i,j} \right|^{\frac{1}{\alpha-1}+1}=1.$$
As $I(A) = \frac{\gamma}{\alpha}$, we have
\begin{equation} \label{pbdualdiagneg} I(A) =\frac{a}{2} \left( \sum_{i\neq j } \left| X_{i,j} \right|^{\frac{1}{\alpha-1}+1} \right)^{-(\alpha-1)} \geq \frac{a}{2} \left( \max_{\underset{X\geq 0}{ \tr X =1}} \sum_{i\neq j } \left| X_{i,j} \right|^{\frac{1}{\alpha-1}+1} \right)^{-(\alpha-1)}.\end{equation}
In the following lemma, we compute the maximum on the right-hand side.

\begin{Lem}\label{maxgeq2}
Let $\beta \geq 2$. We have for any $n\in \NN$, $n\geq 2$,
$$ \max \left\{ \sum_{1\leq i\neq j\leq n} \left| X_{i,j} \right|^{\beta} : X \in H_n(\CC), X\geq 0, \tr X =1 \right\} = \max_{2\leq k\leq n} (k-1)k^{1-\beta}.$$
\end{Lem}
\begin{proof}
Let $\phi : X\in H_n(\CC) \mapsto \sum_{i\neq j} |X_{i,j}|^{\beta}$. Note that $\phi$ is convex, and that the constraints set, $$S = \left\{ X \in H_n(\CC) : X\geq 0, \tr X =1 \right\},$$
is also convex. As a consequence of \cite{Rockafeller}[Corollary 18.5.1], $\phi$ attains its maximum at an extreme point of the set $S$, which is of the form $x x^*$, with $x$ a unit vector of $\CC^n$. We deduce that, 
$$  \max_S \phi = \max \left\{ \sum_{1\leq i\neq j\leq n} \left| x_i x_j\right|^{\beta} : x \in \CC^n, ||x||=1 \right\}.$$
We can re-write the maximum on the right-and side of the above equation as,
$$ \max\left\{\left\langle By,y \right\rangle : \forall i\in \{1,...,n\}, \ y_i \geq 0, \sum_{i=1}^n y_i^{2/\beta} =1 \right\},$$
where
 $$B = \left(
     \raisebox{0.5\depth}{%
       \xymatrixcolsep{1ex}%
       \xymatrixrowsep{1ex}%
       \xymatrix{
        0 \ar @{.}[ddddrrrr]& 1 \ar @{.}[rrr] \ar @{.}[dddrrr] &  & &1  \ar @{.}[ddd]  \\
         1 \ar@{.}[ddd] \ar@{.}[dddrrr]& & & & \\
         &&&& \\
         &&&&1 \\
         1 \ar@{.}[rrr] & & & 1& 0
       }%
     }
   \right)\in H_n(\CC). $$

Applying the result of Lemma \ref{maxfqnonpos}, with $\delta = 2/\beta$, we get the claim.

\end{proof}
We come back at the proof of Proposition \ref{calcul expli}, (e). Note that, as $1 < \alpha <2$, we have $1 +\frac{1}{\alpha-1}\geq 2$. From \eqref{pbdualdiagneg} together with Lemma \ref{maxgeq2}, we get
$$c \geq \frac{a}{2}\left( \max_{n\geq 2} (n-1) n^{-\frac{1}{\alpha-1}} \right)^{-(\alpha-1)} = \frac{a}{2} \min \frac{n}{\left(n-1\right)^{\alpha-1}}.$$
But,
$$ \frac{a}{2} \frac{n}{\left(n-1\right)^{\alpha-1}} = I (B^{(n)}\left( 0,1\right)),$$
where $B^{(n)}\left( 0,1\right)$ is defined in \eqref{defmat}. As $1\in \supp(\nu_2)$, we have $B^{(n)}\left( 0,1\right) \in \mathcal{D}$, which ends the proof of the case (e).

%
%
%
%
%
%


(f). Assume finally $1< \alpha <2$, and $\supp(\nu_1) = \supp(\nu_2) = \{-1\}$. Let $n\geq 1$ and consider the minimization problem 
$$  \inf\left\{ I(A) : A\in S_n(\RR), \lambda_A =1, A_{i,j} \leq 0, \forall i\leq  j \right\}.$$
 The same arguments as in the case (e), show that the minimizer $A$ is such that $A_{i,i} = 0$ for all $i\in \{1,...,n\}$. If $A$ is a simple eigenvalue of $A$, then, the same analysis can be carried as in the case (d), and yields
$$I(A) \geq \left(\sup_{k,l\in \NN} \sup \left\{ \left\langle By,y\right \rangle : \sum_{i=1}^{k+l} y_i^{2(\alpha-1)/\alpha} =1, y \in [0,+\infty)^{k+l} \right\}\right)^{-(\alpha-1)},$$
with
 $$B=\left(
\begin{array}{c|c}
0_k & \left( \frac{2}{a} \right)^{\frac{1}{\alpha-1}}U_{k,l} \\ \hline 
 \left( \frac{2}{a} \right)^{\frac{1}{\alpha-1}}{}^tU_{k,l} &0_l
\end{array}\right)\in S_{k+l}(\RR),$$
where $U_{k,l}$ is the matrix of size $k\times l$ whose entries are all equal to $1$, and $0_k$, $0_l$ are the null matrices of sizes $k\times k$ and $l\times l$ respectively. Due to Lemma \ref{maxfq2}, we have
$$\sup_{k,l\in \NN} \sup \left\{ \left\langle By,y\right \rangle : \sum_{i=1}^{k+l} y_i^{2(\alpha-1)/\alpha} =1, y \in [0,+\infty)^{k+l} \right\} = \left(\frac{2}{a}\right)^{\frac{1}{\alpha-1}}2^{-\frac{1}{\alpha-1}}.$$
Therefore, $I(A) \geq a$.

Now, if $1$ is not a simple eigenvalue of $A$, then we have by \cite{Zhan}[Theorem 3.32],
$$ I(A) = \frac{a}{2} \sum_{i\neq j } |A_{i,j}|^{\alpha} = \frac{a}{2} \sum_{i,j} |A_{i,j}|^{\alpha} \geq \frac{a}{2} \sum_{i=1}^n |\lambda_i|^{\alpha}\geq a,$$
where $\lambda_1,...,\lambda_n$ are the eigenvalues of $A$.

In both cases, $I(A) \geq a$. We deduce that $c \geq a$, and as 
$$ I\left( \begin{array}{cc}
0 & -1\\
-1 &0 
\end{array}
\right) = a,
$$
we get the claim.
\end{proof}

\section{Appendix }
\subsection{Linear algebra tools}

\begin{Pro}\label{formule Frobenius}
Let $p, q$ be two integers. Let $A \in \mathcal{M}_{p,q}(\CC), B \in \mathcal{M}_{q, p}(\CC)$. Then,
$$\det\left(I_p - AB\right) = \det\left(I_q- BA\right).$$
\end{Pro}

\begin{Lem}[Weyl's inequality]From \cite[p.415]{Guionnet}. \label{Weyl}
For any Hermitian matrix $X\in H_n(\CC)$, we denote by $\lambda_k$ its eigenvalues with $\lambda_1(X) \leq ...\leq \lambda_n(X)$. Let $A$ and $E$ be in $H_n(\CC)$. For all $k\in \{1,...,n\}$, we have
$$\lambda_k(A)+\lambda_1(E)\leq \lambda_k(A+E) \leq  \lambda_k(A) + \lambda_k(E).$$
\end{Lem}


\subsection{Concentration inequalities}

\begin{Pro}(Bennett's inequality, see \cite[p. 35]{Massart}) \label{bennett}
Let $X_1,...,X_n$ be independent random variable with finite variance such that $X_i \leq b$ for some $b>0$ almost surely for all $i\leq n$. Let
$$S = \sum_{i=1}^n \left(X_i - \EE X_i \right)$$
and $v = \sum_{i=1}^n \EE[X_i^2]$. Then for any $t>0$,
$$\PP\left(S>t  \right) \leq \exp\left( - \frac{v}{b^2} h\left(\frac{bt}{v} \right) \right),$$
where $h(u) = (1+u) \log(1+u) - u$ for $u>0$.
\end{Pro}

\begin{Lem}\cite[p.249]{Massart}\label{concentration transport 1}
Let $\mathcal{X}$ a measurable space.
Let $f : \mathcal{X}^n \to [0,+\infty)$ be a measurable function, and let $X_1,...X_n$ be independent random variables taking their values in $\mathcal{X}$. Define $Z =f(X_1,...X_n)$. Assume that there exist measurable functions $c_i : \mathcal{X}^n \to [0,+\infty)$ such that for all $x, y \in \mathcal{X}^n$,
$$ f(y) - f(x) \leq \sum_{i=1} \Car_{x_i \neq y_i} c_i(x).$$
Setting
$$v = \EE\sum_{i=1}^n\left(c_i(X)^2\right) \ \ \text{and}\  \ v_{\infty} = \sup_{x\in \mathcal{X}^n}\sum_{i=1}^n c_i(x)^2,$$
we have for all $t>0$,
$$\PP\left( Z\geq \EE(Z) +t\right) \leq e^{-t^2/2v},$$
and
$$\PP\left(Z\leq \EE(Z)- t \right) \leq e^{-t^2/2v_{\infty}}.$$

\end{Lem}

\newpage

\bibliographystyle{plain}
\bibliography{LDPeimaxnongaussianv2.bib}{}

\begin{thebibliography}{10}

\bibitem{Guionnet}
G.~W. Anderson, A.~Guionnet, and O.~Zeitouni.
\newblock {\em An introduction to random matrices}, volume 118 of {\em
  Cambridge Studies in Advanced Mathematics}.
\newblock Cambridge University Press, Cambridge, 2010.

\bibitem{BenArous}
G.~Ben Arous and A.~Guionnet.
\newblock {Large deviations for Wigner's law and Voiculescu's non-commutative
  entropy}.
\newblock {\em Probabilty theory and related fields}, 108:517--542, 1997.

\bibitem{Silverstein}
Z.~Bai and J.~W. Silverstein.
\newblock {\em Spectral analysis of large dimensional random matrices}.
\newblock Springer Series in Statistics. Springer, New York, second edition,
  2010.

\bibitem{Bai}
Z.~D. Bai and Y.~Q. Yin.
\newblock {Necessary and sufficient conditions for almost sure convergence of
  the largest eigenvalue of a Wigner matrix}.
\newblock {\em Ann. Probab.}, 16(4):1729--1741, 1988.

\bibitem{Dembo}
G.~Ben~Arous, A.~Dembo, and A.~Guionnet.
\newblock Aging of spherical spin glasses.
\newblock {\em Probab. Theory Related Fields}, 120(1):1--67, 2001.

\bibitem{BGM}
F.~Benaych-Georges, A.~Guionnet, and M.~Maida.
\newblock Large deviations of the extreme eigenvalues of random deformations of
  matrices.
\newblock {\em Probab. Theory Related Fields}, 154(3-4):703--751, 2012.

\bibitem{Benaych}
F.~Benaych-Georges and R.~Rao Nadakuditi.
\newblock {The eigenvalues and eigenvectors of finite, low rank pertubation of
  large random matrices}.
\newblock {\em Advances in Mathematics}, 227:494--521, 2011.

\bibitem{bhatia}
R.~Bhatia.
\newblock {\em Matrix analysis}, volume 169 of {\em Graduate Texts in
  Mathematics}.
\newblock Springer-Verlag, New York, 1997.

\bibitem{Bordenave}
C.~Bordenave and P.~Caputo.
\newblock A large deviation principle for {W}igner matrices without {G}aussian
  tails.
\newblock {\em Ann. Probab.}, 42(6):2454--2496, 2014.

\bibitem{Cabanal}
T.~Cabanal~Duvillard and A.~Guionnet.
\newblock Large deviations upper bounds for the laws of matrix-valued processes
  and non-communicative entropies.
\newblock {\em Ann. Probab.}, 29(3):1205--1261, 2001.

\bibitem{Clarke}
F.~Clarke.
\newblock {\em Functional analysis, calculus of variations and optimal
  control}, volume 264 of {\em Graduate Texts in Mathematics}.
\newblock Springer, London, 2013.

\bibitem{Zeitouni}
A.~Dembo and O.~Zeitouni.
\newblock {\em Large deviations techniques and applications}, volume~38 of {\em
  Stochastic Modelling and Applied Probability}.
\newblock Springer-Verlag, Berlin, 2010.
\newblock Corrected reprint of the second (1998) edition.

\bibitem{Feral}
D.~F{\'e}ral and S.~P{\'e}ch{\'e}.
\newblock The largest eigenvalue of rank one deformation of large {W}igner
  matrices.
\newblock {\em Comm. Math. Phys.}, 272(1):185--228, 2007.

\bibitem{Furedi}
Z.~F{\"u}redi and J.~Koml{\'o}s.
\newblock The eigenvalues of random symmetric matrices.
\newblock {\em Combinatorica}, 1(3):233--241, 1981.

\bibitem{Groux}
B.~Groux.
\newblock Asymptotic freeness for rectangular random matrices and large
  deviations for sample covariance matrices with sub-gaussian tails.
\newblock {\em arXiv}, 1505.05733 [math.PR].

\bibitem{Maida}
A.~Guionnet, M.~Maïda, and F.~Benaych-Georges.
\newblock {Fluctuations of the extreme eigenvalues of finite rank deformations
  of random matrices}.
\newblock {\em Electronic Journal of Probability}, 16(60):1621--1662, 2011.

\bibitem{Zeitounisphint}
A.~Guionnet and O.~Zeitouni.
\newblock Large deviations asymptotics for spherical integrals.
\newblock {\em J. Funct. Anal.}, 188(2):461--515, 2002.

\bibitem{Hardy}
A.~Hardy.
\newblock A note on large deviations for 2{D} {C}oulomb gas with weakly
  confining potential.
\newblock {\em Electron. Commun. Probab.}, 17:no. 19, 12, 2012.

\bibitem{Mylene}
M.~Ma{\"{\i}}da.
\newblock Large deviations for the largest eigenvalue of rank one deformations
  of {G}aussian ensembles.
\newblock {\em Electron. J. Probab.}, 12:1131--1150 (electronic), 2007.

\bibitem{Massart}
P.~Massart, G.~Lugosi, and S.~Boucheron.
\newblock {\em {Concentration Inequalities : A Nonasymptotic Theory of
  Independence}}.
\newblock Oxford University Press, 2013.

\bibitem{Peche}
S.~P{\'e}ch{\'e}.
\newblock The largest eigenvalue of small rank perturbations of {H}ermitian
  random matrices.
\newblock {\em Probab. Theory Related Fields}, 134(1):127--173, 2006.

\bibitem{resolvantentries}
A.~Pizzo, D.~Renfrew, and A.~Soshnikov.
\newblock {Fluctuations of matrix entries of regular functions of Wigner
  matrices}.
\newblock {\em Journal of Statistical Physics}, 146(3):550--591, 2012.

\bibitem{Soshnikov}
A.~Pizzo, D.~Renfrew, and A.~Soshnikov.
\newblock {On finte rank deformations of Wigner matrices}.
\newblock {\em Ann. Inst. H. Poincar{\'e} Probab. Statist.}, 49(120):64--94,
  2013.

\bibitem{Rockafeller}
R.~T. Rockafellar.
\newblock {\em Convex analysis}.
\newblock Princeton Landmarks in Mathematics. Princeton University Press,
  Princeton, NJ, 1997.

\bibitem{Wigner}
Eugene~P. Wigner.
\newblock On the distribution of the roots of certain symmetric matrices.
\newblock {\em Ann. of Math. (2)}, 67:325--327, 1958.

\bibitem{Zhan}
X.~Zhan.
\newblock {\em Matrix inequalities}, volume 1790 of {\em Lecture Notes in
  Mathematics}.
\newblock Springer-Verlag, Berlin, 2002.

\end{thebibliography}

\noindent Fanny Augeri\\
Institut de Math\'ematiques de Toulouse\\
118 route de Narbonne. 31062 Toulouse cedex 09. France.\\
E-mail : faugeri@math.univ-toulouse.fr 
\end{document}